\documentclass[reqno]{amsart}
\usepackage[utf8]{inputenc}

\usepackage{fullpage}
\usepackage{amsfonts, amsmath, amssymb, amsthm,mathtools,tikz-cd}
\usepackage{etoolbox}
\usepackage{bbm}
\usepackage{dynkin-diagrams}
\newcommand{\Z}{\mathbb{Z}}
\newcommand{\1}{\mathbbm{1}}
\newcommand{\kf}{\mathbb{K}}

\newcommand{\Q}{\mathbb{Q}}
\newcommand{\R}{\mathbb{R}}
\newcommand{\N}{\mathbb{N}}

\newcommand{\He}{\mathcal{H}}

\newcommand{\Hom}{\text{Hom}}
\usetikzlibrary{braids}
\usetikzlibrary{knots}
\theoremstyle{definition}
\newtheorem*{thm*}{Theorem}
\newtheorem{thm}{Theorem}[section]
\newtheorem{cor}[thm]{Corollary}
\newtheorem{lem}[thm]{Lemma}

\newtheorem{prop}[thm]{Proposition}
\newtheorem{rem}[thm]{Remark}

\newtheorem{ex}[thm]{Example}

\newtheorem{defn}[thm]{Definition}
\numberwithin{equation}{section}
\usepackage{graphicx}
\usepackage{caption}

\usepackage{hyperref}

\hypersetup{
    colorlinks=false,
    pdfborder={0 0 0},
    pdfborderstyle={/S/U/W 0},
}

\title{Categorification of quantum boson algebras}
\author{Sam Qunell}
\begin{document}
\begin{abstract}
    We produce graded monoidal categorifications of the quantum boson algebras in any symmetrizable Kac-Moody type. Our categories are defined in terms of diagrammatic generators and relations and have a faithful 2-representation on Khovanov-Lauda and Rouquier's categorification of the corresponding positive part quantum group. We use our construction to produce interesting bases of the quantum boson algebras and quantum bosonic extensions.
\end{abstract}
\maketitle
\tableofcontents
\section{Introduction}
In this paper, we construct graded monoidal categories $\mathcal{B}(C)$ categorifying the quantum boson algebra $B(C)$ for any symmetrizable generalized Cartan matrix $C$. Our category $\mathcal{B}(C)$ acts faithfully on Khovanov-Lauda and Rouquier's categorification of $U_q^+(\mathfrak{g}(C))$ via the induction and restriction functors for quiver Hecke algebra modules. We produce explicit bases of Hom spaces in $\mathcal{B}(C)$ based on certain planar diagrams. The category $\mathcal{B}(C)$ satisfies a universal property coming from its explicit generators and relations and from a localization.
\subsection{Quantum boson algebras}
The quantum boson algebras play an important yet subtle role in the theory of quantum groups. These algebras are defined for any symmetrizable Kac-Moody Lie algebra and have a faithful action on the corresponding positive part quantum group $U_q^+(\mathfrak{g})$. The quantum boson algebras were originally introduced by Kashiwara in \cite{kashi} in his study of crystal bases on $U_q^+(\mathfrak{g})$. Similar to how the action of the full quantum group gives rise to the Kashiwara operators on any simple finite-dimensional representation, the action of the quantum boson algebra on $U_q^+(\mathfrak{g})$ yields the Kashiwara operators for the crystal basis of $U_q^+(\mathfrak{g})$. For a simple finite-dimensional representation $V$ of $U_q(\mathfrak{g})$, the crystal basis of $U_q^+(\mathfrak{g})$ descends to one of $V$ via the quotient $U_q^+(\mathfrak{g})\twoheadrightarrow V$. So, the action of the quantum boson algebras controls the theory of crystal bases.  

The quantum boson algebra $B(C)$ for $(C_{ij})_{i,j\in I}$ is defined similarly to the quantum group for $C$. The algebra $B(C)$ has generators $\{E_i,F_i\}_{i\in I}$ so that both the $E_i$'s and $F_i$'s satisfy quantum Serre relations and so that the \emph{quantum boson relations} are satisfied:
\[F_iE_j-q_i^{-C_{ij}}E_jF_i=\frac{\delta_{ij}}{1-q_i^2}.\]
The quantum boson algebra $B(C)$ acts on $U_q^+(\mathfrak{g}(C))$ with $E_i$ acting as right multiplication by $E_i$ and with $F_i$ acting as the corresponding adjoint operator under Lusztig's bilinear form on $U_q^+(\mathfrak{g}(C))$. These adjoint operators are also known as Lusztig's twisted derivations \cite{lusbook}.

The quantum boson algebras and quantum boson relations have appeared in several related contexts since their introduction. In \cite{qshuffle}, Leclerc shows that $B(C)$ acts on the quantum shuffle algebra for $C$ in finite type and uses this structure to study the dual canonical bases of $U_q^+(\mathfrak{g}(C))$. For $\mathfrak{g}$ of type $ADE$, Hernandez and Leclerc prove that the deformed Grothendieck ring of a subcategory of finite-dimensional representations for the quantum Loop algebra $U_q(L\mathfrak{g})$ has a presentation in terms of generators and relations that generalize the quantum boson relations \cite{dehallalg}. Kashiwara, Kim, Oh, and Park prove that the algebra with these generators and relations admits a nondegenerate symmetric bilinear form and global basis theory for any symmetrizable Kac-Moody type \cite{newkashiboson}. They refer to this algebra as the \emph{bosonic extension} and show that elements of their basis correspond to $(q,t)$-characters of simple $U_q(L\mathfrak{g})$-representations in ADE type.

A modern approach to the structure theory of quantum groups is \emph{categorification}. For $I$ the set of vertices in the Dynkin diagram of $\mathfrak{g}(C)$ and for each $\alpha\in \N[I]$, Khovanov-Lauda and Rouquier produce graded algebras $H_\alpha(C)$ called \emph{quiver Hecke} or \emph{KLR algebras} \cite{khla}\cite{kholau2}\cite{2km}. Khovanov and Lauda prove in \cite{khla} and \cite{kholau2} that the category $\bigoplus_{\alpha} H_\alpha(C)-\text{grproj}$ of finitely generated graded projective modules over the KLR algebras is a monoidal categorification of Lusztig's integral form of $U_q^+(\mathfrak{g}(C))$.

The action of $B(C)$ on $U_q^+(\mathfrak{g}(C))$ has been found in this categorical context. The action of the $E_i$ is the induction functor $H_{\alpha}(C)-\text{grMod}\rightarrow H_{\alpha+\alpha_i}(C)-\text{grMod}$, and $F_i$ acts via the corresponding restriction. These functors restrict to acting on the categories of all graded projective modules $H_\alpha(C)-\text{grProj}$. These functors were used Lauda and Vazirani in \cite{categorcrystal} and independently by Kang and Kashiwara in \cite{kk} to study categorifications of highest weight simple representations of $U_q(\mathfrak{g})$. It is also proven in \cite{kk} that these functors satisfy categorical analogues of the quantum boson relations. Here, the scalar $1/(1-q_i^2)$ is interpreted as the infinite direct sum $\text{Id}\oplus q_i^2\text{ Id}\oplus q_i^4\text{ Id}\oplus\dots$. In \cite{itsmeeeeeeeee}, we use these categorical quantum boson relations to produce 2-representations of $U_q^+(\tilde{\mathfrak{sl}}_{n+1})$ on $U_q^+(\mathfrak{sl}_{n+1})$. It is therefore natural to hope that these functors come from the action of a monoidal category defined in terms of generators and relations. We note that Kang, Kashiwara, and Kim construct related monoidal categorifications in \cite{kashi-qt-categor}. However, their construction is an abelian category, and hence it does not have a natural representation on $\bigoplus_\alpha H_\alpha(C)-\text{grProj}$. Moreover, their construction is less tractable outside of ADE type. To contrast, our results for $\mathcal{B}(C)$ work uniformly well in any symmetrizable Kac-Moody type. See also \cite{jantz} for another categorical interpretation of the deformed Grothendieck ring of $U_q(L\mathfrak{g})$. We expect that our results in \cite{itsmeeeeeeeee} can be reinterpreted as the existence of a functor from the categorification $\mathcal{U}_q^+(C)$ of $U_q^+(C)$ into the homotopy category of $\mathcal{B}(C)$, although we do not prove this.

\subsection{Main results}
We construct our category $\mathcal{B}(C)$ in several steps. Extra caution is needed in the construction since it is not easy to incorporate infinite biproducts into the usual diagrammatic theory. Firstly, for a matrix of KLR parameters $Q$, we define a monoidal category $\mathcal{A}_2(Q)$ generated by objects $E_i,F_i$ and morphisms $X_i:E_i\rightarrow E_i$, $T_{ij}:E_iE_j\rightarrow E_jE_i$, $\eta_i:\1\rightarrow F_iE_i$, and $\epsilon_i:E_iF_i\rightarrow \1$. The $X_i$ and $T_{ij}$ satisfy the same KLR relations as in the categorification of the full quantum group \cite{qha}\cite{kholau3}. The $\eta_i$ and $\epsilon_i$ are the unit and counit of a one-sided duality $E_i \dashv F_i$. We only require one-sided duality since the induction and restriction functors on KLR-algebra modules are only one-sided adjoints. This should be contrasted with the two-sided duality up to shifts in the full quantum group categorification \cite{kholau3}. We associate planar diagrams to each morphism obtained by composing and tensoring the defining morphisms. We use the homotopy techniques of \cite{kholau3} to produce explicit spanning sets $B_{X,Y}$ of each Hom space $\Hom_{\mathcal{A}_2(Q)}(X,Y)$ in terms of certain minimal diagrams with no double crossings or self-intersections. In case $Q$ is determined by a symmetrizable generalized Cartan matrix $C$, we may consider a graded version $\mathcal{A}_2(C)$ of $\mathcal{A}_2(Q)$ where we introduce formal shifts of $q^nX$ of each object and consider only morphisms of appropriate degree. The same homotopy arguments show that all graded Hom spaces in $\mathcal{A}_2(C)$ are bounded below in degree and finite-dimensional in each degree.

The category $\mathcal{A}_2(C)$ acts on $\bigoplus_\alpha H_\alpha(C)-\text{grMod}$ via the induction and restriction functors. Our first main result is Theorem \ref{thm:a2_2rep_faithful}.
\begin{thm*}
    The 2-representation of $\mathcal{A}_2(C)$ on $\bigoplus_\alpha H_\alpha(C)-\text{grMod}$ is faithful. The graded Hom space $\sum_{i\in \Z} \Hom_{\mathcal{A}_2(C)}(q^iX,Y)$ has a basis given by $B_{X,Y}$.
\end{thm*}
Showing nondegeneracy of a category defined in terms of generators and relations is in general extremely difficult. Khovanov and Lauda originally could only prove nondegeneracy of their categorification of the idempotented quantum group in type A \cite{kholau3}. Later work in \cite{kk},\cite{web1}, and \cite{web2} extended this result to other types, although this required categorifying many finite-dimensional representations and several lengthy computations. In our case, we can prove nondegeneracy for all symmetrizable types in a uniform way since the natural faithful representation of $B(C)$ is already categorified.

The quantum Serre relations for the $E_i$ are categorified by certain contractible chain complexes in $\mathcal{A}_2(C)$ \cite{kholau2}. By adjunction, a similar chain complex categorifies the quantum Serre relations for the $F_i$. It remains to produce the categorification of the quantum boson relations in these diagrammatic terms. One technical difficulty is in adjoining certain infinite direct sums so that we may express the scalar $1/(1-q_i^2)$. In a general additive category, it may not be true that an infinite coproduct also has the structure of a product. However, due to our explicit bases for each graded Hom space, we prove in Lemma \ref{lem:a2_biprod} that all of the coproducts that we need are biproducts.

We observe that Kang and Kashiwara's isomorphism of functors $\rho_{ij}:q_i^{-C_{ij}}E_iF_j\oplus \delta_{ij}\bigoplus_{n\in \N}q_i^n \1\xrightarrow{\sim}F_jE_i$ is given by a certain direct sum of $\mathcal{A}_2(C)$-morphisms. The inverse isomorphism is not easily expressed in terms of our generating morphisms, so to categorify the quantum boson relations, we must localize at a set of morphisms containing the $\rho_{ij}$. We prove in Lemma \ref{lem:rightfrac} that the appropriate set of morphisms admits a calculus of right fractions. This allows us to  give explicit bases for each Hom space in the localization. The calculus of right fractions has a natural topological interpretation in terms of restricting what types of diagrams can arise from morphisms in $\mathcal{A}_2(C)$.

After localizing, we obtain our desired category $\mathcal{B}(C)$ which has an explicit universal property. Denote by $\mathcal{B}(C)^i$ the idempotent completion. Our second main result is Theorem \ref{thm:gammaiso}, where we identify the Grothendieck group of $\mathcal{B}(C)^i$.
\begin{thm*}
    Let $\textbf{B}$ denote the algebra $\Z[q,q^{-1}][(1/(1-q_i^2))_{i\in I}]$, and denote by $_\textbf{B}B(C)$ the $\textbf{B}$-subalgebra of $B(C)$ generated by the divided powers $E_i^n/[n]_i!$ and $F_i^n/[n]_i!$. Then as $\textbf{B}$-algebras, we have an isomorphism $_\textbf{B}B(C)\xrightarrow{\sim} K_0(\mathcal{B}(C)^i)$.
\end{thm*}
Part of this theorem involves identifying a $\Q$-bilinear Hom-like form on $\Q(q)\otimes_{\textbf{B}}K_0(\mathcal{B}(C)^i)$ with a related bilinear form on $B(C)$ that was introduced in \cite{newkashiboson}. Since our category is defined with morphisms given by linear combinations of diagrams, it is natural to ask whether this bilinear form has a purely graphical description similar to that of the bilinear form on the full quantum group given in \cite{kholau3}. This does not follow directly from our categorification since certain necessary diagrams are not realizable by morphisms in $\mathcal{A}_2(C)$. However, we prove in Theorem \ref{thm:qthom_diagram} that the bilinear form on $B(C)$ can be interpreted as computing a graded dimension of a vector space spanned by a larger class of diagrams between monomials in the $E_i$ and $F_i$. In fact, we prove a more general result for the bosonic extension. This construction allows us to view the quantum boson algebra as a quotient of the free algebra on the same generators by the kernel of the bilinear form.

Theorem \ref{thm:gammaiso} also gives us a basis of indecomposable object classes for $B(C)$ (up to shifts). It is automatically true that this basis has structure constants for multiplication within $\N[q,q^{-1},(1/(1-q_i^2))_{i\in I}]$ and has a similarly convenient action on the indecomposable basis of $U_q^+(C)$. It follows from the proof of Theorem \ref{thm:gammaiso} that this indecomposable basis of $B(C)$ can be described explicitly. There is a vector space isomorphism $U_q^+(C) \otimes_{\Q(q)} U_q^+(C)\simeq B(C)$ with $E_i\otimes 1\rightarrow E_i$ and $1\otimes E_i\rightarrow F_i$.  If $S$ is the indecomposable basis for $U_q^+(C)$, then the image of $S\otimes S$ under this isomorphism is the indecomposable basis of $B(C)$ up to shifts. It is natural to then guess an interesting basis for the bosonic extension, which as a vector space is isomorphic to the restricted infinite tensor $\bigotimes'_{n\in \Z} U_q^+(C)$. We extend the above results to prove in Theorem \ref{thm:bznicebasis} that the restricted infinite tensor $\bigotimes'_{n\in \Z} S$ has several interesting properties resembling that of an indecomposable basis. These last few results suggest the existence of an additive graded monoidal categorification of the bosonic extension. 
\subsection{Organization}
This paper is organized as follows. In Section \ref{sec:back}, we fix notation for the various algebras and categories that we will use and review important structures on them. In Section \ref{sec:qboson}, we construct the categorification of the quantum boson algebra. In Subsections \ref{subsec:intermediate}, \ref{subsec:topology}, and \ref{subsec:2rep}, for any symmetrizable Cartan matrix $C$ we we construct a category $\mathcal{A}_2(C)$ by generators and relations without insisting on any direct sums or quantum boson relations. We construct a 2-representation of $\mathcal{A}_2(C)$ on $\bigoplus_\alpha H_\alpha(C)-\text{grMod}$. In Subsection \ref{subsec:faithful}, we prove that this is a faithful 2-representation and thereby produce an explicit basis for each Hom space in $\mathcal{A}_2(C)$. In Subsection \ref{subsec:localize}, we localize the appropriate extension of $\mathcal{A}_2(C)$ at a certain set of morphisms in order to categorify the quantum boson relations. In Subsection \ref{subsec:K0}, we show that the Grothendieck group of our quantum boson category is isomorphic to the appropriate integral form of $B(C)$. In Subsection \ref{subsec:bases}, we characterize the indecomposable basis of $B(C)$ and its action on $U_q^+(C)$. In Section \ref{sec:graphical}, we produce a graphical description of a certain bilinear form on the bosonic extension for $C$. We also produce a basis of the bosonic extension based on our results in Subsection \ref{subsec:bases}.
\subsubsection*{Acknowledgments}
This research was conducted as part of the author's Ph.D. thesis in mathematics at UCLA. This work was supported by NSF grant DMS-2302147. The project is also supported by Simons Foundation award number 376202. We thank Rapha\"el Rouquier for his guidance and feedback. 
\section{Background}\label{sec:back}
\subsection{General Notation}
We denote by $\N\coloneqq \Z_{\geq 0}$ the set of nonnegative integers. By the phrase ``$\kf$-linear category" for $\kf$ a field, we mean a category enriched in $\kf$-vector spaces. We do not assume existence of a zero object or finite direct sums.

We denote by $S_n$ the symmetric group on $n$ symbols, and for $\omega\in S_n$, we denote by $l(\omega)$ the length of any reduced presentation of this element.

When $\mathcal{C}$ is a small category, we denote by $\text{Ob}(\mathcal{C})$ the set of objects of $\mathcal{C}$. We may write $M\in \mathcal{C}$ to mean that $M\in\text{Ob}(\mathcal{C})$. In a category, for morphisms $f:X\rightarrow Y$ and $g:Y\rightarrow Z$, we denote the composition as $g\circ f:X\rightarrow Z$. In a monoidal category with monoidal product $\otimes$, for objects $A$ and $B$, we abbreviate the monoidal product $A\otimes B$ as $AB$. We often denote the identity morphism of an object $A$ as $A$ when no ambiguity is possible. For $f$ a morphism $A\rightarrow X$ and $g$ a morphism $B\rightarrow Y$, we abbreviate $fg:=f\otimes g:AB\rightarrow XY$. We often refer to the monoidal product in a monoidal category as the \emph{tensor}.

If $V$ is a $\Z$-graded $\kf$-vector space, we denote by $V_i$ the graded component in degree $i$. If moreover $V$ has finite-dimensional graded components, then denote by $\text{grdim}(V)$ as the formal infinite sum $\sum_{i\in \Z} \text{dim}_\kf (V_i)q^i$. If $V_i=0$ for all $i<<0$, then the graded dimension takes values in $\N[[q]][q^{-1}]$. If $T$ is the shift functor on graded vector spaces satisfying $T(V)_i=V_{i-1}$, then $\text{grdim}(T(V))=q*\text{grdim}(V)$. 
\subsection{Quantum algebras}\label{subsec:qas}
We fix notation for the various quantum algebras that we use.

\begin{defn}
    A \emph{symmetrizable generalized Cartan matrix} $(C_{ij})_{i,j\in I}$ is a finite integer-valued square matrix satisfying the following properties.
    \begin{enumerate}
        \item For each $i\in I$, we have $C_{ii}=2$.
        \item For each $i,j\in I$ with $i\neq j$, we have $C_{ij} \leq 0$.
        \item $C_{ij}=0$ iff $C_{ji}=0$.
        \item There exists a diagonal matrix $(D_{ij})_{i,j\in I}=\text{diag}(D_i)$ whose diagonal entries are positive relatively prime integers such that $DC=S$, where $S$ is a symmetric matrix.
    \end{enumerate}
\end{defn}

\begin{defn}
    Let $(C_{ij})_{i,j\in I}$ be a symmetrizable generalized Cartan matrix with associated diagonal matrix $D$. Let $q$ be an indeterminant, and denote by $q_i\coloneqq q^{D_{i}}$. For each $n\in \N$ and $i\in I$, denote by $[n]_i\coloneqq (q_i^n-q_i^{-n})/(q_i-q_i^{-1})$ and $[n]_i!\coloneqq [n]_i * [n-1]_i * \dots [1]_i$. For each $n\in \N$, each $k$ with $0\leq k \leq n$, and each $i\in I$, denote by $\binom{n}{k}_i \coloneqq [n]_i!/([k]_i! \dot [n-k]_i!)$.
\end{defn}

\begin{defn}
    Let $(C_{ij})_{i,j\in I}$ be a symmetrizable generalized Cartan matrix with indexing set $I$ and let $\mathfrak{g}=\mathfrak{g}(C)$ be the associated complex Kac-Moody Lie algebra. We define the \emph{positive part of the quantum enveloping algebra} of $\mathfrak{g}$ to be the associative $\Q(q)$-algebra $U_q^+(\mathfrak{g})$ generated by elements $\{E_i\}_{i\in I}$ satisfying the \emph{quantum Serre relations}.
    \[\sum_{k=0}^{1-C_{ij}}\binom{1-C_{ij}}{k}_i (-1)^kE_i^kE_jE_i^{1-C_{ij}-k} =0 \text{ for $i\neq j$}.\]

    For $a\in \N$ and $i\in I$, we also define the \emph{divided powers} $E_i^{(a)}\coloneqq E_i^a/[a]_i!$.
\end{defn}

Denote by $\N[I]$ the semigroup of all nonnegative formal linear combinations of $I$ elements. For $i\in I$, we denote by 
$\alpha_i$ its corresponding generator in $\N[I]$. We similarly define $\Z[I]$. The algebra $U_q^+(\mathfrak{g})$ has a natural $\N[I]$-grading, with $gr(E_i)=\alpha_i$. For $\alpha\in \N[I]$, denote by $U_q^+(\mathfrak{g})_\alpha$ the corresponding homogeneous subspace. For $\N[I]$-elements $\alpha=\sum_i a_i\alpha_i$ and $\alpha'=\sum_i b_i\alpha_i$ where each $a_i,b_i\in \N$, we say $\alpha \leq \alpha'$ if each $a_i\leq b_i$, and say $\alpha < \alpha'$ if $\alpha\leq \alpha'$ and $\alpha \neq \alpha'$. 

\begin{defn}
    On $U_q^+(\mathfrak{g})$, there is a unique symmetric bilinear form $(*,*)_L$ defined by $(1,1)_L=1$, $(1,E_i)_L=0$, $(E_i,E_j)_L=\frac{\delta_{ij}}{1-q_i^2}$, and extended to the rest of $U_q^+(\mathfrak{g})$ via the Hopf pairing rule $(ab,c)_L=(a\otimes b,\Delta(c))_L$, for $\Delta$ the comultiplication of the standard twisted bialgebra structure. This form is nondegenerate.
\end{defn}
\begin{defn}
    There is a $q$-antilinear automorphism of $U_q^+(\mathfrak{g})$ denoted $\bar{*}$ given by $\bar{q}=q^{-1}$ and $\bar{E_i}=E_i$. The $q$-semilinear form $(\bar{*},*)_L$ is nondegenerate since $(*,*)_L$ is.
\end{defn}
\begin{defn}
    The \emph{quantum boson algebra} associated to a symmetrizable generalized Cartan matrix $(C_{ij})_{i,j\in I}$ is the $\Q(q)$-algebra $B(C)$ with generators $\{E_i,F_i\}_{i\in I}$ satisfying the following relations. Both the $\{E_i\}_{i\in I}$ and the $\{F_i\}_{i\in I}$ satisfy quantum Serre relations. Moreover, the \emph{quantum boson relations} hold:
    \[F_iE_j-q_i^{-C_{ij}}E_jF_i=\frac{\delta_{ij}}{1-q_i^2}.\] We similarly define divided powers $E_i^{(a)}\coloneqq E_i^a/[n]_i!$ and $F_i^{(a)}\coloneqq F_i^a/[n]_i!$.
    We also denote this algebra as $B(\mathfrak{g})$ for $\mathfrak{g}=\mathfrak{g}(C)$.
\end{defn}
Kashiwara showed in \cite{kashi} that $B(\mathfrak{g})$ acts on $U_q^+(\mathfrak{g})$ as follows. The element $E_i$ acts by right multiplication by $E_i$, and $F_i$ acts as the corresponding adjoint operator $E_i^*$ under $(*,*)_L$, albeit with different choices of normalization. It is known that this action is faithful, see for example \cite{newkashiboson}.

We consider a few different lattice subrings.

\begin{defn}
    Denote by $\mathbf{A}$ the ring $\Z[q,q^{-1}]$ and by $\mathbf{B}$ the ring $\Z[q,q^{-1}][(1/(1-q_i^2))_{i\in I}]$. Note that both $\mathbf{A}$ and $\mathbf{B}$ are subrings of $\Q(q)$. For $\mathfrak{g}$ a symmetrizable Kac-Moody Lie algebra, denote by $_\mathbf{A}U_q^+(\mathfrak{g})$ the $\mathbf{A}$-subalgebra of $U_q^+(\mathfrak{g})$ generated by the divided powers $E_i^{(a)}$. Likewise, denote by $_\mathbf{B}B(\mathfrak{g})$ the $\mathbf{B}$-subalgebra of $B(\mathfrak{g})$ generated by the $E_i^{(a)}$ and the $F_i^{(b)}$.
\end{defn}

The quantum boson algebra is the main focus of this article, although some of its properties are best studied via a larger algebra called the bosonic extension. For $C$ of type $ADE$, Hernandez and Leclerc proved that a renormalization of this algebra is isomorphic to a deformed Grothendieck ring for a certain subcategory of finite dimensional representations of the quantum loop algebra $U_q(L\mathfrak{g})$ \cite{dehallalg}.
\begin{defn}
    The \emph{bosonic extension} associated to a symmetrizable generalized Cartan matrix $(C_{ij})_{i,j\in I}$ is the $\Q(q)$-algebra $B_\Z(C)$ with generators $E_{i,n}$ for $i\in I$ and $n\in \Z$ satisfying the following relations.
    \[\sum_{k=0}^{1-C_{ij}}(-1)^k\binom{1-C_{ij}}{k}_iE_{i,n}E_{j,n}E_{i,n}^{1-C_{ij}-k}=0 \text{ for $i\neq j,$ any $n\in \Z$},\]
    \[E_{i,n+k}E_{j,n}=q_i^{(-1)^kC_{ij}}E_{j,n}E_{i,n+k}+\delta_{(i,j),(k,1)}\frac{1}{1-q_i^2} \text{ for $i,j\in I$, $n\in \Z$, $k\in \Z^+$.}\]
\end{defn}
The second set of relations above will also be referred to as the \emph{quantum boson relations}. This algebra is studied in depth in \cite{newkashiboson}. We need a bilinear form and some symmetries of $B_\Z(C)$. First, some setup.

There is a $\Z[I]$-grading on $B_\Z(C)$ defined by $\text{gr}(E_{i,n})=(-1)^n\alpha_i$. For $\alpha\in \Z[I]$, we denote by $B_\Z(C)_\alpha$ the corresponding weight space. The algebra $B_\Z(C)$ comes with two important symmetries. There is a $q$-antilinear antiautomorphism $\bar{*}$ that fixes generators, i.e. $\bar{q}=q^{-1}$ and $\bar{E}_{i,n}=E_{i,n}$. Note that this is an antiinvolution unlike the bar involution on $U_q^+(C)$. There is another $q$-antilinear antiautomorphism $\mathcal{D}$ defined by $\mathcal{D}(q)=q^{-1}$ and $\mathcal{D}(E_{i,n})=E_{i,n+1}$. The composition $\bar{\mathcal{D}}\coloneqq \bar{*}\circ \mathcal{D}=\mathcal{D}\circ \bar{*}$ is the $q$-linear automorphism satisfying $\bar{\mathcal{D}}(E_{i,n})=E_{i,n+1}$.

Denote by $V\coloneqq \bigotimes'_{n\in \Z} U_q^+(\mathfrak{g}(C))$ the restricted infinite tensor product over $\Z$-many copies of $U_q^+(\mathfrak{g}(C))$, i.e. the vector space generated by the infinite simple tensors $\dots v_n \otimes v_{n+1}\otimes v_{n+2}\dots$ with each $v_{n}\in U_q^+(\mathfrak{g}(C))$ and all but finitely many of the $v_n=1$. It is proven in Corollary 4.4 of \cite{newkashiboson} that the linear map $G:V \rightarrow B_\Z (C)$ given by $G(\otimes_{n\in \Z} v_n)=\prod_{n\in \Z} \bar{\mathcal{D}}^n(v_n)=\dots \bar{\mathcal{D}}^n(v_n)\bar{\mathcal{D}}^{n+1}(v_{n+1})\dots$ is an isomorphism of vector spaces. 

Denote by $B_{\Z,n}(C)$ the subalgebra generated by the $E_{i,n}$, and for $\alpha\in \Z[I]$, denote by $B_{\Z,n}(C)_\alpha\coloneqq B_{\Z,n}(C)\cap B_\Z(C)_\alpha$. We similarly define the $B_{\Z,\leq n}(C),B_{\Z,\geq n}(C),B_{\Z,<n}(C), B_{\Z,>n}(C)$ and their weight spaces. The map $G_n:U_q^+(\mathfrak{g}(C))\rightarrow B_{\Z,n}(C)$ given by $E_i\rightarrow E_{i,n}$ is an isomorphism of algebras. The linear isomorphism with $V$ implies that $B_\Z(C)$ decomposes as 
\begin{equation}\label{eqn:bzdecomp}
B_\Z(C)\simeq \bigoplus_{(\alpha_n)_{n\in \Z}}{\bigotimes}'_{n\in \Z} B_{\Z,n}(C)_{\alpha_n},
\end{equation} where we only consider $(\alpha_n)_{n\in \Z}$ with only finitely many of the $\alpha_n$ nonzero. Denote by $P:B_\Z(C)\rightarrow B_\Z(C)$ the linear projection onto the subspace $\bigotimes'_{n\in \Z} B_{\Z,n}(C)_0\simeq \Q(q)$ via the decomposition in Equation \ref{eqn:bzdecomp}.

The following properties of $(*,*)_\Z$ are the appropriate versions of Lemma 5.3 and Theorem 5.4 of \cite{newkashiboson}. Some small changes were required due to our different normalization and since our form is $q$-semilinear and not symmetric. These can be proven with the techniques of \cite{newkashiboson}.

\begin{prop}\label{prop:qthomform}
There is a $\Q$-bilinear form $(*,*)_\Z$ on $B_\Z(C)$ defined by $(X,Y)_\Z \coloneqq P(\mathcal{D}(X)Y)$. It satisfies the following properties for each $X,Y\in B_\Z(C)$. 
\begin{enumerate}
    \item $(*,*)_\Z$ is $q$-semilinear, i.e. $(X,qY)_\Z=(q^{-1}X,Y)_\Z=q(X,Y)_\Z$.
    \item The $\Q(q)$-bilinear form $(\bar{*},*)_\Z$ is symmetric, i.e. $(\bar{X},Y)_\Z=(\bar{Y},X)_\Z$.
    \item If $X\in B_\Z(C)_{\alpha_X}$ and $Y\in B_\Z(C)_{\alpha_Y}$ with $\alpha_X\neq \alpha_Y$ then $(X,Y)_\Z=0$.
    \item $(*,*)_\Z$ is nondegenerate.
    \item $(X,Y)_\Z=(\bar{\mathcal{D}}(X),\bar{\mathcal{D}}(Y))_\Z$
    \item $(E_{i,m}X,Y)_\Z=(X,E_{i,m+1}Y)_\Z$ and $(XE_{i,m},Y)_\Z=(X,YE_{i,m-1})_\Z$.
    \item Fix $m\in\Z,$ homogeneous $X,W\in B_{\Z,>m}(C)$ and $Y,Z\in B_{\Z,\leq m}(C)$ with $X\in B_\Z(C)_{\alpha_X}$ and $Z\in B_\Z(C)_{\alpha_Z}$. Then $(XY,ZW)_\Z=q^{-(\alpha_X,\alpha_Z)}(X,W)_\Z(Y,Z)_\Z$. 
\end{enumerate}
\end{prop}

The following is similar to Proposition 5.6 (iii) of \cite{newkashiboson}. It is a consequence of the graphical description of $(*,*)_\Z$ given in Theorem \ref{thm:qthom_diagram}, but we also prove it directly for convenience.
\begin{prop}
    \item For each $n\in \Z$ and any $A,B\in U_q^+(\mathfrak{g}(C))$, we have $(\bar{A},B)_L=(G_n(A),G_n(B))_\Z$.
\end{prop}
\begin{proof}
     Fix $n\in \Z$. Let $P_n$ be the projection $B_\Z(C)\twoheadrightarrow B_{\Z,n}(C)$ via the decomposition in Equation \ref{eqn:bzdecomp}. Let $F_i'$ be the $Q(q)$-linear operator on $U_q^+(C)\simeq B_{\Z,n}(C)$ defined by $F_i'(X)=G_n^{-1}(P_n(E_{i,n+1}G_n(X)))$. Then $F_i'(1)=0$ and $F_i'(E_{j,n}X)=q_i^{-C_{ij}}E_{j,n}F_i'(X)+\delta_{ij}/(1-q_i^2)$. So, we can compute inductively that $F_i'$ is the adjoint to $E_j\times$ under $(\bar{*},*)$ since this adjoint also satisfies the quantum boson relations. It is also easy to see that for any $Y,Y'\in B_{\Z,n+1}(C)$ and $X\in B_{\Z,n}(C)$ that $P_n(YP_n(Y'X))=P_n(YY'X)$. 
     
     We now show that the two bilinear forms are equal. Both forms are $\N[I]$-homogeneous and $q$-semilinear, so it is enough to verify equality on monomials of the same $\N[I]$-weight. We compute that for any monomial $E_{i_1}\dots E_{i_k}\in U_q^+(C)$ and $X\in U_q^+(C)$ that $(\overline{E_{i_1}\dots E_{i_k}},X)_L=(1,F'_{i_k}\dots F'_{i_1}(X))_L= F'_{i_k}\dots F'_{i_1}(X)=G_n^{-1}(P_n(E_{i_k,n+1}\dots E_{i_1,n+1}G_n(X)))$. Since this must be a scalar, we have that \[G_n^{-1}(P_n(E_{i_k,n+1}\dots E_{i_1,n+1}G_n(X)))=G_n^{-1}(P(E_{i_k,n+1}\dots E_{i_1,n+1}G_n(X)))=(G_n(E_{i_1}\dots E_{i_k}),G_n(X))_\Z,\] as desired.
\end{proof}
There is an embedding of algebras $B(C)\hookrightarrow B_\Z(C)$ given by $E_i\rightarrow E_{i,0}$ and $F_i\rightarrow E_{i,1}$. The  form $(*,*)_\Z$ on $B_\Z(C)$ induces a $q$-semilinear form on $B(C)$, which we denote by $(*,*)_2$. It is nondegenerate due to conditions 5 and 8 of Proposition \ref{prop:qthomform}.

\subsection{KLR algebras}\label{subsec:klr}
We review Khovanov-Lauda-Rouquier (KLR) algebras. These algebras were originally defined by Khovanov, Lauda, and Rouquier in \cite{2km} and \cite{khla}, and all results of this subsection are from there. For a more introductory survey, see \cite{brundan}.

\begin{defn}\label{defn:seq}
    For a set $I$, denote by Seq the set of formal finite sequences in the set $I$. For $\textbf{i},\textbf{j}\in $ Seq, we may write by either $\textbf{i}\textbf{j}$ or $\textbf{i}\cdot \textbf{j}$ the concatenation of the two sequences. We will often identify a length one sequence with its unique entry. For $\textbf{i}\in $ Seq, denote by $\textbf{i}_n$ the $n$'th entry of $\textbf{i}$. For $1\leq n \leq |\textbf{i}|$, we define the truncations $\textbf{i}_{\geq n}$ and $\textbf{i}_{\leq n}$ in the natural way.
     For $\textbf{i}\in$ Seq, we denote by $|\textbf{i}|$ the length of the sequence and by $\bar{\textbf{i}}$ the reversed sequence, i.e. the sequence obtained by reading $\textbf{i}$ from right to left. 
\end{defn}

\begin{defn}
    Define $\N[I]$ as before. There is a homomorphism of monoids Ab: Seq $\rightarrow \N[I]$ with $i\rightarrow \alpha_i$ and concatenation of sequences sent to addition of elements. For $\textbf{i}\in $ Seq, we denote by $\alpha_{\textbf{i}}$ the image of this element in $\N[I]$. For $\alpha\in \N[I]$, denote by $\text{Seq}_\alpha$ the preimage $\text{Ab}^{-1}(\alpha)$. For $\alpha\in \N[I]$, denote by $|\alpha|\coloneqq |\sum_{i\in I}c_i * \alpha_i|=\sum_{i\in I}c_i$, where $c_i\in \N$. The symmetric group $S_{|\alpha|}$ has a transitive left action on $\text{Seq}_\alpha$ with the transposition $s_k$ acting on $\textbf{i}$ by swapping $\textbf{i}_k$ with $\textbf{i}_{k+1}$.
\end{defn}

\begin{defn}\label{klr}
    
Let $I$ be a finite indexing set and let $\mathbb{K}$ be a field. Let $(Q_{ij})_{i,j\in I}$ be a $\kf [u,v]$-valued matrix such that $Q_{ii}=0$ for all $i\in I$. 

We define the \emph{KLR algebra} $H_{\alpha}(Q)$ to be the associative $\mathbb{K}$-algebra generated by elements $\{1_\textbf{i}\}_ {\textbf{i}\in \text{Seq}_\alpha}\cup \{x_i\}_{1\leq i\leq |\alpha|}\cup \{\tau_i\}_{1\leq i\leq |\alpha|-1}$ with relations
\begin{enumerate}
    \item $\text{The elements } 1_\textbf{i} \text{ are orthogonal idempotents whose sum is } 1$,

    \item $1_\textbf{i}x_i=x_i1_\textbf{i},$
    
    \item $1_\textbf{i}\tau_i=\tau_i1_{s_i(\textbf{i})}$

    \item $x_ix_j=x_jx_i$, 

    \item $(\tau_ix_j-x_{s_{i}(j)}\tau_i)1_\textbf{i}= \begin{cases}
    1_\textbf{i} \text{ if } j=i+1 \text{ and } \textbf{i}_i=\textbf{i}_{i+1}\\
    -1_\textbf{i} \text{ if } j=i \text{ and } \textbf{i}_i=\textbf{i}_{i+1}\\
    0 \text{ otherwise}
    \end{cases}$,

    \item $\tau_i\tau_j=\tau_j\tau_i$ if $|i-j|>1$,
    \item $\tau_i^21_\textbf{i}=Q_{v_i,v_{i+1}}(x_i,x_{i+1})1_\textbf{i},$

    \item $(\tau_{i+1}\tau_i\tau_{i+1}-\tau_i\tau_{i+1}\tau_i)1_\textbf{i}=\delta_{\textbf{i}_i,\textbf{i}_{i+2}}\frac{Q_{\textbf{i}_{i},\textbf{i}_{i+1}}(x_{i+2},x_{i+1})-Q_{\textbf{i}_i,\textbf{i}_{i+1}}(x_{i},x_{i+1})}{x_{i+2}-x_i}1_\textbf{i}$.

\end{enumerate}

It can be shown that $Q_{\textbf{i}_i,\textbf{i}_{i+1}}(x_i,x_{i+1})-Q_{\textbf{i}_i,\textbf{i}_{i+1}}(x_{i+2},x_{i+1})$ is always divisible by $x_{i+2}-x_i$, and so the expression on the right-hand side of relation (8) is always a well-defined element of $H_\alpha(Q)$. The relations above will henceforth be called the \emph{KLR relations}. 

If we have $\alpha\in \N[I]$ and $\textbf{j}\in \text{Seq}_\beta$ for some $\beta \in \N[I]$ with $|\beta|<|\alpha|$, then we denote 
\[1_{*\textbf{j}}:=\sum_{\substack{\textbf{i}\in  \text{Seq}_{\alpha}\\ 
    \textbf{i}=\textbf{j}'\textbf{j},\\
    \textbf{j}'\in \text{Seq}
    _{\alpha-\beta}}} 1_\textbf{i}\in H_\alpha(Q)
                    .\]

For any $\alpha\in \N[I]$, $1\leq i\leq |\alpha|$, and $1\leq j\leq |\alpha|-1$, we write that $x_{-i}:=x_{|\alpha|-i+1}$ and $\tau_{-j}:=\tau_{|\alpha|-j}$ as elements of $H_\alpha$. Note that in this notation we have $(\tau_{-1}x_{-1}-x_{-2}\tau_{-1})1_{*11}=1_{*11}$, and that under a module embedding $H_\alpha\hookrightarrow 1_{*i}H_{\alpha+\alpha_i}$, we have $x_{-i}\rightarrow x_{-i-1}1_{*i}$ and $\tau_{-i}\rightarrow \tau_{-i-1}1_{*i}$.

These algebras are also called \emph{quiver Hecke algebras}. When $Q$ is clear, we simply write $H_\alpha$ in place of $H_\alpha(Q)$.
\end{defn}

We associate KLR algebras to a symmetrizable generalized Cartan matrix $(C_{ij})_{i,j\in I}$ as follows. The matrix $Q$ is also indexed by $I$. We fix a choice of $\kf$-scalars $\{t_{ij}\}$ and $\{s_{ij;pq}\}$, where the first set is indexed over pairs $i\neq j\in I$ and where the second is also indexed over $p,q \in \N$ with $p<-C_{ij}$ and $q<-C_{ji}$. These scalars are required to satisfy the following properties.

\begin{enumerate}
    \item All $t_{ij}$ are invertible.
    \item $s_{ij;pq}=s_{ji;qp}$.
    \item $t_{ij}=t_{ji}$ when $C_{ij}=0$.
    \item $s_{ij;pq}=0$ unless $d_ip+d_jq=-d_iC_{ij}$.
\end{enumerate}
With such a choice of scalars, we define $Q_{ij}$ by $Q_{ii}=0$, $Q_{ij}=t_{ij}$ when $C_{ij}=0$, and
\[Q_{ij}=t_{ij}u^{-C_{ij}}+\sum_{p,q\vert d_ip+d_jq=-d_iC_{ij}} s_{ij;pq} u^pv^q+t_{ji}v^{-C_{ji}}.\]

For $\alpha\in \N[I]$, we then define $H_\alpha(C)\coloneqq H_\alpha(Q)$ associated to this choice of scalars. It is easy to check that this $Q$ satisfies $Q_{ij}(u,v)=Q_{ji}(v,u)$. Note that $Q_{ij}(u,v)$ is homogeneous for the grading $\text{deg}(u)=2d_i$ and $\text{deg}(v)=2d_j$. In fact, we may define a grading on $H_\alpha(Q)$ by $\text{deg}(1_\textbf{i})=0$, $\text{deg}(x_i1_\textbf{i})=2d_{\textbf{i}_i}$, and $\text{deg}(\tau_i1_\textbf{i})=-d_{\textbf{i}_i}C_{\textbf{i}_i\textbf{i}_{i+1}}$.

\subsection{Categories}
We describe the many types of categories that we use. We also define some of the specific categories studied in this paper and review a key theorem from the literature.

All categories used in this paper are $\kf$-linear. 

\begin{defn}
    A \emph{strictly graded category}, or \emph{strictly} $\Z$-\emph{graded category}, is a category equipped with an autoequivalence $T$ such that $T^k(X)\ncong X$ for all objects $X$ and all $k\in \Z$. We may also refer to $T$ as the \emph{shift}. Functors between graded categories are required to commute with the shifts.
\end{defn}
When our category is additive, this autoequivalence endows the Grothendieck group with the structure of an $\mathbf{A}=\Z[q,q^{-1}]$-module, where $q[M]=[T(M)]$ and $q$ is an indeterminant. If, moreover, our category has a compatible monoidal structure, then the Grothendieck group has the structure of an $\mathbf{A}$-algebra.

    There is another equivalent definition of an additive $\Z$-graded category. We may also define an additive strictly $\Z$-graded category as an additive category that is enriched in the category of $\Z$-graded abelian groups. In what follows, we freely interchange between morphisms $T(M)\rightarrow N$ and morphisms $M\rightarrow N$ of degree 1. 

Let $\mathcal{C}$ be a strictly $\Z$-graded category with autoequivalence $T$. For each pair of objects $X$ and $Y$ of $\mathcal{C}$ we denote by $\text{Hom}_\mathcal{C}^\bullet(X,Y)$ the $\Z$-graded $\kf$ vector space with $\text{Hom}_\mathcal{C}^\bullet (X,Y)_i=\text{Hom}_\mathcal{C}(T^i(X),Y)$. In case $X=Y$, this vector space has the structure of a graded $\kf$-algebra.

We now define one of the fundamental categories studied in this paper.
\begin{defn}\label{2cat}
    Let $(Q_{ij})_{i,j\in I}$ be a finite $\kf [u,v]$-valued matrix such that $Q_{ii}=0$ for all $i\in I$. We define $\mathcal{A}_1(Q)$ to be the strict monoidal $\kf$-linear category generated by objects $E_i$ and morphisms $X_i:E_i\rightarrow E_i$ and $T_{ij}:E_iE_j\rightarrow E_jE_i$, where $i\in I$. These morphisms satisfy the following relations. 
    \begin{enumerate}        
        \item $T_{ij}\circ X_iE_j-E_iX_j\circ T_{ij}=\delta_{ij}\text{id}_{E_iE_j},$
        \item 
        $T_{ij}\circ E_iX_j-X_jE_i\circ T_{ij}=-\delta_{ij}\text{id}_{E_iE_j}$,
        \item $T_{ij}\circ T_{ji}=Q_{ij}(E_jX_i,X_jE_i)$,
        \item $ T_{jk}E_i\circ E_jT_{ik}\circ T_{ij}E_k- E_kT_{ij}\circ T_{ik}E_j \circ E_iT_{jk}=\delta_{i,k}\frac{Q_{ij}(X_iE_j,E_iX_j)E_i-E_iQ_{ij}(E_jX_i,X_jE_i)}{X_iE_jE_i-E_iE_jX_i}$.
    \end{enumerate}
    The relations above will also be called the \emph{KLR relations}.
\end{defn}
One can show that after formally adding finite direct sums that there is an isomorphism of rings $H_\alpha(Q)\xrightarrow{\sim} 
 \text{End}_{\mathcal{A}_1(Q)}(\bigoplus_{\textbf{i}\in \text{Seq}_\alpha} E_{\textbf{i}_{|\alpha|}}\dots E_{\textbf{i}_1})$. Under this map, 
 \[1_\textbf{i}\rightarrow \text{id}_{E_{\textbf{i}_{|\alpha|}}\dots E_{\textbf{i}_1}},\]
 \[x_i1_\textbf{i}\rightarrow E_{\textbf{i}_{|\alpha|}}\dots E_{\textbf{i}_{i+1}}X_{\textbf{i}_i}E_{\textbf{i}_{i-1}}\dots E_{\textbf{i}_1},\]
 \[\tau_{i}1_\textbf{i}\rightarrow E_{\textbf{i}_{|\alpha|}}\dots E_{\textbf{i}_{i+2}}T_{\textbf{i}_{i+1},\textbf{i}_i}E_{\textbf{i}_{i-1}}\dots E_{\textbf{i}_1}.\]
 Note that the entries of $\textbf{i}$ are written left-to-right in $H_n(Q)$, but are written right-to-left on the right-hand side.

When $Q$ is determined by a symmetrizable Cartan matrix, we can also put a strict $\Z$-grading on $\mathcal{A}_1(Q)$. Since our morphisms are based on the KLR algebra, they come with a natural grading. We use that $\text{deg}(\text{id}_{E_i})=0$, $\text{deg}(X_i)=\text{deg}(x_i)=2$, and $\text{deg}(T_{ij})=\text{deg}(\tau_11_{ij})=-C_{ij}$. We formally add shifted versions of all objects. 
 \begin{defn}
Define $\mathcal{A}_1(C)$ to be the monoidal category whose objects are formal shifts $q^n X$ of those in $\mathcal{A}_1(Q)$ and whose morphisms are defined by $\Hom_{\mathcal{A}_2(C)}(q^n X,q^m Y)=\Hom_{\mathcal{A}_2(Q)}^\bullet(X,Y)_{n-m}$.
 \end{defn}
 
 We would also like this category to be additive and idempotent closed. There are several idempotents in $H_\alpha(Q)$, but the associated morphisms in $\mathcal{A}_1(C)$ do not always have an image object.

 \begin{defn}
     For an additive category $\mathcal{C}$, we denote by $\mathcal{C}^i$ the \emph{idempotent (Karoubi) completion} of $\mathcal{C}$. The objects of this category are pairs $(M,e)$ where $M$ is an object of $\mathcal{C}$ and $e$ is an idempotent endomorphism of $M$. The morphisms from $(M,e)$ to $(M',e')$ are morphisms $f:M\rightarrow M'$ such that $e'fe=f$. the pair $(M,e)$ should then be viewed as the ``image of $e$" as an object of $\mathcal{C}^i$.
 \end{defn}
 Denote by $\mathcal{U}_q^+(C)$ the idempotent completion of the closure of $\mathcal{A}_1(C)$ under finite direct sums.

 It is easier to study this category after showing it is equivalent to another category defined in terms of KLR algebras. First, some notation.

 \begin{defn}
     For $R=\bigoplus_{i\in \Z} R_i$ a graded ring with graded pieces $R_i$, denote by $R-\text{grmod}$ the category whose objects are graded finitely-generated (left) $R$-modules and morphisms are degree zero graded maps. This is a graded category with $T(\{M_i\})_j=M_{j-1}$. Similarly, we have a graded full subcategory $R-\text{grproj}$, which is the category of finitely generated (left) projective graded $R$-modules.  We may consider also graded modules that are not necessarily finitely generated. The corresponding categories are denoted $R-\text{grMod}$ and $R-\text{grProj}$ with otherwise the same notation. 
 \end{defn}
 In the cases we consider, $R_i=0$ for $i<<0$ and $\text{dim}_{\kf}(R_i)<\infty$, so if our module $\{M_i\}$ is finitely generated, we have $M_i=0$ for $i<<0$. 
 \begin{defn}
      For $\alpha\in \N[I]$, denote $\He^{fg}_\alpha:= H_\alpha-\text{grproj}$. Then define $\He^{fg}(C)\coloneqq\bigoplus_{\alpha\in \N[I]} \He^{fg}_\alpha$. This is clearly a strictly graded additive category, and it can be given the structure of a monoidal category as follows. For any $\alpha, \beta\in \N[I]$, there is a canonical injective (non-unital) morphism of graded rings $H_{\alpha}\otimes_{\kf} H_{\beta}\hookrightarrow H_{\alpha+\beta}$ given by concatenation, i.e. $1_{\textbf{i}}\otimes_{\kf} 1_{\textbf{i}'}\rightarrow 1_{\textbf{i}\textbf{i}'}$. The corresponding induction functor can be converted into our desired bifunctor. We define the bifunctor $\otimes:H_{\alpha}-\text{grmod}\times H_{\beta}-\text{grmod}\rightarrow H_{\alpha+\beta}-\text{grmod}$ via $(M,N)\rightarrow H_{\alpha+\beta}\otimes_{\alpha,\beta}(M\otimes_{\kf} N)$, where the outermost tensor is over $H_{\alpha}\otimes_{\kf} H_{\beta}$-modules. One can show that products of projective modules remain projective, and therefore, these products can be restricted to $\He^{fg}(C)$. The unit object of the corresponding monoidal product on $\He^{fg}(C)$ is $H_{0}(C)= \kf\in \mathcal{H}_0^{fg}$. This product endows $K_0(\He^{fg}(C))$ with the structure of an $\mathbf{A}$-algebra. 
 \end{defn}
 For any $\textbf{i}\in \text{Seq}_\alpha$, we have that $H_{\alpha}1_\textbf{i}$ is an object of $\He^{fg}_\alpha$. However, these are in general not indecomposable.
 
 There is an equivalence additive monoidal graded categories $\mathcal{U}^+_q(C)\rightarrow \He^{fg}(C)$ sending $E_i$ to $H_{\alpha_i}(C)\in \He^{fg}_{\alpha_i}$. Since $\mathcal{U}^+_q(C)$ is defined in terms of generators and relations, there is a straightforward process to defining monoidal functors out of it. Since $\He^{fg}(C)$ has very explicit objects, it is easier to define functors acting on it.

 The following result is the key reason why KLR algebras have become so prominent in categorical representation theory. 
 \begin{thm}[\cite{kholau2}]\label{decat}
     There is an isomorphism of $\N[I]$-graded $\mathbf{A}$-algebras $_\mathbf{A}{U_{q}^+(\mathfrak{g}(C))}\simeq  K_0(\He^{fg}(C))$.
 \end{thm}
Under this isomorphism, $E_i\rightarrow [H_{\alpha_i}]$. We henceforth use the symbol $E_i$ to refer to both the element $E_i\in U_q^+(\mathfrak{g}(C))$ and the object $E_i\in \mathcal{U}_q^+(C)$ when no confusion is possible. We also need a related result which was established in Section 3.8.3 of \cite{kholau3}.
\begin{prop}{\cite{kholau3}}\label{prop:klr_tensor_K0}
For any $\alpha,\beta \in \N[I]$, we have that \[K_0(H_\alpha\otimes_{\kf}H_\beta)\simeq K_0(H_\alpha)\otimes_{\mathbf{A}}K_0(H_\beta)\simeq  {_\mathbf{A}U_q^+(\mathfrak{g})_\alpha}\otimes_{\mathbf{A}} {_\mathbf{A}U_q^+(\mathfrak{g})_\beta}\]
\end{prop}

Some of the other classical structures on $U_q^+(\mathfrak{g})$ can be given categorical meaning.

\begin{defn}\label{form}
    There is a $\Z[[q]][q^{-1}]$-valued, $\Z$-bilinear form $(*,*)$ on $K_0(\He^{fg}_\alpha)$ defined by \[([P],[Q])=\text{grdim}(\text{Hom}_{H_\alpha}^\bullet(P,Q)).\] It can be proven that each $H_\alpha$ has finite dimensional graded components and is bounded below in degree. The same therefore follows for each $\text{Hom}_{\mathcal{H}_\alpha^{fg}}^\bullet (X,Y)$, so this form is well-defined. It is $q$-semilinear in the sense that $([P],[T(Q)])=([T^{-1}(P)],[Q])=q*([P],[Q])$. We can extend this form to all of $K_0(\He^{fg})$ by enforcing that if $P\in \He^{fg}_\alpha$ and $Q\in \He^{fg}_\beta$ for $\alpha\neq \beta$ then $([P],[Q])=0$.
\end{defn}

This is  the form $(*,*)'$ studied in \cite{khla}. On $U_q^+(\mathfrak{g})$, the symmetric form $(\overline{*},*)$ can be identified with $(*,*)_L$.

\begin{defn}\label{2rep}
    For an additive strictly $\Z$-graded monoidal category $\mathcal{C}$, we define a \emph{2-representation} of $\mathcal{C}$ on an additive strictly $\Z$-graded category $\mathcal{D}$ to be a strict monoidal additive graded functor $\rho:\mathcal{C}\rightarrow End_{\oplus,\Z}(\mathcal{D})$. Given a 2-representation of $\mathcal{C}$, we obtain a representation of $\mathbf{A}$-algebras of $K_0(\mathcal{C})$ on $K_0(\mathcal{D})$.
\end{defn}
The following example is the basis for our later constructions.
\begin{ex}\label{ex:leftmult}
    The right multiplication representation of $U_q^+(\mathfrak{g})$ is naturally categorified. For $M$ a right $H_\alpha$-module and $N$ a left $H_\alpha$-module, denote by $M \otimes_\alpha N$ the corresponding $H_\alpha$-bilinear tensor product. Recall the monoidal structure on $\He^{fg}(C)$. For any $N\in \He^{fg}_\alpha$, the monoidal product $N\otimes H_{\alpha_i}$ is therefore identified with $H_{\alpha+\alpha_i}1_{*i}\otimes_{\alpha} N$. Here, we are viewing $H_{\alpha+\alpha_1}1_{*i}$ as a unital $(H_{\alpha+\alpha_i},H_\alpha)$-bimodule. We therefore define $E_{i,\alpha}:\He^{fg}_\alpha\rightarrow \He^{fg}_{\alpha+\alpha_i}$ to be the functor $H_{\alpha+\alpha_i}1_{*i}\otimes_\alpha$, and define $E_i:\He^{fg}(C)\rightarrow \He^{fg}(C)$ to be the sum over all $E_{i,\alpha}$. We can also make the identification $E_iE_j=H_{\alpha+\alpha_i+\alpha_j}1_{*ji}\otimes_{\alpha}$. We have natural transformations 
    \[X_i:E_i\rightarrow E_i,\]  \[a\otimes_\alpha b\rightarrow ax_{|\alpha|+1}\otimes_\alpha b\] and 
    \[T_{ij}:E_iE_j\rightarrow E_jE_i,\]
    \[a\otimes_\alpha b\rightarrow a\tau_{|\alpha|+1}\otimes_\alpha b.\]
    
\end{ex}

\subsection{Localization of categories}
We review the calculus of fractions for localization of categories and the necessary adjustments for the extra categorical structure that we need. The calculus of fractions for categories was introduced in \cite{calc_of_frac}. We ignore set-theoretic issues; all categories used in this paper can be taken to be locally small.

\begin{defn}
    Let $\mathcal{C}$ be a category and let $S$ be a collection of the morphisms in $\mathcal{C}$. We say that a category $\mathcal{C}[S^{-1}]$ equipped with a functor $\mathcal{L}:\mathcal{C}\rightarrow \mathcal{C}[S^{-1}]$ is the \emph{localization of} $\mathcal{C}$ \emph{with respect to} $S$ if the following are satisfied.
    \begin{enumerate}
        \item (invertibility of $S$) For each $s\in S$, we have that $Q(s)$ is an isomorphism in the category $\mathcal{C}[S^{-1}]$.
        \item (universal property) Fix a functor $F:\mathcal{C}\rightarrow \mathcal{D}$ such that each $F(s)$ is an isomorphism in $\mathcal{D}$ for $s\in S$. Then there exists a unique functor $\bar{F}:\mathcal{C}[S^{-1}]\rightarrow D$ such that $F=\bar{F}\circ Q$.
    \end{enumerate}
\end{defn}
Such a localization may be constructed by taking a category with the same objects and with morphisms defined by formal sequences of $\mathcal{C}$ morphisms and inverses of $S$ elements. So, when we refer to $\emph{the}$ localization, we implicitly mean a category with the above properties that has the same objects as $\mathcal{C}$ and such that $\mathcal{L}$ preserves objects.

If the collection $S$ satisfies some extra conditions, then the morphisms in the localization become much easier to study.

\begin{defn}\label{defn:calc_of_frac}
    Let $\mathcal{C}$ be a category and denote by $S$ a collection of morphisms in $\mathcal{C}$. We say that $S$ \emph{admits a calculus of right fractions} if the following 4 conditions are satisfied.
    \begin{enumerate}
        \item $S$ contains all identity morphisms.
        \item  $S$ is closed under composition.
        \item Every cospan $X\xrightarrow{f} Y\xleftarrow{s} Z$ with $s\in S$ can be completed to a commutative diagram
        \begin{center}
         \begin{tikzcd}
            W \arrow[r,"g"] \arrow[d,"t"]& Z \arrow[d,"s"] \\
            X \arrow[r,"f"]& Y
        \end{tikzcd}
        \end{center}
   
        with $t\in S$.
        \item For each triplet $(f,g,s)$ such that $f$ and $g$ are morphisms $X\rightarrow Y$ and $s$ is an $S$-morphism from $Y$ to $Z$ such that $s\circ f=s\circ g$, there exists an $S$-morphism $t:W\rightarrow X$ such that $f\circ t=g\circ t$.
    \end{enumerate}
\end{defn}

The following is well known.
\begin{prop}
    Fix $\mathcal{C}$ a category and $S$ a collection of the morphisms in $\mathcal{C}$. Suppose that $S$ admits a calculus of right fractions. Then the localization $\mathcal{C}[S^{-1}]$ admits the following description.
    \begin{itemize}
        \item Objects of $\mathcal{C}[S^{-1}]$ are the same as those in $\mathcal{C}$.
        \item $\Hom_{\mathcal{C}[S^{-1}]}(A,B)$ is the collection of equivalence classes of spans from $A$ to $B$ of the form $A\xleftarrow{s} X \xrightarrow{f} B$ with $f\in \Hom_{\mathcal{C}}(X,B)$ and $s\in S$. We say that two spans $A\xleftarrow{s} X \xrightarrow{f} B$ and $A\xleftarrow{t} X' \xrightarrow{g} B$ are equivalent if there exists a span $X\xleftarrow{a} Y \xrightarrow{a'} X'$ such that $f\circ a = g\circ a'$, $s\circ a=s'\circ a'$, and $s\circ a\in S$.
        \item Given two spans $A\xleftarrow{s} X \xrightarrow{f} B$ and $B\xleftarrow{t} Y \xrightarrow{g} C$, we obtain the composition in $\Hom_{\mathcal{C}[S^{-1}]}(A,C)$ as follows. First, we complete the cospan $X\xrightarrow{f} B \xleftarrow{t} Y$ to a commutative square and obtain the span $X\xleftarrow{a} Z \xrightarrow{b} Y$. Then the composition is the equivalence class of the span $A\xleftarrow{s\circ a} Z \xrightarrow{g\circ b}C$.
    \end{itemize}
\end{prop}
We denote by $fs^{-1}$ the morphism $\Hom_{\mathcal{C}[S^{-1}]}(A,B)$ corresponding to the span $A\xleftarrow{s} X \xrightarrow{f} B$.

One alternative way to show that this category of right fractions exists is to use a coreflective localization of $\mathcal{C}$.
\begin{prop}\label{prop:coreflect_to_localiz}
    Suppose that there is a functor $R:\mathcal{C}\rightarrow \mathcal{D}$ with a fully faithful left adjoint $L:\mathcal{D}\rightarrow \mathcal{C}$. Let $S$ be the set of $\mathcal{C}$-morphisms that are sent to isomorphisms by $R$. Then $S$ admits a calculus of right fractions and the induced functor $\mathcal{C}[S^{-1}]\rightarrow D$ is an equivalence.
\end{prop}
We refer to the essential image of $R$ as a \emph{coreflective subcategory} of $\mathcal{C}$.

Enriched versions of localization also exist. Specifically, we use the following. See \cite{stacks-project} Lemma 12.8.1 for explicit calculations in the case of enrichment over abelian groups. 

\begin{prop}\label{prop:kf_enrich}
    Suppose that $\mathcal{C}$ is a category enriched in $\kf$-vector spaces. Let $S$ be a subset of the set of all morphisms in $\mathcal{C}$ that admits a calculus of right fractions. Then there is a unique structure of a $\kf$-enriched category on $\mathcal{C}[S^{-1}]$ such that the canonical functor $\mathcal{L}$ is a $\kf$-enriched functor. It is defined as follows. Given any two morphisms from $X$ to $Y$, we may apply the right calculus of fractions conditions and assume they have a ``common denominator", i.e. that they can be written $fs^{-1}$ and $gs^{-1}$ for some morphisms $f,g:Z\rightarrow Y$ and some $S$-morphism $s:Z\rightarrow X$. Then their sum is $(f+g)s^{-1}$, and for $r\in \kf$, we define the scalar product $r(fs^{-1})\coloneqq (rf)s^{-1}$. The localization satisfies a universal property with respect to $\kf$-linear categories and $\kf$-linear functors.
\end{prop}

With a minor additional assumption, the localization of a monoidal category is also  monoidal. See \cite{monoidal_local} for related theory.

\begin{prop}\label{prop:monoidal_localiz}
    Let $\mathcal{C}$ be a category and let $S$ be a collection of morphisms in $\mathcal{C}$ that admits a calculus of right fractions. Suppose also that $\mathcal{C}$ is monoidal. If $S$ is closed under the monoidal product, then $\mathcal{C}[S^{-1}]$ admits a monoidal enrichment such that the functor $\mathcal{L}$ is monoidal. This localization satisfies a universal property with respect to monoidal categories and lax monoidal functors. For $\kf$-linear categories, this is compatible with the $\kf$-enrichment of Proposition \ref{prop:kf_enrich}.
\end{prop}
Likewise, if our category is strictly graded, the localization inherits a strict grading so long as $S$ is closed under shifts.

\section{The quantum boson category}\label{sec:qboson}
\subsection{An intermediate category \texorpdfstring{$\mathcal{A}_2(Q)$}{A2(Q)}}\label{subsec:intermediate}
We define a monoidal category which we use in the construction of our quantum boson category. In this category, we do not insist on any direct sums or quantum boson relations.  
\begin{defn}
    Let $I$ be a finite indexing set and let $(Q_{ij})_{i,j\in I}$ be a $\kf[u,v]$-valued matrix satisfying $Q_{ii}=0$ and $Q_{ij}(u,v)=Q_{ji}(v,u)$. We define $\mathcal{A}_2(Q)$ to be the monoidal $\kf$-linear category defined by the following generators. The generating objects are the symbols $E_i$ and $F_i$ for each $i\in I$ along with a monoidal unit $\1$. We express the generating morphisms diagrammatically. 

    \begin{itemize}

    \item For each $i\in I$, we have the identity morphisms $\text{Id}_{E_i}:E_i\rightarrow E_i$, $\text{Id}_{F_i}:F_i\rightarrow F_i$, and $\text{Id}_\1:\1\rightarrow \1$.
    We draw $\text{Id}_\1$ as an empty diagram.
    
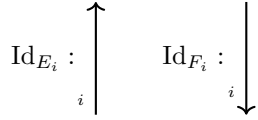
\begin{center}
            \begin{tikzpicture}
        
                \draw (-.7,.75)
                node {$\text{Id}_{E_i}:$};
                
                \draw (-.2,.2) node {\tiny $i$};
              
                \draw[->, thick] (0,0) -- (0,1.5);

                \draw (1.3,.75)
                node {$\text{Id}_{F_i}:$};
                
                \draw (1.8,.3) node {\tiny $i$};
              
                \draw[<-, thick] (2,0) -- (2,1.5);
            \end{tikzpicture}
        \end{center}
    \item For each $i\in I$, we define $X_i:E_i\rightarrow E_i$. 
        \begin{center}
            \begin{tikzpicture}
    
                \draw (-.2,.2) node {\tiny $i$};
                \draw (0, .75) node{$\bullet$};
                \draw[->, thick] (0,0) -- (0,1.5);
            \end{tikzpicture}
        \end{center}

    \item For each $i,j\in I$, we define $T_{ij}:E_iE_j\rightarrow E_jE_i.$
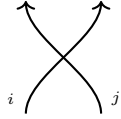
        \begin{center}
            \begin{tikzpicture}
                \draw (1.2, .2) node{\tiny $j$};
                \draw (-.2, .2) node{\tiny $i$};
        
                \draw[->, thick] (0,0) .. controls (0,.5) and (1,1) .. (1,1.5);
                \draw[->, thick] (1,0) .. controls (1,.5) and (0,1) .. (0,1.5);
                
            \end{tikzpicture}
        \end{center}
        
    \item For each $i\in I$, we define $\eta_i: \1\rightarrow F_iE_i$ and $\epsilon_i: E_iF_i\rightarrow \1$.
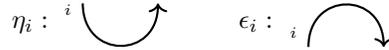
        \begin{center}
                \begin{tikzpicture}
                    \draw (-.7, .3)
                    node{$\eta_i:$};
                    \draw (-0.2, .3) node[above]{\tiny $i$};
                    \draw[->, thick] (0,.58).. controls ++(0,-.75) and ++(0, -.75) .. (1,.58);

                    \draw (2.3, .3)
                    node{$\epsilon_i:$};
                    \draw (2.8, 0) node[above]{\tiny $i$};
                    \draw[->, thick] (3,.0).. controls ++(0,.75) and ++(0, .75) .. (4,0);
                \end{tikzpicture}
        \end{center}

    \end{itemize}

    Objects in this category are words in the $E_i$ and $F_i$. All morphisms in the category are obtained by composing, tensoring, and taking linear combinations of those described above.  
    
    Suppose $g$ and $f$ are morphisms with associated diagrams. Composition of morphisms $g\circ f$ is expressed by placing the diagram for $g$ on top of the diagram for $f$. The monoidal product of morphisms $g\otimes f$ is expressed by placing the diagram for $g$ to the left of the diagram for $f$.
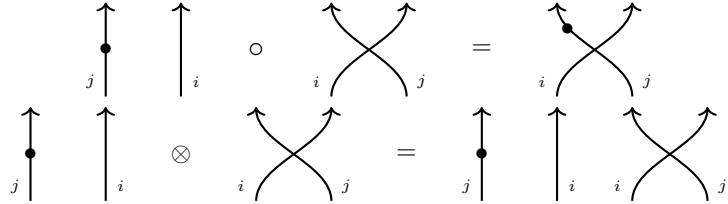
    \begin{center}
            \begin{tikzpicture}
                \draw (1.2, .2) node{\tiny $i$};
                \draw (-.2, .2) node{\tiny $j$};
                \draw[->, thick] (0,0) -- (0,1.25);
                \draw[->, thick] (1,0) -- (1,1.25);
                \draw (0,.625) node{$\bullet$};

                \draw (2,.625) node{\Large $\circ$};
                
                \draw (4.2, .2) node{\tiny $j$};
                \draw (2.8, .2) node{\tiny $i$};
    
                \draw[->, thick] (3,0) .. controls (3,.5) and (4,.75) .. (4,1.25);
                \draw[->, thick] (4,0) .. controls (4,.5) and (3,.75) .. (3,1.25);

                \draw (5,.625) node{$=$};

                \draw (7.2, .2) node{\tiny $j$};
                \draw (5.8, .2) node{\tiny $i$};
    
                \draw[->, thick] (6,0) .. controls (6,.5) and (7,.75) .. (7,1.25);
                \draw[->, thick] (7,0) .. controls (7,.5) and (6,.75) .. (6,1.25);
                \draw (6.14,.9) node{$\bullet$};
            \end{tikzpicture}
        \end{center}

  \begin{center}
            \begin{tikzpicture}
                \draw (1.2, .2) node{\tiny $i$};
                \draw (-.2, .2) node{\tiny $j$};
                \draw[->, thick] (0,0) -- (0,1.25);
                \draw[->, thick] (1,0) -- (1,1.25);
                \draw (0,.625) node{$\bullet$};

                \draw (2,.625) node{$\otimes$};
                
                \draw (4.2, .2) node{\tiny $j$};
                \draw (2.8, .2) node{\tiny $i$};
    
                \draw[->, thick] (3,0) .. controls (3,.5) and (4,.75) .. (4,1.25);
                \draw[->, thick] (4,0) .. controls (4,.5) and (3,.75) .. (3,1.25);

                \draw (5,.625) node{$=$};

                \draw[->, thick] (6,0)--(6,1.25);
                \draw[->, thick] (7,0)--(7,1.25);
                \draw (6,.625) node{$\bullet$};
                \draw (7.2, .2) node{\tiny $i$};
                \draw (5.8, .2) node{\tiny $j$};

                \draw (9.2, .2) node{\tiny $j$};
                \draw (7.8, .2) node{\tiny $i$};
    
                \draw[->, thick] (8,0) .. controls (8,.5) and (9,.75) .. (9,1.25);
                \draw[->, thick] (9,0) .. controls (9,.5) and (8,.75) .. (8,1.25);
                
            \end{tikzpicture}
        \end{center}
    
    We often refer to the monoidal product as the tensor or as horizontal composition. We are not yet making any claims about topological properties of the diagrams associated to morphisms. For the moment, the diagrams are simply a visual aid.

    Before describing the relations, we introduce some notation.
    
    For each $i\in I$ and $m\in \N$, we define $X_i^m:E_i\rightarrow E_i$ as the $m$-fold composition of $X_i$ with itself. For the corresponding diagram, we mark the dot with an $m$. 
        \begin{center}
            \begin{tikzpicture}
    
                \draw (-.2,.2) node {\tiny $i$};
                \draw (0, .75) node{$\bullet$};
                \draw (.2, .75) node{\tiny $m$};
                \draw[->, thick] (0,0) -- (0,1.5);
            \end{tikzpicture}
        \end{center}
    For a more general polynomial in several variables $f\in \kf[x_1,\dots x_n]$ and for each word $E_{i_1}\dots E_{i_n}$, there is an element $f\in \text{End}_{\mathcal{A}_2(Q)}(E_{i_1}\dots E_{i_n})$, where $x_j^l$ corresponds to the morphism $E_{i_1}\dots X_{i_j}^l\dots E_{i_n}$. We denote this diagrammatically as 
    
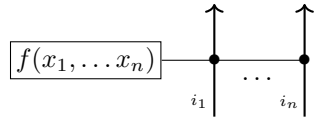
\begin{center}
            \begin{tikzpicture}
    
                \draw (-.2,.2) node {\tiny $i_1$};
                \draw (1,.2) node {\tiny $i_n$};
                \draw (0,.75) node {$\bullet$};
                \draw (1.2,.75) node {$\bullet$};
                \draw (.6,.5) node {$\dots$};
                \draw[->, thick] (0,0) -- (0,1.5);
                \draw[->, thick] (1.2,0) -- (1.2,1.5);
                \draw[-] (1.2,.75) -- (-.7,.75);
                \draw (-1.71,.75) node [inner sep=2pt, draw] {$f(x_1,\dots x_n)$};
            \end{tikzpicture}
        \end{center}

     A named morphism $D$ may be depicted in a diagram as its name inside a box and without any incident braids
    \begin{tikzpicture}
    \draw (4.17,.75) node [draw] {$D$};
    \end{tikzpicture}
    if it is more complex than a generating morphism or a polynomial.
    
    The generating morphisms are required to satisfy the following defining relations.
    \begin{itemize}
    \item (Monoidal condition) For any objects $X,Y$, we have $\text{Id}_X \text{Id}_Y=\text{Id}_{XY}$. This justifies the depiction of the monoidal product of morphisms as a horizontal composition of diagrams. 
    \item For each $i\in I$, $F_i$ is right dual to $E_i$, and the corresponding unit and counit are $\eta_i$ and $\epsilon_i$ respectively. 
            \begin{center}
                \begin{tikzpicture}
                    \draw [->, thick] (0,1.2) to (0,.2) .. controls ++(0,-.5) and ++(0,-.5) .. (1,.2) to (1,.8) .. controls ++(0,.5) and ++(0,.5)..(2,.8) to (2,-.2);
                    \draw (-0.2,.8) node[above]{\tiny $i$};
                    \draw (3, .375) node[above]{$=$};
                    \draw [->, thick] (4,1.2) to (4,-.2);  \draw (3.8,.8) node[above]{\tiny $i$};

                \draw [->, thick] (6,-.2) to (6,.8) .. controls ++(0,.5) and ++(0,.5) .. (7,.8) to (7,.2) .. controls ++(0,-.5) and ++(0,-.5)..(8,.2) to (8,1.2);
                \draw (5.8,.1) node[above]{\tiny $i$};
                \draw (9, .375) node[above]{$=$};
                \draw [<-, thick] (10,1.2) to (10,-.2);
                \draw (9.8,.1) node[above]{\tiny $i$};
                \end{tikzpicture}
            \end{center}

    \item The KLR relations
    \[T_{ij}\circ X_iE_j-E_iX_j\circ T_{ij}=X_jE_i\circ T_{ij}-T_{ij}\circ E_iX_j=\delta_{ij}E_iE_j\]

        \begin{center}
            \begin{tikzpicture}
                \draw (1.2, .2) node{\tiny $j$};
                \draw (-.2, .2) node{\tiny $i$};
                \draw (0, .1)
                node{$\bullet$};
                \draw[->, thick] (0,0) .. controls (0,.5) and (1,.75) .. (1,1.25);
                \draw[->, thick] (1,0) .. controls (1,.5) and (0,.75) .. (0,1.25);

                \draw (1.5,.75)
                node{$-$};

                \draw (3.2, .2) node{\tiny $j$};
                \draw (1.8, .2) node{\tiny $i$};
                \draw (2.95, 1.0)
                node{$\bullet$};
                \draw[->, thick] (2,0) .. controls (2,.5) and (3,.75) .. (3,1.25);
                \draw[->, thick] (3,0) .. controls (3,.5) and (2,.75) .. (2,1.25);

                \draw (3.5,.75)
                node{$=$};

                \draw (5.2, .2) node{\tiny $j$};
                \draw (3.8, .2) node{\tiny $i$};
                \draw (4.05, 1.0)
                node{$\bullet$};
                \draw[->, thick] (4,0) .. controls (4,.5) and (5,.75) .. (5,1.25);
                \draw[->, thick] (5,0) .. controls (5,.5) and (4,.75) .. (4,1.25);

                \draw (5.5,.75)
                node{$-$};

                \draw (7.2, .2) node{\tiny $j$};
                \draw (5.8, .2) node{\tiny $i$};
                \draw (7, .1)
                node{$\bullet$};
                \draw[->, thick] (6,0) .. controls (6,.5) and (7,.75) .. (7,1.25);
                \draw[->, thick] (7,0) .. controls (7,.5) and (6,.75) .. (6,1.25);

                \draw (7.75,.75)
                node{$=\delta_{ij}$};

                \draw[->, thick] (8.25,0) to (8.25,1.25);
                \draw[->, thick] (9.25,0) to (9.25,1.25);
            \end{tikzpicture}
        \end{center}
        \[T_{ij}\circ T_{ji}=Q_{ij}(E_jX_i,X_jE_i)\]
        \begin{center}
        \begin{tikzpicture}
                \draw (.95, .2) node{\tiny $i$};
                \draw (-.2, .2) node{\tiny $j$};
                \draw[->, thick] (0,0) .. controls (0,.25) and (.75,.5) .. (.75,.75);
                \draw[->, thick] (.75,0) .. controls (.75,.25) and (0,.5) .. (0,.75);
                \draw[->, thick] (0,.75) .. controls (0,1) and (.75,1.25) .. (.75,1.5);
                \draw[->, thick] (.75,.75) .. controls (.75,1) and (0,1.25) .. (0,1.5);

                \draw (1.5,.75)
                node{$=$};

                \draw[->, thick] (2,0) to (2,1.5);
                \draw[->, thick] (2.75,0) to (2.75,1.5);
                \draw (2.95, .2) node{\tiny $i$};
                \draw (1.8, .2) node{\tiny $j$};
                \draw (2.75, .75) node{$\bullet$};
                \draw (2,.75) node{$\bullet$};
                \draw[-] (2,.75) -- (3.25,.75);
                \draw (4.17,.75) node [inner sep=2pt, draw] {$Q_{ij}(x_2,x_1)$};
                
            \end{tikzpicture}
        \end{center}
        
        \[ T_{jk}E_i\circ E_jT_{ik}\circ T_{ij}E_k- E_kT_{ij}\circ T_{ik}E_j \circ E_iT_{jk}=\delta_{i,k}\frac{Q_{ij}(X_iE_j,E_iX_j)E_i-E_iQ_{ij}(E_jX_i,X_jE_i)}{X_iE_jE_i-E_iE_jX_i}\]

        \begin{center}
        \begin{tikzpicture}
            \draw (.85, .2) node{\tiny $j$};
            \draw (-.2, .2) node{\tiny $i$};
            \draw (1.65, .2) node{\tiny $k$};
            \draw[->, thick] (0,0) .. controls (0,.45) and (1.5,1.8) .. (1.5,2.25);
            \draw[->, thick] (1.5,0) .. controls (1.5,.45) and (0,1.8) .. (0,2.25);
            \draw[->, thick] (.75,0) .. controls (.75,.25) and (0,.75) .. (0,1.125) .. controls (0,1.5) and (.75,2.0) .. (.75,2.25);

            \draw (2,1.125) node{$-$};
            
            \draw (3.2, .2) node{\tiny $j$};
            \draw (2.3, .2) node{\tiny $i$};
            \draw (4.15, .2) node{\tiny $k$};
            \draw[->, thick] (2.5,0) .. controls (2.5,.45) and (4,1.8) .. (4,2.25);
            \draw[->, thick] (4,0) .. controls (4,.45) and (2.5,1.8) .. (2.5,2.25);
            \draw[->, thick] (3.25,0) .. controls (3.25,.25) and (4,.75) .. (4,1.125) .. controls (4,1.5) and (3.25,2.0) .. (3.25,2.25);

            \draw (4.5, 1.125) node {$=$};

            \draw[->, thick] (5,0)--(5,2.25);
            \draw[->, thick] (5.75,0)--(5.75,2.25);
            \draw[->, thick] (6.5,0)--(6.5,2.25);
            \draw (4.85, .2) node{\tiny $i$};
            \draw (5.55, .2) node{\tiny $j$};
            \draw (6.65, .2) node{\tiny $k$};
            \draw (5,1.125) node{$\bullet$};
            \draw (5.75,1.125) node{$\bullet$};
            \draw (6.5,1.125) node{$\bullet$};
            \draw [-] (5,1.125)--(7,1.125);
            \draw (8.81,1.125) node[inner sep=2pt, draw]{$\delta_{i,k}\frac{Q_{ij}(x_1,x_2)-Q_{ij}(x_3,x_2)}{x_1-x_3}$};
        \end{tikzpicture}
        \end{center}
    \end{itemize}

\end{defn}

It is important to note that the $E_i$ and $F_i$ are only dual on one side unlike the $E_i$ and $F_i$ in the categorification of the full idempotented quantum groups of \cite{kholau3} and \cite{qha}.
\begin{rem}
   An idempotented version of $\mathcal{A}_2(Q)$ is used implicitly in the categorification of the full idempotented quantum groups. 
\end{rem}

We introduce more notation for certain common morphisms. 
\begin{defn}
    Define Seq as in Definition \ref{defn:seq} for the set $I$. Given $\textbf{i}\in$ Seq, we associate objects $E_\textbf{i}$ of $\mathcal{A}_2(Q)$ via $i\rightarrow E_i$ for each $i\in I$ and concatenation of sequences corresponds to monoidal product of objects. We similarly define the $F_{\textbf{i}}$.
\end{defn} 

\begin{defn}
    \begin{itemize}
        \item For each $i\in I$ and $m\in \N$, we define with a slight abuse of notation $X_i^m:F_i\rightarrow F_i$ as the morphism. 
        \begin{center}
            \begin{tikzpicture}
    
                \draw (-.2,1.3) node {\tiny $i$};
                \draw (0, .75) node{$\bullet$};
                \draw (.2,.75) node{\tiny $m$};
                \draw[<-, thick] (0,0) -- (0,1.5);

                \draw (1,.5) node[above]{$\coloneqq$};
                \draw[-, thick] (2,0).. controls ++(0,-.5) and ++(0, -.5) .. (3,0);
                \draw (3, .75) node{$\bullet$};
                \draw (3.2,.75) node{\tiny $m$};
                \draw[->, thick] (3,0) -- (3,1.5);
                \draw[-, thick] (3,1.5).. controls ++(0,.5) and ++(0, .5) .. (4,1.5);

                \draw[->, thick] (2,1.5) to (2,0);
                \draw[->, thick] (4,1.5) to (4,0);

                \draw(1.8,1.3) node{\tiny $i$};
            \end{tikzpicture}
        \end{center}

        \item Fix nonempty $\textbf{i}\in $ Seq and denote $n\coloneqq |\textbf{i}|$. Fix $\sigma\in S_n$. Then for any fixed reduced presentation $\sigma=s_{i_k}\dots s_{i_1}$, we define the morphism $T_\sigma:E_{\textbf{i}}\rightarrow E_{\sigma(\textbf{i})}$ as the morphism $E_{\textbf{i}_{\leq i_k}}T_{i_k}E_{\textbf{i}_{\geq i_{k}+1}}\circ \dots E_{\textbf{i}_{\leq i_1}}T_{i_1}E_{\textbf{i}_{\geq i_{1}+1}}$. Although these morphisms depend on $\textbf{i}$ and on a reduced presentation for $\sigma$, we will only use this notation once these choices are fixed. 
        
        \item For each $i,j\in I$, we define $\sigma_{ij}:E_iF_j\rightarrow F_jE_i$ as the morphism $\sigma_{ij}\coloneqq F_jE_i\epsilon_j \circ F_jT_{ji}F_j \circ \eta_j E_i F_j$.
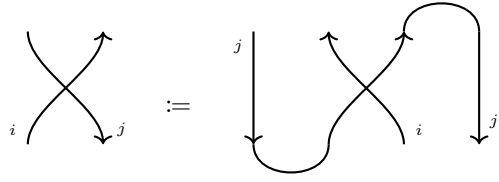
        \begin{center}
            \begin{tikzpicture}
                \draw (1.25, .2) node{\tiny $j$};
                \draw (-.2, .2) node{\tiny $i$};
        
                \draw[->, thick] (0,0) .. controls (0,.5) and (1,1) .. (1,1.5);
                \draw[<-, thick] (1,0) .. controls (1,.5) and (0,1) .. (0,1.5);

                \draw (2,.5) node{$\coloneqq$};

                \draw[-, thick] (3,0).. controls ++(0,-.5) and ++(0,-.5) .. (4,0);
        
                \draw[->, thick] (4,0) .. controls (4,.5) and (5,1) .. (5,1.5);
                \draw[->, thick] (5,0) .. controls (5,.5) and (4,1) .. (4,1.5);

                \draw[-, thick] (5,1.5).. controls ++(0,.5) and ++(0,.5) .. (6,1.5);

                \draw[->, thick] (3,1.5) to (3,0);
                \draw[->, thick] (6,1.5) to (6,0);

                \draw (2.8,1.3) node{\tiny $j$};
                \draw (6.2,.3) node{\tiny $j$};
                \draw (5.2,.2) node{\tiny $i$};
            \end{tikzpicture}
        \end{center}

        \item For each $i,j\in I$, we define with a slight abuse of notation $T_{ij}:F_iF_j\rightarrow F_jF_i$ as the morphism $T_{ij}\coloneqq F_jE_i\epsilon_j \circ F_j\sigma_{ji}F_j \circ \eta_j E_i F_j$.
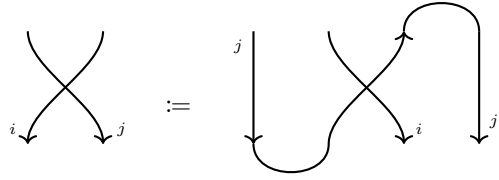
        \begin{center}
            \begin{tikzpicture}
                \draw (1.25, .2) node{\tiny $j$};
                \draw (-.2, .2) node{\tiny $i$};
        
                \draw[<-, thick] (0,0) .. controls (0,.5) and (1,1) .. (1,1.5);
                \draw[<-, thick] (1,0) .. controls (1,.5) and (0,1) .. (0,1.5);

                \draw (2,.5) node{$\coloneqq$};

                \draw[-, thick] (3,0).. controls ++(0,-.5) and ++(0,-.5) .. (4,0);
        
                \draw[->, thick] (4,0) .. controls (4,.5) and (5,1) .. (5,1.5);
                \draw[<-, thick] (5,0) .. controls (5,.5) and (4,1) .. (4,1.5);

                \draw[-, thick] (5,1.5).. controls ++(0,.5) and ++(0,.5) .. (6,1.5);

                \draw[->, thick] (3,1.5) to (3,0);
                \draw[->, thick] (6,1.5) to (6,0);

                \draw (2.8,1.3) node{\tiny $j$};
                \draw (6.2,.3) node{\tiny $j$};
                \draw (5.2,.2) node{\tiny $i$};
            \end{tikzpicture}
        \end{center}

        \item Fix nonempty $\textbf{i}\in $ Seq and denote $n\coloneqq |\textbf{i}|$. Fix $\sigma\in S_n$. Then for any fixed reduced presentation $\sigma=s_{i_k}\dots s_{i_1}$, we define with a slight abuse of notation the morphism $T_\sigma:F_{\textbf{i}}\rightarrow F_{\sigma(\textbf{i})}$ as the morphism $F_{\textbf{i}_{\leq i_k}}T_{i_k}F_{\textbf{i}_{\geq i_{k}+1}}\circ \dots F_{\textbf{i}_{\leq i_1}}T_{i_1}F_{\textbf{i}_{\geq i_{1}+1}}$. Although these morphisms depend on $\textbf{i}$ and on a reduced presentation for $\sigma$, we will only use this notation once these choices are fixed. 
    \end{itemize}
\end{defn}

A few relations involving these additional morphisms can be deduced quickly from the definitions. We will not need these, but they are convenient for calculations and for visualization. We depict only the simplest of these extra relations and leave the rest to the reader. Note that each of these relations shows that certain diagrams represent the same morphism if they are equal up to ambient isotopy and perhaps certain error terms with fewer crossings.

\begin{lem}\label{lem:easy_qbos_rels}
    The following relations of morphisms hold in $\mathcal{A}_2(Q)$ for any $i,j\in I$ and $m,n\in \N$.
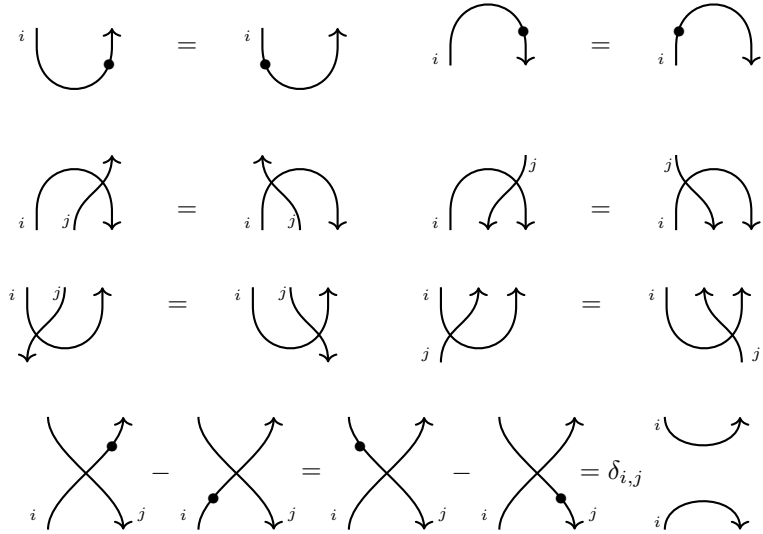
        \begin{equation*}            \begin{tikzpicture}
                    \draw (-0.2, .7) node[above]{\tiny $i$};
                    \draw[->, thick] (0,1) to (0,.75).. controls ++(0,-.75) and ++(0, -.75) .. (1,.75) to (1,1);
                    \draw (.96,.5) node{$\bullet$};

                    \draw (2,.75) node{$=$};

                    \draw (2.8, .7) node[above]{\tiny $i$};
                    \draw[->, thick] (3,1) to (3,.75).. controls ++(0,-.75) and ++(0, -.75) .. (4,.75) to (4,1);
                    \draw (3.04,.5) node{$\bullet$};

                    \draw (5.3, .4) node[above]{\tiny $i$};
                    \draw[->, thick] (5.5,.5) to (5.5,.75).. controls ++(0,.75) and ++(0, .75) .. (6.5,.75) to (6.5,.5);
                    \draw (6.47,.95) node{$\bullet$};

                    \draw (7.5,.75) node{$=$};

                    \draw (8.3, .4) node[above]{\tiny $i$};
                    \draw[->, thick] (8.5,.5) to (8.5,.75).. controls ++(0,.75) and ++(0, .75) .. (9.5,.75) to (9.5,.5);
                    \draw (8.54,.95) node{$\bullet$};

                \end{tikzpicture}
        \end{equation*}

    \begin{equation*}            \begin{tikzpicture}
                    \draw (-0.2, .4) node[above]{\tiny $i$};
                    \draw[->, thick] (0,.5) to (0,.75).. controls ++(0,.75) and ++(0, .75) .. (1,.75) to (1,.5);
                    \draw[->, thick] (.5,.5) .. controls ++(0,.5) and ++(0,-.5) .. (1,1.5);
                    \draw (.4, .4) node[above]{\tiny $j$};
                    
                    \draw (2,.75) node{$=$};

                    \draw (2.8, .4) node[above]{\tiny $i$};
                    \draw[->, thick] (3,.5) to (3,.75).. controls ++(0,.75) and ++(0, .75) .. (4,.75) to (4,.5);
                    \draw[->, thick] (3.5,.5) .. controls ++(0,.5) and ++(0,-.5) .. (3,1.5);
                    \draw (3.4, .4) node[above]{\tiny $j$};

                    \draw (5.3, .4) node[above]{\tiny $i$};
                    \draw[->, thick] (5.5,.5) to (5.5,.75).. controls ++(0,.75) and ++(0, .75) .. (6.5,.75) to (6.5,.5);
                    \draw[<-, thick] (6,.5) .. controls ++(0,.5) and ++(0,-.5) .. (6.5,1.5);
                    \draw (6.6, 1.15) node[above]{\tiny $j$};
                    
                    \draw (7.5,.75) node{$=$};

                    \draw (8.3, .4) node[above]{\tiny $i$};
                    \draw[->, thick] (8.5,.5) to (8.5,.75).. controls ++(0,.75) and ++(0, .75) .. (9.5,.75) to (9.5,.5);
                    \draw[<-, thick] (9,.5) .. controls ++(0,.5) and ++(0,-.5) .. (8.5,1.5);
                    \draw (8.4, 1.15) node[above]{\tiny $j$};
                    
                \end{tikzpicture}
        \end{equation*}

        \begin{equation*}                    \begin{tikzpicture}
                    \draw (-0.2, .7) node[above]{\tiny $i$};
                    \draw (.4, .7) node[above]{\tiny $j$};
                    \draw[->, thick] (0,1) to (0,.75).. controls ++(0,-.75) and ++(0, -.75) .. (1,.75) to (1,1);
                    \draw[<-, thick] (0,0).. controls ++(0,.5) and ++(0,-.5) .. (.5,1);

                    \draw (2,.75) node{$=$};

                    \draw (2.8, .7) node[above]{\tiny $i$};
                    \draw[->, thick] (3,1) to (3,.75).. controls ++(0,-.75) and ++(0, -.75) .. (4,.75) to (4,1);
                    \draw[<-, thick] (4,0).. controls ++(0,.5) and ++(0,-.5) .. (3.5,1);
                    \draw (3.4, .7) node[above]{\tiny $j$};

                    \draw (5.3, .7) node[above]{\tiny $i$};
                    \draw (5.3, -.1) node[above]{\tiny $j$};
                    \draw[->, thick] (5.5,1) to (5.5,.75).. controls ++(0,-.75) and ++(0, -.75) .. (6.5,.75) to (6.5,1);
                    \draw[->, thick] (5.5,0).. controls ++(0,.5) and ++(0,-.5) .. (6,1);

                    \draw (7.5,.75) node{$=$};

                    \draw (8.3, .7) node[above]{\tiny $i$};
                    \draw[->, thick] (8.5,1) to (8.5,.75).. controls ++(0,-.75) and ++(0, -.75) .. (9.5,.75) to (9.5,1);
                    \draw[->, thick] (9.5,0).. controls ++(0,.5) and ++(0,-.5) .. (9,1);
                    \draw (9.7, -.1) node[above]{\tiny $j$};
        \end{tikzpicture}
    \end{equation*}

    \begin{equation*}
            \begin{tikzpicture}
                \draw (1.25, .2) node{\tiny $j$};
                \draw (-.2, .2) node{\tiny $i$};
        
                \draw[->, thick] (0,0) .. controls (0,.5) and (1,1) .. (1,1.5);
                \draw[<-, thick] (1,0) .. controls (1,.5) and (0,1) .. (0,1.5);

                \draw (.85, 1.1) node{$\bullet$};
                \draw (1.5,.75) node{$-$};

                \draw (3.25, .2) node{\tiny $j$};
                \draw (1.8, .2) node{\tiny $i$};
        
                \draw[->, thick] (2,0) .. controls (2,.5) and (3,1) .. (3,1.5);
                \draw[<-, thick] (3,0) .. controls (3,.5) and (2,1) .. (2,1.5);

                \draw (2.2, .4) node{$\bullet$};

                \draw (3.5,.75) node{$=$};

                \draw (5.25, .2) node{\tiny $j$};
                \draw (3.8, .2) node{\tiny $i$};
        
                \draw[->, thick] (4,0) .. controls (4,.5) and (5,1) .. (5,1.5);
                \draw[<-, thick] (5,0) .. controls (5,.5) and (4,1) .. (4,1.5);

                \draw (5.5,.75) node{$-$};

                \draw (7.25, .2) node{\tiny $j$};
                \draw (5.8, .2) node{\tiny $i$};
        
                \draw[->, thick] (6,0) .. controls (6,.5) and (7,1) .. (7,1.5);
                \draw[<-, thick] (7,0) .. controls (7,.5) and (6,1) .. (6,1.5);
                \draw (4.15,1.1) node{$\bullet$};
                \draw (6.82,.4) node{$\bullet$};

                \draw (7.5, .75) node{$=\delta_{i,j}$};

                \draw (8.1, 1.4) node{\tiny $i$};
                \draw (8.1, .1) node{\tiny $i$};
                
                \draw[->, thick] (8.2,0) .. controls (8.2,.5) and (9.2,.5) .. (9.2,0);
                \draw[->, thick] (8.2,1.5) .. controls (8.2,1) and (9.2,1) .. (9.2,1.5);
            \end{tikzpicture}
    \end{equation*}

\end{lem}
\begin{proof}
    These relations can be proven directly from the definitions via the usual adjunction techniques. See \cite{brundan2km} for a more detailed explanation.
\end{proof}

The objects of $\mathcal{A}_2(Q)$ can be viewed as words in the alphabet $I$ where each letter has a formal sign attached. We introduce some notation to make this precise.

\begin{defn}
    Denote by $\pm I$ the set of all elements in $I$ with formal signs attached. Denote by SSeq the set of \emph{signed sequences} of $I$, i.e. finite sequences in $\pm I$. We allow the empty sequence. Seq is naturally a subset of SSeq, where we pick a positive sign for each element of $I$. We extend all of the notation of Definition \ref{defn:seq} to this context. For $\textbf{i}\in$ SSeq, we say that a \emph{positive entry} of $\textbf{i}$ is an entry $\textbf{i}_n$ that has a positive sign. We likewise define negative entries. For $\textbf{i}\in$ SSeq, denote by $-\textbf{i}$ the sequence where all entries have their signs flipped.
    
    Every object of $\mathcal{A}_2(Q)$ determines a signed sequence as follows. We define a function $sseq:\text{Ob}(\mathcal{A}_2(Q))\rightarrow \text{SSeq}$ via $sseq(E_i)= i$, $sseq(F_i)= -i$, and a monoidal product of symbols maps to a concatenation of sequences. The function $sseq$ is a bijection from $\text{Ob}(\mathcal{A}_2(Q))$ to SSeq. So, for $\textbf{i}\in $ SSeq, denote by $E_{\textbf{i}}$ the corresponding object of $\mathcal{A}_2(Q)$. We often identify $\text{Ob}(\mathcal{A}_2(Q))$ with SSeq. We also often identify an entry of a sequence $\textbf{i}$ with its corresponding tensor component in $E_{\textbf{i}}$. 

\end{defn}

Any morphism in $\mathcal{A}_2(Q)$ that can be described as a vertical and horizontal composite of the generating morphisms has some diagram associated to it. For such a morphism $D:E_\textbf{i}\rightarrow E_\textbf{j}$ with $\textbf{i},\textbf{j}\in $ SSeq, mark $|\textbf{i}|$ points on $\R\times \{0\}$ in order according to the sequence $\textbf{i}$ and mark $|\textbf{j}|$ points on $\R\times \{1\}$ in order according to the sequence $\textbf{j}$. The diagram associated to $D$ is the decorated image of an immersed 1-dimensional manifold with boundary inside $\R\times [0,1]$. We write the source of this immersion as a set of braids $[0,1]^ {\coprod(|\textbf{ij}|)/2}\sqcup [0,1]^{\coprod n}$ into $\R\times [0,1]$, where the endpoints of each of the first $(|\textbf{ij}|)/2$ intervals are mapped onto the marked points on $\R\times \{0,1\}$. We also allow for now decorated immersions from an extra component $[0,1]^{\coprod n}$ where for each of the $n$ intervals, the images of the endpoints coincide. These correspond to closed loops in our diagram, although we eventually show that no closed loops are possible i.e. we may take $n=0$. The image of each copy of $[0,1]$ will be referred to as a $\emph{braid}$. Each braid is labeled by a vertex $i$ and is oriented, and the endpoints of each braid are compatible with the label and orientation. We refer to the image of any subinterval $[a,b]$ of a braid as a \emph{strand}. Any number of dots may appear on any open interval of any braid away from local maxima, minima, and crossings.

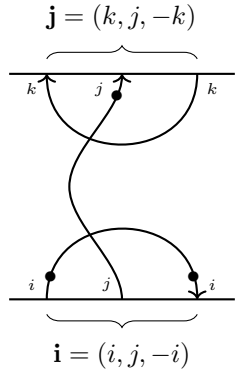
\begin{center}
    \begin{tikzpicture}
        \draw[-,thick] (-.5,0)--(2.5,0);
        \draw[-,thick] (-.5,3)--(2.5,3);
        \draw[->,thick] (0,0) .. controls (0,1.25) and (2,1.25) .. (2,0);
        \draw[->,thick] (2,3) .. controls (2,1.75) and (0,1.75) .. (0,3);
        \draw[->, thick] (1,0) .. controls (1,.5) and (.3,1) .. (.3,1.5) .. controls (.3,2) and (1,2.5)..(1,3);

        \draw (.05,.3) node{$\bullet$};
        \draw (1.95,.3) node{$\bullet$};
        \draw (.93, 2.7) node{$\bullet$};

        \draw (-.2,.2) node{\tiny$i$};
        \draw (.8,.2) node{\tiny$j$};
        \draw (2.2,.2) node{\tiny$i$};
        \draw (-.2,2.8) node{\tiny$k$};
        \draw (.7,2.8) node{\tiny$j$};
        \draw (2.2,2.8) node{\tiny$k$};

        \draw[decorate,decoration={brace,amplitude=5pt,raise=-1ex}](0,3.35) -- (2, 3.35);
        \draw (1, 3.75) node{$\textbf{j}=(k, j,-k)$};

        \draw[decorate,decoration={brace,amplitude=5pt,mirror, raise=-1ex}](0,-.35) -- (2, -.35);
        \draw (1, -.75) node{$\textbf{i}=(i, j, -i)$};
    \end{tikzpicture}
\end{center}

We do not yet claim anything about the topological properties of these diagrams, i.e. about the presence of self-intersections or loops, although we will prove that these cannot be realized by morphisms in our category. We will also show that the precise spacing of entries on the boundary or the precise bending of braids to accommodate the $\eta_i$ and $\epsilon_i$ is not relevant, so we avoid detailing this.

\begin{defn}
     For any $\textbf{i}$, $\textbf{j}\in $ SSeq, we say that an $\mathcal{A}_2$-diagram from $E_{\textbf{i}}$ to $E_{\textbf{j}}$ is any diagram that can be obtained as the diagram associated to the morphism 
     \[A_{i_k}f_{i_k}B_{i_k}\circ \dots \circ A_{i_1}f_{i_1}B_{i_1},\]
     where each $f_{i_j}$ is among our generating morphisms, the $A_{i_j}$ and $B_{i_j}$ can be any objects, the domain of $A_{i_1}f_{i_1}B_{i_1}$ is $E_{\textbf{i}}$, and the codomain of $A_{i_k}f_{i_k}B_{i_k}$ is $E_{\textbf{j}}$. We say that the \emph{height} of the diagram associated to such a decomposition is $k$. Note that morphisms may have multiple such decompositions of potentially different heights. There is a natural function from the set of $\mathcal{A}_2$-diagrams from $X$ to $Y$ to the Hom space $\Hom_{\mathcal{A}_2(Q)}(X,Y)$, and we often identify an $\mathcal{A}_2$-diagram from $X$ to $Y$ with its image under this map. Since no two caps or cups can be at the same height in an $\mathcal{A}_2$-diagram, each braid in the diagram not connecting two entries of the target has a unique global maximum within $\R\times [0,1]$, and likewise each braid not connecting two entries of the source has a unique global minimum.
\end{defn}
For morphisms $f:A\rightarrow B$ and $g:X\rightarrow Y$, we can always factor $fg=fY\circ Ag$. So, the morphisms associated to all of the $\mathcal{A}_2$-diagrams from $X$ to $Y$ span $\Hom_{\mathcal{A}_2(Q)}(X,Y)$.

The category $\mathcal{A}_2(Q)$ has an interesting symmetry that will descend to the quantum boson category. Let $\mathcal{A}_2(Q)^{co}$ denote the monoidal category that is isomorphic to $\mathcal{A}_2(Q)$ as categories and that has reversed monoidal product. 

\begin{defn}
    
There is an isomorphism $\phi$ of monoidal categories \begin{align*}
\phi:\mathcal{A}_2(Q)&\rightarrow  \mathcal{A}_2(Q)^{co}\\
\phi(E_i)&=F_i,\\
\phi(F_i)&=E_i,\\
\phi(X_i:E_i\rightarrow E_i)&=X_i:F_i\rightarrow F_i,\\
\phi(T_{ij}:E_iE_j\rightarrow E_jE_i)&=T_{ji}:F_jF_i\rightarrow F_iF_j,\\
\phi(\eta_i)&=\eta_i,\\
\phi(\epsilon_i)&=\epsilon_i.
\end{align*}
In general, $\phi(E_\textbf{i})=E_{-\overline{\textbf{i}}}$. Informally, the functor $\phi$ flips an $\mathcal{A}_2$-diagram over the $y$-axis and inverts the orientation of each braid. We see that $\phi$ is an isomorphism because its square is the identity.
\end{defn}

\begin{center}
\begin{tikzpicture}
        \draw (-.85,.5) node{$\phi$};
        \draw (-.6,.5) node{\Huge$($};
        
        \draw[->,thick] (0,0) .. controls (0,.9) and (2,.9) .. (2,0);
        \draw[->,thick] (1,0) .. controls (1,.5) and (0,.5)..(0,1);
        \draw (-.2,.1) node{\tiny$i$};
        \draw (.8,.1) node{\tiny$j$};
        \draw (2.2,.1) node{\tiny$i$};
        \draw (.1,.3) node{$\bullet$};

        \draw (2.75,.5) node{\Huge$)$ \normalsize$=$};

        \draw[->,thick] (3.5,0) .. controls (3.5,.9) and (5.5,.9) .. (5.5,0);
        \draw[<-,thick] (4.5,0) .. controls (4.5,.5) and (5.5,.5)..(5.5,1);
        \draw (3.3,.1) node{\tiny$i$};
        \draw (4.3,.1) node{\tiny$j$};
        \draw (5.7,.1) node{\tiny$i$};
        \draw (5.4,.3) node{$\bullet$};
\end{tikzpicture}
\end{center}
There are other symmetries from $\mathcal{A}_2(Q)$ to its opposite (contravariant) category, but they do not descend to the quantum boson category, so we will not detail them.

We describe a 2-representation of $\mathcal{A}_2(Q)$ that will be used in the next subsections.

\begin{prop}\label{prop:a2_2rep_exists}
    There is a 2-representation of $\mathcal{A}_2(Q)$ on $\bigoplus_{\alpha} H_\alpha(Q)-\text{Mod}$ extending the right-multiplication 2-representation of Example \ref{ex:leftmult}.
\end{prop}
\begin{proof}
    We need only produce a right adjoint to the induction functor $E_i=H_{\alpha+\alpha_i}1_{*i}\otimes_\alpha -$. We see that the restriction $F_i=1_{*i}H_{\alpha+\alpha_i}\otimes_{\alpha+\alpha_i} -$ is the (unique) right adjoint. The counit $\epsilon_i$ is given by the bimodule homomorphism $H_{\alpha}1_{*i}\otimes_{\alpha-\alpha_i}1_{*i}H_\alpha \rightarrow H_{\alpha}$, $a\otimes b\rightarrow ab$, and the unit $\eta_i: H_{\alpha}\rightarrow 1_{*i}H_{\alpha}1_{*i}$ is the non-unital left inclusion of bimodules with $1_{\textbf{i}}a\rightarrow 1_{\textbf{i}i}a$.
\end{proof}
\subsection{Topological properties of morphisms}\label{subsec:topology}
We show that certain minimal $\mathcal{A}_2$-diagrams are determined by more basic combinatorial data. We introduce some more notation first.

\begin{defn}
    Given $\textbf{i},\textbf{j}\in$ SSeq, we say that an  $(\textbf{i},\textbf{j})$-\emph{pairing} is a perfect matching of the ordered sets of entries in $\textbf{i}$ and $\textbf{j}$ such that each entry is mapped to an entry with the same vertex label and compatible orientation. Note that if $|\textbf{ij}|$ is not even then there are no such pairings. 
\end{defn}

Any $\mathcal{A}_2$-diagram from $E_\textbf{i}$ to $E_\textbf{j}$ determines an $(\textbf{i},\textbf{j})$-pairing by matching the two endpoints of each braid. Note that not all pairings can arise in this way. For example, there is a nontrivial $(\emptyset,i\cdot -i)$-pairing, but we show below that $\Hom_{\mathcal{A}_2(Q)}(\1 , E_iF_i)=0$. Essentially, the possible orientations of each braid are restricted due to the fact that only certain oriented local maxima and minima are possible. We therefore denote by $p'(\textbf{i},\textbf{j})$ the subset of all $(\textbf{i},\textbf{j})$-pairings that can be realized by an $\mathcal{A}_2$-diagram between the corresponding objects. If $m\in p'(\textbf{i},\textbf{j})$ is the pairing associated to some $\mathcal{A}_2$-diagram $D$ from $E_\textbf{i}$ to $E_{\textbf{j}}$, then for each pair $(a,b)$ of $m$, we say that $D$ \emph{pairs} or \emph{connects} $a$ with $b$.

In what follows, we will use dashed lines in our diagrams to show that specific pairings are induced without describing the exact morphism. We may occasionally suppress vertex labels or irrelevant braids in our diagrams for visual clarity. 

The following lemma shows that certain pairings cannot arise from $\mathcal{A}_2$-diagrams. Note that all of the following forbidden pairings are possible in Khovanov, Lauda, and Rouquier's categorification of the idempotented quantum group since the $E_i$ and $F_i$ are two-sided adjoints here \cite{kholau3}\cite{qha}. 

\begin{lem}\label{lem:a2_forbidden_pairings}
    For $X$ and $Y\in \mathcal{A}_2(Q)$, and $D$ an $\mathcal{A}_2$-diagram from $X$ to $Y$, the corresponding pairing in $p'(sseq(X),sseq(Y))$ has none of the following forbidden pairs.
    \begin{enumerate}
        \item ``No leftwards caps" A negative entry $a$ of $sseq(X)$ paired with a positive entry $b$ of $sseq(X)$ such that $a$ is an earlier entry than $b$
        \begin{center}
            \begin{tikzpicture}
                \draw[->,thick,dashed] (1,0)..controls (1,.5) and (0,.5) .. (0,0);
                \draw[-,thick] (-.4,0)--(1.4,0);
            \end{tikzpicture}
        \end{center}
        \item ``No leftwards cups" A positive entry $a$ of $sseq(Y)$ paired with a negative entry $b$ of $sseq(Y)$ such that $a$ is an earlier entry than $b$

        \begin{center}
            \begin{tikzpicture}
                \draw[->,thick,dashed] (1,0)..controls (1,-.5) and (0,-.5) .. (0,0);
                \draw[-,thick] (-.4,0)--(1.4,0);
            \end{tikzpicture}
        \end{center}
        
        \item ``No leftwards crossings". A pair of entries $(a,b)$ of $sseq(X)$ and a pair of entries $(c,d)$ of $Y$ such that $a$ is paired with $d$, $b$ is paired with $c$, $a$ is negative, $b$ is positive, $a$ is an earlier entry in $sseq(X)$ than $b$, and $c$ is an earlier entry in $sseq(Y)$ than $d$.

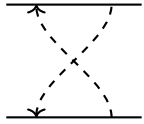
        \begin{center}
            \begin{tikzpicture}        
                \draw[<-, thick,dashed] (0,0) .. controls (0,.5) and (1,1) .. (1,1.5);
                \draw[->, thick,dashed] (1,0) .. controls (1,.5) and (0,1) .. (0,1.5);
                \draw[-,thick] (-.4,0)--(1.4,0);
                \draw[-,thick] (-.4,1.5)--(1.4,1.5);

            \end{tikzpicture}
        \end{center}
    \end{enumerate}
\end{lem}

\begin{proof}
    We argue by induction on the height of our diagram. The base case of height 1 is clear by inspecting the defining morphisms. For the inductive step, fix an $\mathcal{A}_2$-diagram $D$ from $X$ to $Y$ of height at least $2$ and suppose its pairing has any of the above forbidden pairs. We show that we can obtain an $\mathcal{A}_2$-diagram of shorter height, possibly with different source and target, with one of the above forbidden pairs.

    Suppose $D$ pairs a negative entry $a$ of $sseq(X)$ with a positive entry $b$ of $sseq(X)$ such that $a$ is an earlier entry than $b$. Then in $D$, the braid connecting $a$ with $b$ must achieve a unique global maximum. This global maximum is necessarily within the rightwards oriented $W\epsilon_iW'$ for some $W$ and $W'$. Factorize $D$ as $D''\circ W\epsilon_iW'\circ D'$. The only way for $D'$ to connect the endpoints of this $\epsilon_i$ to $a$ and $b$ that is compatible with orientations and this global maximum is to connect the right endpoint to $a$ and the left endpoint to $b$. It is easy to see that this leads to a forbidden pair of the third type within $D'$, which is shorter than $D$. The argument for when $D$ has a forbidden pair of the second type is similar. 

    
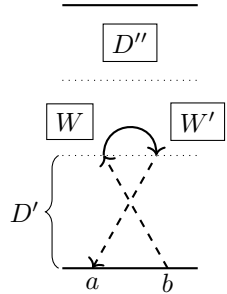
\begin{center}
        \begin{tikzpicture}
            \draw[-,thick] (-.4,0) -- (1.4,0);
            \draw[dotted] (-.4, 1.5) -- (1.4,1.5);
            \draw[dotted] (-.4, 2.5) -- (1.4,2.5);
            \draw[-,thick] (-.4,3.5) -- (1.4,3.5);

            \draw (.5,3) node[draw]{$D''$};
            \draw[->,thick] (.15,1.5) ..controls (.15,2) and (.85,2) .. (.85,1.5);
            \draw (-.3, 1.95) node[draw]{$W$};
            \draw (1.4, 1.95) node[draw]{$W'$};

            \draw[<-, thick, dashed] (0,0) -- (.85,1.5);
            \draw[->, thick, dashed] (1,0) -- (.15,1.5);
            \draw (0,-.2) node{$a$};
            \draw (1,-.2) node{$b$};

            \draw (-.9,.75) node{$D'$};
            \draw[decorate,decoration={brace,amplitude=5pt, raise=-1ex}](-.6,0) -- (-.6, 1.5);

        \end{tikzpicture}
    \end{center}

    Now suppose that the pairing associated to $D$ has entries $a,b$ of $X$ and entries $c,d$ of $Y$ as in the third type of forbidden pair.  Then the braid connecting $a$ to $d$ must cross that connecting $b$ to $c$. We factor $D$ as $D''\circ D'$, where $D''$ is height 1, and write that the target of $D'$, which is the source of $D''$, is $Z$. The diagram $D'$ must pair both $a$ and $b$ to some entries $a'$ and $b'$ in $Z$, or else $D$ cannot pair them with $c$ and $d$. If $b'$ is an earlier entry than $a'$, then $D'$ is a shorter diagram with the same type of forbidden pair. 
    
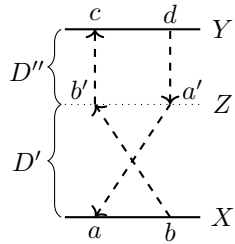
\begin{center}
        \begin{tikzpicture}
            \draw[-,thick] (-.4,0) -- (1.4,0);
            \draw[dotted] (-.4, 1.5) -- (1.4,1.5);
            \draw[-,thick] (-.4,2.5) -- (1.4,2.5);

            \draw (-.9,1.95) node{$D''$};
            \draw[decorate,decoration={brace,amplitude=5pt, raise=-1ex}](-.6,1.5) -- (-.6, 2.5);

            \draw[<-, thick, dashed] (0,0) -- (1,1.5);
            \draw[->, thick, dashed] (1,0) -- (0,1.5);
            \draw (0,-.2) node{$a$};
            \draw (1,-.2) node{$b$};
            \draw (-.2,1.7) node{$b'$};
            \draw (1.3,1.7) node{$a'$};

            \draw (-.9,.75) node{$D'$};
            \draw[decorate,decoration={brace,amplitude=5pt, raise=-1ex}](-.6,0) -- (-.6, 1.5);

            \draw[->,thick,dashed](0,1.5) -- (0,2.5);
            \draw[<-,thick,dashed](1,1.5)--(1,2.5);

            \draw (0,2.7) node{$c$};
            \draw (1,2.7) node{$d$};

            \draw(1.7,0) node{$X$};
            \draw(1.7,1.5) node{$Z$};
            \draw(1.7,2.5) node{$Y$};
        \end{tikzpicture}
    \end{center}

    If instead $a'$ is earlier than $b'$, then by inspecting the height 1 diagrams, we see that $D''$ can only be the cap $\epsilon_i$ connected to either $a'$ or $b'$. Without loss of generality, assume that $D''$ is the cap connecting $b'$ to its right neighbor $e$. Then $D''$ must pair $c$ to an entry $c'$ that is earlier than $a'$. Since $D''$ is height 1, we see $D'$ must pair $c'$ with $e$. But this creates a forbidden leftward cup in the shorter diagram $D'$. 
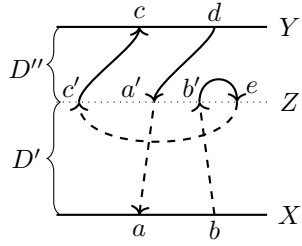
     \begin{center}
        \begin{tikzpicture}
            \draw[-,thick] (-1.1,0) -- (1.7,0);
            \draw[dotted] (-1.1, 1.5) -- (1.7,1.5);
            \draw[-,thick] (-1.1,2.5) -- (1.7,2.5);

            \draw (-1.5,1.95) node{$D''$};
            \draw[decorate,decoration={brace,amplitude=5pt, raise=-1ex}](-1.2,1.5) -- (-1.2, 2.5);

            \draw[<-, thick, dashed] (0,0) -- (.2,1.5);
            \draw[->, thick, dashed] (1,0) -- (.8,1.5);
            \draw (0,-.2) node{$a$};
            \draw (1,-.2) node{$b$};
            \draw (-.1,1.7) node{$a'$};
            \draw (.7,1.7) node{$b'$};
            \draw(1.5,1.7) node{$e$};

            \draw (-1.5,.75) node{$D'$};
            \draw[decorate,decoration={brace,amplitude=5pt, raise=-1ex}](-1.2,0) -- (-1.2, 1.5);

            \draw[<-,thick] (.2,1.5)..controls (.2,1.75) and (1,2.3)..(1,2.5);
            \draw[->,thick] (.8,1.5) ..controls (.8,1.9) and (1.3,1.9).. (1.3,1.5);
            \draw[->,thick](-.8,1.5)..controls (-.8,1.75) and (0,2.25) .. (0,2.5);

            \draw (0,2.7) node{$c$};
            \draw (1,2.7) node{$d$};
            \draw (-.9,1.7) node{$c'$};

            \draw(2,0) node{$X$};
            \draw(2,1.5) node{$Z$};
            \draw(2,2.5) node{$Y$};

            \draw [->,thick,dashed] (1.3,1.5) ..controls (1.3,.8) and (-.8,.8)..(-.8,1.5);
        \end{tikzpicture}
    \end{center}
    So, by induction, none of the three forbidden pairs can occur in an $\mathcal{A}_2$-diagram.
\end{proof}

We now begin to show that in a certain sense the $\mathcal{A}_2$-diagrams are determined by the induced pairing and by the number of dots on each braid. Certain topological defects which do not affect the induced pairing are not actually possible in an $\mathcal{A}_2$-diagram. Note again that the following are all possible in Khovanov, Lauda, and Rouquier's categorification of the idempotented quantum group.

\begin{lem}\label{lem:a2_self_int}
For any $X,Y\in \mathcal{A}_2(Q)$, there is no $\mathcal{A}_2$-diagram from $X$ to $Y$ in which some strand crosses itself.
\end{lem}
\begin{proof}
    Suppose towards contradiction that we have an $\mathcal{A}_2$-diagram $D$ that can be factorized as $D=D''\circ E_\textbf{i}T_{ij}E_\textbf{j}\circ D'$ in which the two strands in the indicated crossing are part of the same braid $f:[0,1]\rightarrow \R\times [0,1]$ in $D$. In particular, we must have $i=j$. Let $a,b,c,d$ be the smallest real numbers for which $f(t)$ is the upper left, resp. upper right, lower left, lower right endpoints of the crossing at times $t=a$, resp $t=b,c,d$. Due to the orientations and without loss of generality, we may assume that $a<d<b<c$. We also denote by $a,b,c,d$ the corresponding entries of the copies of $\textbf{i}ij\textbf{j}$ at the top and bottom of this crossing. Due to the ordering of $a,b,c,d$, we see that $D''$ must pair $b$ with some entry $b'$ of $\textbf{i}$ or $\textbf{j}$. It cannot be an entry of $\textbf{i}$ without creating a forbidden leftward cap as in Lemma \ref{lem:a2_forbidden_pairings}. Similarly, $D'$ must pair $d$ with some entry $d'$ of $\textbf{i}$. Since $b'\neq d'$, we must have that $D''$ pairs $d'$ with another entry of $\textbf{i}ij\textbf{j}$. Repeating the same argument, it is an entry $d''$ of $\textbf{i}$, and in particular $d''\neq b'$.
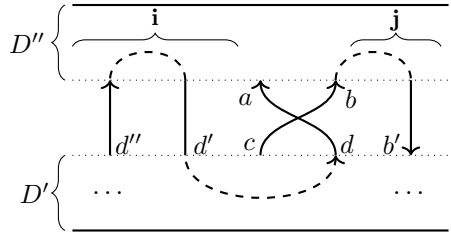
    \begin{center}
        \begin{tikzpicture}
            \draw[->,thick](0,0)..controls (0,.5) and (1,.5) .. (1,1);
            \draw[->,thick](1,0)..controls (1,.5) and (0,.5) .. (0,1);
            \draw[-,dotted] (-2.5,0) -- (2.5,0);
            \draw[-,dotted] (-2.5,1) -- (2.5,1);
            \draw[-,thick] (-2.5,2) --(2.5,2);
            \draw[-,thick] (-2.5,-1)--(2.5,-1);

            \draw (-.15,.15) node{$c$};
            \draw (1.15,.15) node{$d$};
            \draw (-.2, .75) node{$a$};
            \draw (1.2, .75) node{$b$};

            \draw[-,thick,dashed] (1,1) ..controls(1,1.5) and (2,1.5)..(2,1);
            \draw[->,thick] (2,1)--(2,0);
            \draw[<-,thick,dashed] (1,0)..controls(1,-.75) and (-1,-.75)..(-1,0);

            \draw (1.75,.15) node{$b'$};
            \draw (-.75,.15) node{$d'$};

            \draw[-,thick] (-1,0)--(-1,1);

            \draw[-,thick,dashed](-1,1)..controls(-1,1.5) and (-2,1.5)..(-2,1);
            \draw[<-,thick] (-2,1)--(-2,0);

            \draw (-1.75,.15) node{$d''$};
            \draw(-2,-.5) node{$\dots$};
            \draw(2,-.5) node{$\dots$};

            \draw (-3,-.5) node{$D'$};
            \draw[decorate,decoration={brace,amplitude=5pt, raise=-1ex}](-2.75,-1) -- (-2.75, 0);

            \draw (-3.1,1.5) node{$D''$};
            \draw[decorate,decoration={brace,amplitude=5pt, raise=-1ex}](-2.75,1) -- (-2.75, 2);

            \draw[decorate,decoration={brace,amplitude=5pt, raise=-1ex}](-2.5,1.6) -- (-.3, 1.6);
            \draw (-1.4,1.85) node{\small$\textbf{i}$};

            \draw[decorate,decoration={brace,amplitude=5pt, raise=-1ex}](1.2,1.6) -- (2.4, 1.6);
            \draw (1.8,1.8) node{\small$\textbf{j}$};

        \end{tikzpicture}
    \end{center}
    
    We note that $\textbf{i}$ has finitely many entries, and so by repeating this argument, we see that it is impossible that $b$ and $d$ are part of the same strand.
    
\end{proof}

\begin{lem}\label{lem:a2_no_loop}
     There are no $\mathcal{A}_2$-diagrams with closed loops. Moreover, $\text{dim}(\Hom_{\mathcal{A}_2(Q)}(\1,\1)) \leq 1$.
\end{lem}
\begin{proof}
    Any closed loop in an $\mathcal{A}_2$-diagram obtains a unique global maximum via some $W\epsilon_iW':WE_iF_iW'\rightarrow WW'$ and unique global minimum via some $V\eta_iV':VF_iE_iV'\rightarrow VV'$. The positive entry of this global maximum must connect to the positive entry of the global minimum, and likewise for the negative entries. This is impossible due to Lemma \ref{lem:a2_forbidden_pairings}. 
    
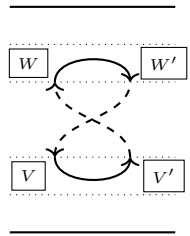
    \begin{center}
        \begin{tikzpicture}
            \draw[-,thick] (-.6,-1)--(1.6,-1);
            \draw[-,dotted] (-.6,-.5)--(1.6,-.5);
            \draw[-,dotted] (-.6,0)--(1.6,0);
            \draw[-,dotted] (-.6,1)--(1.6,1);
            \draw[-,dotted] (-.6,1.5)--(1.6,1.5);
            \draw[-,thick] (-.6,2)--(1.6,2);

            \draw[->,thick] (0,1)..controls (0,1.4) and (1,1.4)..(1,1);
            \draw[->,thick] (0,0)..controls (0,-.4) and (1,-.4)..(1,0);
            \draw[<-,thick,dashed] (0,0) ..controls (0,.5) and (1,.5)..(1,1);
            \draw[->,thick,dashed] (1,0) ..controls (1,.5) and (0,.5)..(0,1);

            \draw (-.35,1.25) node[draw]{\tiny$W$};
            \draw (1.45,1.25) node[draw]{\tiny$W'$};
            \draw (-.35,-.25) node[draw]{\tiny$V$};
            \draw (1.45,-.25) node[draw]{\tiny$V'$};
        \end{tikzpicture}
    \end{center}
    The second claim follows from the observation that any braid in an $\mathcal{A}_2$-diagram from the monoidal unit to itself is necessarily a closed loop.
    
\end{proof}
It is not obvious that the Hom spaces in the category $\mathcal{A}_2(Q)$ are nonzero and that $\text{dim}(\Hom_{\mathcal{A}_2(Q)}(\1,\1))=1$. In the graded case, we prove this later in Theorem \ref{thm:a2_2rep_faithful}. 

It is possible for there to be distinct $\mathcal{A}_2$-diagrams that yield the same pairing. We show that such diagrams are equal up to a linear combination of diagrams with fewer crossings. Morally, this is due to the KLR relations in our category. The formal proof requires tools from categorical logic and topology.

Results like the following lemma are very common for categories that are defined diagrammatically. See \cite{isotopy} for the general result in the necessary formalism. This result also shows that the precise spacing of tensor components in input and output objects is not important.

\begin{prop}\label{prop:a2_isotopy}
    Fix any two objects $X,Y\in \mathcal{A}_2(Q)$ and any two $\mathcal{A}_2$-diagrams from $X$ to $Y$. If these diagrams are equal up to an ambient boundary-preserving planar isotopy, then the associated homomorphisms in $\Hom_{\mathcal{A}_2(Q)}(X,Y)$ are equal.
\end{prop}

We may therefore think that the Hom spaces in $\mathcal{A}_2(Q)$ are spanned by ambient isotopy classes of $\mathcal{A}_2$-diagrams. Due to the KLR relations, we can refine this result.

We now show that we can essentially study $\mathcal{A}_2$-diagrams up to ``boundary-preserving dot-sliding homotopy". Similar arguments are made in \cite{kholau3}, and we adopt conventions from here. This method allows us to reduce our spanning sets even further. 

\begin{lem}\label{lem:a2_htpy}
    Let $D$ and $D'$ be two minimal $\mathcal{A}_2$-diagrams from $X$ to $Y$. Suppose that $D$ and $D'$ determine the same $(sseq(X),sseq(Y))$-pairing and have the same number of dots on each braid. Then the corresponding elements of $\Hom_{\mathcal{A}_2(Q)}(X,Y)$ for $D$ and $D'$ are equal up to a linear combination of minimal $\mathcal{A}_2$-diagrams with fewer crossings.
\end{lem}
We sketch a proof. Details are given in Section 8 of \cite{lausl2}. 
\begin{proof}
  First, apply the dot-sliding relations to $D$ so that all dots are at the bottom of their respective strands, i.e. the corresponding morphism occurs before any crossing, cap, or cup. Any error term introduced this way has fewer crossings. The data of the $(sseq(X),sseq(Y))$-pairing associated to $D$ determines which braids in $D$ must cross which others. So, the braids attached to specified endpoints in $D$ cross iff the corresponding braids cross in $D'$, although not necessarily in the same order. Proposition \ref{prop:a2_isotopy} and the KLR relations allow us to change the order of three-way crossings in $D$, possibly up to error terms with fewer crossings. Lemmas \ref{lem:a2_self_int} and \ref{lem:a2_no_loop} show that there are no other topological pathologies to our diagrams. So, applying the dot-slide relations again takes us to the diagram $D'$. An induction on the number of crossings then gives the argument.
\end{proof}

\begin{cor}\label{cor:a2_no_double}
    For any $X,Y\in \mathcal{A}_2(Q)$, we have that $\Hom_{\mathcal{A}_2(Q)}(X,Y)$ is spanned by \emph{minimal} $\mathcal{A}_2$-diagrams, i.e. those diagrams containing no pair of strands crossing each other more than once.
\end{cor}
\begin{proof}
    For any diagram in which two strands cross more than once, apply Proposition \ref{prop:a2_isotopy} and Lemma \ref{lem:a2_htpy} to obtain a diagram where the two strands crossing twice in a row, possibly along with a linear combination of diagrams with fewer crossings. The KLR relations allow us to equate the diagram with the double crossing to a linear combination of diagrams with fewer crossings. So, an induction on the number of crossings gives the claim.
\end{proof}

\begin{cor}\label{cor:a2_dot_htpy}
    For any $X,Y\in \mathcal{A}_2(Q)$, fix a minimal, dotless $\mathcal{A}_2$-diagram realizing each element of $p'(sseq(X),sseq(Y))$. For each chosen diagram $D$, fix an open interval on each braid away from any crossings, minima, or maxima. Let $S_D$ be the set of all $\mathcal{A}_2$-diagrams obtained by putting an arbitrary number of dots on each chosen interval of $D$. Then $B_{X,Y}\coloneqq \bigcup_D S_D$ is a spanning set for $\Hom_{\mathcal{A}_2(Q)}(X,Y)$.
\end{cor}

The set $B_{E_\textbf{i},E_\textbf{j}}$ depends on several choices, so we will only use the notation once these choices have been fixed. Depicted below is an example of a choice of $B_{E_iF_i,F_iE_i}$.

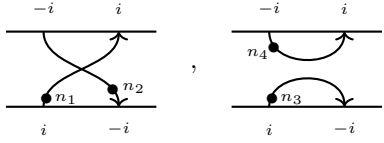
\begin{center}
   \begin{tikzpicture}
       \draw[->,thick] (0,0)..controls(0,.5) and (1,.5)..(1,1);
       \draw[->,thick] (0,1)..controls(0,.5) and (1,.5)..(1,0);
       \draw[-,thick] (-.5,0)--(1.5,0);
       \draw[-,thick] (-.5,1)--(1.5,1);
       \draw (0,-.3) node{\tiny$i$};
       \draw (1,-.3) node{\tiny$-i$};
       \draw (0,1.3) node{\tiny$-i$};
       \draw (1,1.3) node{\tiny$i$};
       \draw (.04,.1) node{$\bullet$};
       \draw (.3,.1) node{\tiny$n_1$};
       \draw (.92,.22) node{$\bullet$};
       \draw (1.2,.25) node{\tiny$n_2$};

       \draw (2,.5) node{$,$};

       \draw[->,thick] (3,0)..controls(3,.5) and (4,.5)..(4,0);
       \draw[->,thick] (3,1)..controls(3,.5) and (4,.5)..(4,1);
       \draw[-,thick] (2.5,0)--(4.5,0);
       \draw[-,thick] (2.5,1)--(4.5,1);
       \draw (3,-.3) node{\tiny$i$};
       \draw (4,-.3) node{\tiny$-i$};
       \draw (3,1.3) node{\tiny$-i$};
       \draw (4,1.3) node{\tiny$i$};
       \draw (3.04,.1) node{$\bullet$};
       \draw (3.3,.1) node{\tiny$n_3$};
       \draw (3.06,.78) node{$\bullet$};
       \draw (2.85,.7) node{\tiny$n_4$};
   \end{tikzpicture}
\end{center}

In this example, all $n_1,n_2,n_3,n_4\in \N$ are allowed. 

After fixing an order on the braids in each minimal diagram $D$, we obtain a bijection from $B_{E_\textbf{i},E_\textbf{j}}$ to $p'(\textbf{i},\textbf{j})\times \N^{(|\textbf{ij}|)/2}$. We will use the notation $d(\textbf{i},\textbf{j})$ for the latter set, and refer to its elements as \emph{dotted pairings} between $\textbf{i}$ and $\textbf{j}$.

\subsection{Gradings}\label{subsec:2rep}
We make small adjustments in the case that our category can be graded. 

The previous results all restrict the kinds of morphisms that can appear in $\mathcal{A}_2(Q)$. We can use these to obtain much stronger results in the case that $Q$ comes from a symmetrizable generalized Cartan matrix. For the rest of this article, we assume that a symmetrizable generalized Cartan matrix $C$ is given and that, after some choice of scalars, the matrix $Q$ is determined from $C$ as described in Subsection \ref{subsec:klr}.

Due to this assumption, we see that this category may be enriched in graded $\kf$-vector spaces via $\text{deg}(\text{id}_{E_i})=\text{deg}(\text{id}_{F_i})=0$, $\text{deg}(X_i)=2d_i$, $\text{deg}(T_{ij})=-d_iC_{ij}$, and $\text{deg}(\epsilon_i)=\text{deg}(\eta_i)=0$. For a morphism $f$, we will also use the notation $f$ to refer to any shifts $q^nf$. Note that the degrees of the unit and counit differ from those for the full quantum group in \cite{qha}\cite{kholau3}.
\begin{defn}
    Define $\mathcal{A}_2(C)$ to be the monoidal category whose objects are formal shifts $q^n X$ of those in $\mathcal{A}_2(Q)$ and whose morphisms are defined by $\Hom_{\mathcal{A}_2(C)}(q^n X,q^m Y)=\Hom_{\mathcal{A}_2(Q)}^\bullet(X,Y)_{n-m}$. We require the shift $q$ to commute with the monoidal product. For objects $X$ and $Y$ of $\mathcal{A}_2(C)$, we define a $\mathcal{A}_2$-diagram from $X$ to $Y$ in exactly the same way as before, although we now require that our morphisms have the correct grading. By a small abuse of notation, we define the function $sseq$ on objects of $\mathcal{A}_2(C)$ as before, where we now also require $sseq(qX)=sseq(X)$. We  refer to tensor components of $X$ of the form $E_i$ or $F_j$ as entries of $X$ or of $sseq(X)$.
    
    Note that any spanning set $B_{X,Y}$ from Corollary \ref{cor:a2_dot_htpy} consists of homogeneous elements, and that the degree of a $B_{X,Y}$ element does not depend on the choice of minimal diagram $D$  or on the specific intervals on each braid for the dots. So, this corollary also applies to $\mathcal{A}_2(C)$, where we only consider dotted pairings of the correct grading. We therefore use the notation $B_{X,Y}$ in the same way for objects $X,Y\in \mathcal{A}_2(C)$, keeping only diagrams of the correct grading. We also use the notation $d_{n}(\textbf{i},\textbf{j})$ to denote the subset of $d(\textbf{i},\textbf{j})$ whose associated diagrams from $E_{\textbf{i}}$ to $E_{\textbf{j}}$ have degree $n$. There is an induced bijection of $d_n(\textbf{i},\textbf{j})$ with $B_{q^{m+n}E_{\textbf{i}},q^mE_{\textbf{j}}}$.
\end{defn} 

The symmetry $\phi$ is well-defined on $\mathcal{A}_2(C)$ since it preserves degrees of diagrams. Here, $\phi(q^mE_{\textbf{i}})=q^mE_{-\overline{\textbf{i}}}$.

The 2-representation of Proposition \ref{prop:a2_2rep_exists} is compatible with the grading. We can therefore convert it into a 2-representation of $\mathcal{A}_2(C)$ on $\bigoplus_\alpha H_\alpha(C)-\text{grMod}$, and we do so for the rest of the paper.

We will not consider the ungraded $\mathcal{A}_2(Q)$ in the rest of the paper. 

\begin{lem}\label{lem:a2_lblfd}
For any $X,Y\in \mathcal{A}_2(C)$, the graded Hom space $\Hom_{\mathcal{A}_2(C)}^\bullet(X,Y)$ is left-bounded and locally finite-dimensional.
\end{lem}

\begin{proof}
    By Corollary \ref{cor:a2_dot_htpy}, this vector space contains a homogeneous spanning set that has a fixed bijection with $d(sseq(X),sseq(Y))$. We need only note that $p'(sseq(X),sseq(Y))$ is finite and that adding dots to any braid increases the grading.
\end{proof}

We will eventually show that the spanning sets $B_{X,Y}$ described in Corollary \ref{cor:a2_dot_htpy} actually constitute a basis. We  do so by proving that the 2-representation of Proposition \ref{prop:a2_2rep_exists} is faithful.

The action of the induction and restriction functors on $\bigoplus_\alpha H_\alpha(C)-\text{grMod}$ is also studied in \cite{kk}. They prove the following important result which is a categorification of the quantum boson relations.
\begin{prop}[\cite{kk} Corollary 3.4]\label{prop:kk}
    For any $\alpha \in \N[I]$ and $i,j\in I$, there is an isomorphism of graded $(H_{\alpha-\alpha_j+\alpha_i},H_\alpha)$-bimodules

    \[q_i^{-C_{ij}}H_{\alpha-\alpha_j+\alpha_i}1_{*i}\otimes_{\alpha-\alpha_j}1_{*j}H_\alpha \oplus \delta_{ij}H_{\alpha}[x_{|\alpha|+1}]\simeq 1_{*j}H_{\alpha+\alpha_i}1_{*i}\]
    given by
    \[(a\otimes_{\alpha-\alpha_j}b,\delta_{ij}c)\rightarrow a\tau_{-1}b + \delta_{ij}c.\]
    Here, we have used the left embeddings of KLR algebras, i.e. $(1_{\textbf{i}i}\otimes_{\alpha-\alpha_j}1_{\textbf{i}j},1_{\textbf{j}})\rightarrow 1_{\textbf{i}ij}\tau_{-1}1_{\textbf{i}ji}+\delta_{ij}1_{\textbf{j}i}c$. In the 2-representation of $\mathcal{A}_2(C)$, there is an isomorphism of functors $q_i^{-C_{ij}}E_iF_j\oplus \delta_{ij}\bigoplus_{n\in \N} q_i^n \1\simeq F_jE_i$ given by $\sigma_{ij} \oplus \delta_{ij}\bigoplus_{n\in \N} F_iX_i^n\circ\eta_i$. This isomorphism is expressed diagrammatically below.

    
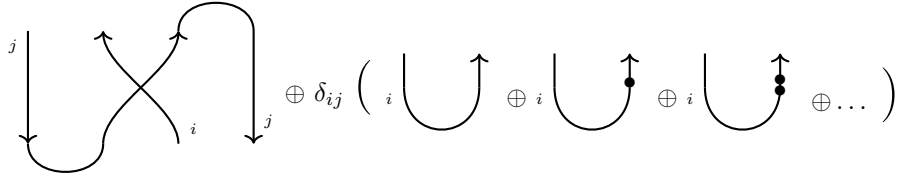
\begin{center}
            \begin{tikzpicture}

                \draw[-, thick] (3,0).. controls ++(0,-.5) and ++(0,-.5) .. (4,0);
        
                \draw[->, thick] (4,0) .. controls (4,.5) and (5,1) .. (5,1.5);
                \draw[->, thick] (5,0) .. controls (5,.5) and (4,1) .. (4,1.5);

                \draw[-, thick] (5,1.5).. controls ++(0,.5) and ++(0,.5) .. (6,1.5);

                \draw[->, thick] (3,1.5) to (3,0);
                \draw[->, thick] (6,1.5) to (6,0);

                \draw (2.8,1.3) node{\tiny $j$};
                \draw (6.2,.3) node{\tiny $j$};
                \draw (5.2,.2) node{\tiny $i$};

                \draw (7,.75)
                node{$\oplus$ $ \delta_{ij}$ \Huge $($};

                \draw (7.8, .4) node[above]{\tiny $i$};
                \draw[-, thick]
                (8,1.2) to (8,.75);
                \draw[<-, thick]
                (9,1.2) to (9,.75);
                \draw[-, thick] (8,.75).. controls ++(0,-.75) and ++(0, -.75) .. (9,.75);

                \draw (9.5,.6)
                node{$\oplus$ };

                \draw (9.8, .4) node[above]{\tiny $i$};
                \draw[-, thick] (10,.75).. controls ++(0,-.75) and ++(0, -.75) .. (11,.75);
                 \draw[-, thick]
                (10,1.2) to (10,.75);
                \draw[<-, thick]
                (11,1.2) to (11,.75);
                \draw (11, .8) node{$\bullet$};

                \draw (11.5,.6)
                node{$\oplus$ };

                \draw (11.8, .4) node[above]{\tiny $i$};
                \draw[-, thick] (12,.75).. controls ++(0,-.75) and ++(0, -.75) .. (13,.75);
                 \draw[-, thick]
                (12,1.2) to (12,.75);
                \draw[<-, thick]
                (13,1.2) to (13,.75);
                \draw (13, .85) node{$\bullet$};
                \draw (13, .7) node{$\bullet$};

                \draw (14,.65)
                node{$\oplus\dots$ \Huge $)$};
                
            \end{tikzpicture}
        \end{center}
\end{prop}
The observation that the formula for the isomorphism can be given in terms of only $X$, $T$, and $\eta$ is, to our knowledge, new. Note that very similar formulas are formally inverted in Rouquier's definitions of 2-Kac-Moody algebras \cite{qha}.

We state a technical result that we will need for the next theorem.
\begin{prop}[\cite{kholau2} Proposition 2.16]
    Fix $\textbf{j}$ and $\textbf{j}' \in $ Seq with $\alpha_{\textbf{j}} \geq \alpha_{\textbf{j}'}$. Then $1_{*\textbf{j}'}H_{\alpha_{\textbf{j}}}$ is free as a left graded $H_{\alpha_{\textbf{j}}-\alpha_{\textbf{j}'}}$-module. A basis can be obtained as follows. For each left coset for $S_{|\textbf{j}'|}$ as a subgroup of $S_{|\textbf{j}|}$, there is a unique minimal representative $\omega$ such that $\omega^{-1}$ does not swap the symbols $1$ through $|\textbf{j}|-|\textbf{j}'|$ amongst themselves. Pick a reduced presentation for each such $\omega$. Denote by $\tau_{\omega}$ the element of $H_{\alpha_{\textbf{j}}}$ associated to this reduced presentation. Then the set $\{ 1_{\textbf{i}\textbf{j}'}p(x_{-|\textbf{j}'|},\dots x_{|\textbf{j}|})\tau_{\omega}\}$ ranging over all $\textbf{i}\in \text{Seq}_{\alpha_{\textbf{j}}-\alpha_{\textbf{j}'}}$, all coset representatives $\omega$, and all monomials $p(x_{-|\textbf{j}'|},\dots x_{|\textbf{j}|})$ in the rightmost $|\textbf{j}'|$ variables constitutes a basis.
\end{prop}

Computations similar to the following corollary appear in Section 3 of \cite{kk}.
\begin{cor}\label{cor:kholau2_ef_basis}
    Fix $\textbf{i},\textbf{j}$, and $\textbf{j}'\in$ Seq with $\alpha_\textbf{j} \geq \alpha_{\textbf{j}'}$. Then $H_{\alpha_{\textbf{j}}-\alpha_{\textbf{j}'}+\alpha_{\textbf{i}}}1_{*\textbf{i}}\otimes_{\alpha_{\textbf{j}}-\alpha_{\textbf{j}'}} 1_{*\textbf{j}'}H_{\alpha_j}$ has a $\kf$-basis that can be obtained as follows. For each left coset for $S_{|\textbf{j}'|}$ as a subgroup of $S_{|\textbf{j}|}$, there is a unique minimal representative $\omega$ that does not swap the symbols $1$ through $|\textbf{j}|-|\textbf{j}'|$ amongst themselves. Pick a reduced presentation for each such $\omega$. Denote by $\tau_{\omega}$ the element of $H_{\alpha_{\textbf{j}}}$ associated to this reduced presentation. Similarly, for each $\sigma \in S_{|\textbf{j}|-|\textbf{j}'|+|\textbf{i}|}$, pick a reduced presentation, and denote by $\tau_{\sigma}$ the element of $H_{\alpha_{\textbf{j}}-\alpha_{\textbf{j}'}+\alpha_{\textbf{i}}}$ associated to this presentation. Then the set $\{\tau_{\sigma}r(x_{1},\dots x_{-1})1_{\textbf{ki}}\otimes 1_{\textbf{k}'\textbf{j}'}p(x_{-|\textbf{j}'|},\dots x_{|\textbf{j}|})\tau_{\omega}\}$ ranging over all $\textbf{k},\textbf{k}'\in \text{Seq}_{\alpha_{\textbf{j}}-\alpha_{\textbf{j}'}}$, all coset representatives $\omega$, all $S_{|\textbf{j}|-|\textbf{j}'|+|\textbf{i}|}$ elements $\sigma$, all monomials $r(x_1,\dots x_{-1})$ in the variables of $H_{\alpha_{\textbf{j}}-\alpha_{\textbf{j}'}+\alpha_{\textbf{i}}}$, and all monomials $p(x_{-|\textbf{j}'|},\dots x_{|\textbf{j}|})$ in the rightmost $|\textbf{j}'|$ variables of $H_{\alpha_j}$ constitutes a basis.
\end{cor}

\subsection{Faithfulness of the 2-representation}\label{subsec:faithful}
We prove a few lemmas about formal direct sums of $\mathcal{A}_2$ objects. We use these to show that the 2-representation of Proposition \ref{prop:a2_2rep_exists} is faithful. 

Proposition \ref{prop:kk} suggests that the morphism $\sigma_{ij} \oplus \delta_{ij}\bigoplus_{n\in \N} F_iX_i^n\circ \eta_i$ should have interesting properties. We must first formally expand the category $\mathcal{A}_2(C)$ to study this morphism.

\begin{defn}
     Denote by $\text{Pre}_2(C)\coloneqq \text{Func}(\mathcal{A}_2(C)^{op},\kf-Vec)$ the category of presheaves on $\mathcal{A}_2(C)^{op}$ valued in $\kf$-vector spaces.
\end{defn} 

Recall the fully faithful Yoneda embedding $\mathcal{Y}:\mathcal{A}_2(C)\hookrightarrow \text{Pre}_2(C)$ given by $X\rightarrow \Hom_{\mathcal{A}_2(C)}(-, X)$ and obvious effect on morphisms. The category $\text{Pre}_2(C)$ is complete and cocomplete and has a strict $\Z$-grading given by $qF(X)=F(q^{-1}X)$ with appropriate shifts on natural transformations. Note that with this convention, the Yoneda embedding commutes with shifting, i.e. $\mathcal{Y}(qX)=q\mathcal{Y}(X)$.

Certain coproducts are especially common in our studies. We use a few notions from \cite{nava} to describe these coproducts.

\begin{defn}
    We say that a coproduct in $\text{Pre}_2(C)$ is \emph{locally finite} if it has the form
    \[\coprod_{i\in \Z} q^i(F_1^{\oplus k_{1,i}}\oplus F_2^{\oplus k_{2,i}}\dots \oplus F_n^{\oplus k_{n,i}})\]

    for some $F_j\in \text{Pre}_2(C)$ and $k_{j,i}\in \N$. We say also that the coproduct is \emph{left-bounded} if there exists some $m \in \Z$ for which $k_{j,i} = 0$ for all $j$ whenever $i < m$.
\end{defn}

\begin{defn}
    Let $\{X_i\}_{i\in S}$ be a set of objects in a $\kf$-linear category $\mathcal{C}$. Suppose that the coproduct $\coprod_{i\in S} X_i$ and the product $\prod_{i\in S} X_i$ exist in $\mathcal{C}$. There is a canonical morphism $\psi:\coprod_i X_i \rightarrow \prod_i X_i$ corresponding to the element of $\prod_i \prod_j \Hom_\mathcal{C}(X_i,X_j)$ that is the identity in each $\Hom_{\mathcal{C}}(X_i,X_i)$ and zero elsewhere. We say that this coproduct (and also this product) is a \emph{biproduct} if $\psi$ is an isomorphism. We denote biproducts with $\bigoplus_i$ instead of just $\prod_i$ or $\coprod_i$.
\end{defn}

\begin{lem}\label{lem:a2_biprod}
    In $\text{Pre}_2(C)$, all left-bounded and locally finite coproducts of objects in $\mathcal{A}_2(C)$ are biproducts.
\end{lem}
\begin{proof}
Denote by $\coprod_{i\in \Z} q^i(F_1^{\oplus k_{1,i}}\oplus F_2^{\oplus k_{2,i}}\dots \oplus F_n^{\oplus k_{n,i}})$ our left-bounded locally finite coproduct of $\mathcal{A}_2(C)$ objects. It is enough to show that the natural transformation
\[\psi:\coprod_{i\in \Z} q^i(F_1^{\oplus k_{1,i}}\oplus F_2^{\oplus k_{2,i}}\dots \oplus F_n^{\oplus k_{n,i}})\rightarrow \prod_{i\in \Z} q^i(F_1^{\oplus k_{1,i}}\oplus F_2^{\oplus k_{2,i}}\dots \oplus F_n^{\oplus k_{n,i}})\]
is object-wise an isomorphism. For any $X\in \mathcal{A}_2(C)$, only finitely many of the $\Hom_{\mathcal{A}_2(C)}(X,q^iF_j^{\oplus k_{j,i}})$ can be nonzero due to Corollary \ref{lem:a2_lblfd}. So, when we evaluate $\psi$ at the object $X$, we see that $\psi$ induces the identity map on the finite biproduct $\bigoplus_{i\in \Z} q^i(\Hom(X,F_1^{\oplus k_{1,i}})\oplus \Hom(X,F_2^{\oplus k_{2,i}})\dots \oplus \Hom(X,F_n^{\oplus k_{n,i}}))$. This gives the claim.

\end{proof}

\begin{defn}
    We say that a left-bounded locally finite biproduct \[\bigoplus_{i\in \Z} q^i(F_1^{\oplus k_{1,i}}\oplus \dots \oplus F_n^{\oplus k_{n,i}})\] is a \emph{\textbf{B}-biproduct} if each series $\sum_{i\in \Z}k_{m,i}q^i$ is contained in the subset of $\N[q,q^{-1}][[q]]$ generated under addition, multiplication, and shifts by the elements $1$ and the $1/(1-q_i^2)\coloneqq 1+q_i^2+q_i^{4}+\dots$ for each $i\in I$. We similarly define $\textbf{B}$-coproducts.
\end{defn}
\begin{defn}
    Denote by ${_\textbf{B}\mathcal{A}_2}(C)$ the full subcategory of $\text{Pre}_2(C)$ containing all $\textbf{B}$-biproducts of $\mathcal{A}_2(C)$ objects. This category is evidently closed under shifts and taking $\textbf{B}$-biproducts. Moreover, it inherits a monoidal product, with the tensor distributing over morphisms between such biproducts in the obvious way. The 2-representation of $\mathcal{A}_2(C)$ on $\bigoplus_\alpha H_\alpha(C)-\text{grMod}$ extends to a 2-representation of ${_\textbf{B}\mathcal{A}_2(C)}$. The symmetry $\phi$ extends to $_\textbf{B} \mathcal{A}_2(C)$ since it preserves shifts. 
\end{defn}

\begin{defn}We say that an object $X$ of $\mathcal{A}_2(C)$ is \emph{reduced} if it has the form $X=q^mE_{\textbf{i}}F_{\textbf{j}}$ for some $m\in \Z$ and $\textbf{i},\textbf{j}\in$ Seq. 
\end{defn}

\begin{defn}
    For $i,j\in I$, denote by $\rho_{ij}$ the ${_\textbf{B}\mathcal{A}_2(C)}$-morphism 
    \[\sigma_{ij} \oplus \delta_{ij}\bigoplus_{n\in \N} F_iX_i^n\circ\eta_i:q_i^{-C_{ij}}E_iF_j\oplus \delta_{ij}\bigoplus_{n\in \N} q_i^n \1\rightarrow F_jE_i. \]
    We denote by $R_{ij}$ the domain object of $\rho_{ij}$.
\end{defn}

\begin{lem}\label{lem:rho_onto}
    For any reduced object $X$ and any object $YF_jE_iZ$ of $\mathcal{A}_2(C)$, any $\mathcal{A}_2(C)$ morphism from $X$ to $YF_jE_iZ$ factors through $Y\rho_{ij}Z$ in ${_\textbf{B}\mathcal{A}_2(C)}$. Moreover, $\Hom_{\mathcal{A}_2(C)}(X,YF_jE_iZ)$ has a spanning set of minimal diagrams that further factor through one of the given summands of $Y(q_i^{-C_{ij}}E_iF_j\oplus \delta_{ij}\bigoplus_{n\in \N}q_i^n \1)Z$ by a minimal diagram.
\end{lem}

\begin{proof}
    By linearity, it is enough to prove the claim for any fixed choice of $B_{X,YF_jE_iZ}$. Since $X$ is reduced, in any $\mathcal{A}_2$-diagram from $X$ to $YF_jE_iZ$, the braids connecting to the upper $F_j$ and to the upper $E_i$ either coincide or cross. All other possibilities lead to a local minimum with the wrong orientation, contradicting Lemma \ref{lem:a2_forbidden_pairings}. We now fix a choice of minimal dotless diagrams for each element of $p'(sseq(X),sseq(YF_jE_iZ))$. For those diagrams with this $F_j$ and $E_i$ crossing, we pick a minimal diagram in which this crossing $Y\sigma_{ij}Z$ is the last operation applied in our diagram. If these $F_j$ and $E_i$ entries connect (and so $i=j$), we pick a minimal diagram in which the last operation applied is the cup $Y\eta_iZ$. So, for each element of $d(sseq(X),sseq(YF_jE_iZ))$, we pick the corresponding entry of our spanning set $B_{X,YF_jE_iZ}$ based on this choice of dotless diagrams. For those dotless diagrams ending in $Y\eta_i Z$, we postcompose with the appropriate number of $YF_iX_iZ$. For all other braids, make any choice of dot placement for which the dots are below the given cups and crossings. We depict this choice of $B_{X,YF_jE_iZ}$ elements below.

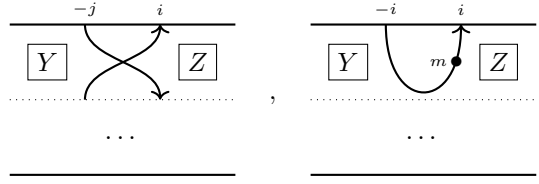
\begin{center}
    \begin{tikzpicture}
        \draw[-,thick] (-1,0)--(2,0);
        \draw[-,dotted] (-1,1)--(2,1);
        \draw[-,thick] (-1,2)--(2,2);
        \draw[->,thick] (0,1)..controls (0,1.5) and (1,1.5)..(1,2);
        \draw[<-,thick] (1,1)..controls (1,1.5) and (0,1.5)..(0,2);
        \draw (-.5,1.5) node[draw]{$Y$};
        \draw (1.5,1.5) node[draw]{$Z$};
        \draw (.5,.5) node{$\dots$};
        \draw (0,2.2) node{\tiny$-j$};
        \draw (1,2.2) node{\tiny$i$};
        \draw (4,2.2) node{\tiny$-i$};
        \draw (5,2.2) node{\tiny$i$};
        
        \draw (2.5,1) node{$,$};

        \draw[-,thick] (3,0)--(6,0);
        \draw[-,dotted] (3,1)--(6,1);
        \draw[-,thick] (3,2)--(6,2);
        \draw[->,thick] (4,2)..controls (4,.8) and (5,.8)..(5,2);
        \draw (3.5,1.5) node[draw]{$Y$};
        \draw (5.5,1.5) node[draw]{$Z$};
        \draw (4.5,.5) node{$\dots$};
        \draw (4.7,1.5) node{\tiny$m$};
        \draw (4.94,1.5) node{$\bullet$};
    \end{tikzpicture}
\end{center}
    
    By construction, each diagram in $B_{X,YF_jE_iZ}$ factors through $Y\rho_{ij}Z\circ YfZ$, where $f$ is one of the direct summand inclusions for the given summands of $Y(q_i^{-C_{ij}}E_iF_j\oplus \delta_{ij}\bigoplus_{n\in \N}q_i^n \1)Z$.
\end{proof}

\begin{thm}\label{thm:a2_2rep_faithful}
For any objects $X,Y\in \mathcal{A}_2(C)$, the spanning sets $B_{X,Y}$ are bases for $\Hom_{\mathcal{A}_2(C)}(X,Y)$. The 2-representation of Proposition \ref{prop:a2_2rep_exists} is faithful.
\end{thm}

\begin{rem}
    The 2-representation of Proposition \ref{prop:a2_2rep_exists} is not full. Lemma \ref{lem:a2_no_loop} shows that there are no interesting endomorphisms of $\1$, i.e. $\text{grdim(Hom}^\bullet_{\mathcal{A}_2(C)}(\1,\1))$ is either $1$ or $0$. However, the identity functor on $\bigoplus_\alpha H_\alpha(C)-\text{grMod}$ has many nontrivial endomorphisms. For example, the identity functor on $H_\alpha(C)-\text{grMod}$ has a nontrivial endomorphism given by multiplication by any element in the center of $H_\alpha(C)$. This center consists of certain symmetric polynomials, see Theorem 2.9 of \cite{khla}.
\end{rem}
\begin{proof}
     We prove both claims by showing that the image of such a $B_{X,Y}$ under our 2-representation is a set of linearly independent natural transformations. Note that by Lemma \ref{lem:a2_htpy} and Lemma \ref{lem:a2_lblfd}, for any objects $X,Y\in \mathcal{A}_2(C)$, two different choices of $B_{X,Y}$ are related by an application of a finite-dimensional unit upper triangular matrix. So, one is linearly independent iff the other is. It is therefore sufficient to prove linear independence of any fixed choice of $B_{X,Y}$.

    We first prove the theorem under the assumption that $X=q^mE_{\textbf{i}}F_{\textbf{j}}$ and $Y=q^n E_{\textbf{i}'}F_{\textbf{j}'}$ for some $m,n\in \Z$ and $\textbf{i}$, $\textbf{i}'$, $\textbf{j}$, $\textbf{j}'\in $ Seq. By Lemma \ref{lem:a2_forbidden_pairings}, in any $\mathcal{A}_2$-diagram from $X$ to $Y$, the braid from an entry in $E_{\textbf{i}'}$ cannot connect to an entry in $F_{\textbf{j}'}$, or else it creates a local minimum with the wrong orientation. So, each entry in $E_{\textbf{i}'}$ must connect to one in $E_{\textbf{i}}$, and similar for the $F_{\textbf{j'}}$ and $F_{\textbf{j}}$. In particular, $|\textbf{i}'|+|\textbf{j}'|\leq |\textbf{i}|+|\textbf{j}|$. Both $X$ and $Y$ yield functors on $\bigoplus_{\alpha} H_{\alpha}(C)-\text{grMod}$, which we also denote by $X$ and $Y$. We evaluate both functors on the object $H_{\alpha_{\textbf{j}}}\in H_{\alpha_{\textbf{j}}}-\text{grMod}$. Note that as elements of $\Z I$, we have that $\alpha_{\textbf{i}}-\alpha_{\textbf{j}}=\alpha_{\textbf{i}'}-\alpha_{\textbf{j}'}$. We can therefore view both functors $X$ and $Y$ as tensoring by some graded $(H_{\alpha_{\textbf{i}}},H_{\alpha_{\textbf{j}}})$-bimodule, and any element of $\Hom_{\mathcal{A}_2(C)}(X,Y)$ determines a graded bimodule homomorphism. We compute that $X(H_{\alpha_{\textbf{j}}})=q^mH_{\alpha_{\textbf{i}}}1_{\bar{\textbf{i}}}\otimes_{\kf} 1_{\textbf{j}}H_{\alpha_{\textbf{j}}}$ and $Y(H_{\alpha_{\textbf{j}}})=q^n H_{\alpha_{\textbf{i}}}1_{*\bar{\textbf{i}'}}\otimes_{\alpha_{\textbf{j}}-\alpha_{\textbf{j}'}} 1_{*\textbf{j}'}H_{\alpha_{\textbf{j}}}$. We evaluate the bimodule homomorphisms associated to a chosen basis of $\Hom_{\mathcal{A}_2(C)}(X,Y)$ on $1_{\bar{\textbf{i}}}\otimes 1_{\textbf{j}}$ and show that the corresponding elements of $Y(H_{\alpha_{\textbf{j}}})$ are linearly independent over $\kf$.

    We now fix a specific spanning set $B_{X,Y}$ for $\Hom_{\mathcal{A}_2(C)}(X,Y)$. We observe that for each pairing in $p'(sseq(X),sseq(Y))$, and for any $\mathcal{A}_2$-diagram realizing this pairing, each braid either connects an $E_i$ in $X$ to an $E_i$ in $Y$, connects an $F_i$ in $X$ to an $F_i$ in $Y$, or connects an $E_i$ in $X$ to an $F_i$ in $X$. For each element $M$ of $d_{m-n}(sseq(X),sseq(Y))$, we pick the diagram that is constructed as follows.
    \begin{enumerate}
        
        \item Denote by $\mathbf{l}$ the (non-consecutive) subsequence of $\textbf{j}$ consisting of the entries that are paired with an entry of $\textbf{i}$. Our $d_{m-n}(\textbf{i}\textbf{j},\textbf{i}'\textbf{j}')$-element $M$ induces a unique element $M'$ of $p'(\textbf{i},\textbf{i}'\overline{\textbf{l}})$. There is a unique element $\sigma$ of the symmetric group $S_{|\textbf{i}|}$ corresponding to $M'$. Pick any reduced presentation $\sigma=s_{i_k}s_{i_{k-1}}\dots s_{i_1}$ and apply the morphism $T_{\sigma}F_{\textbf{j}}$ associated to this choice.
        
        \item There is a unique element $M''$ of $p'(\sigma(\textbf{i})\textbf{j},\textbf{i}'\textbf{j}')$ that yields $M$ after composition with $M'$ and an appropriate dot placement. Denote by $\textbf{l}'$ the (non-consecutive) subsequence of $\sigma(\textbf{i})$ consisting of the entries that are paired with an entry of $\textbf{j}$ by $M''$. As above, we obtain an element of $p'(\textbf{j},\overline{\textbf{l}}'\textbf{j}')$, and thus, an element $\omega$ of the symmetric group $S_{|\textbf{j}|}$ mapping $\textbf{j}$ to $\bar{\textbf{l}}'\textbf{j}'$. Pick any reduced presentation $\omega=s_{j_l}\dots s_{j_1}$ and apply the morphism $E_{\sigma(\textbf{i})}T_{\omega}$ associated to this choice.

         \item Now, place all dots at the top of each strand. On those braids that connect two entries of $X$, we choose to place the dots on the $E$-side. Denote by $r$ the monomial for the dots placed on each $E$ strand in $X$ and by $p$ the monomial for the dots placed on each $F$ strand in $X$ that is not connecting two entries of $X$.
        
        \item Finally, we need only apply a sequence of $\epsilon_i$ to close off the braids connecting two entries of $X$. Since all crossings have been performed and since each positive entry of $X$ is to the left of each negative entry of $X$, there is a unique order in which to apply the $\epsilon_i$, i.e., the local maxima of these braids are nested.
    \end{enumerate}
    This construction is depicted below.
    
    
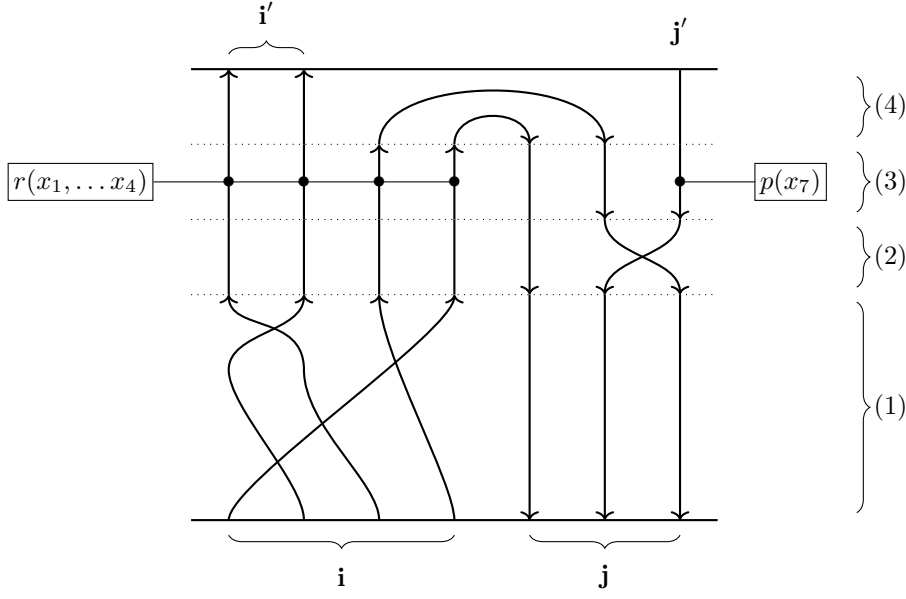
\begin{center}
        \begin{tikzpicture}
            
            \draw[-,thick] (-.5,7)--(6.5,7);
            \draw[-,thick] (-.5,1)--(6.5,1);
            \draw[-,dotted] (-.5,4)--(6.5,4);
            \draw[-,dotted] (-.5,5)--(6.5,5);
            \draw[-,dotted] (-.5,6)--(6.5,6);
            \draw[-,dotted] (-.5,7)--(6.5,7);

            \draw[->,thick] (0,6)--(0,7);
            \draw[->,thick] (1,6)--(1,7);
            \draw[->,thick] (2,5)--(2,6);
            \draw[->,thick] (3,5)--(3,6);
            \draw[-,thick] (4,5)--(4,6);
            \draw[<-,thick] (5,5)--(5,6);
            \draw[-,thick] (6,6)--(6,7);

            \draw(0,5.5) node{$\bullet$};
            \draw(1,5.5) node{$\bullet$};
            \draw(2,5.5) node{$\bullet$};
            \draw(3,5.5) node{$\bullet$};
            \draw(6,5.5) node{$\bullet$};

            \draw[-] (3,5.5)--(-1,5.5);
            \draw (-1.97,5.5) node [inner sep=2pt, draw] {$r(x_1,\dots x_4)$};
            \draw[-] (6,5.5)--(7,5.5);
            \draw (7.47,5.5) node [inner sep=2pt, draw] {$p(x_7)$};

            \draw[->,thick] (0,1)..controls (0,1.5) and (3,3.5) .. (3,4);
            \draw[->,thick] (1,1)..controls (1,1.5) and (0,2.5)..(0,3)..controls (0,3.5) and (1,3.5)..(1,4);
            \draw[->,thick] (2,1)..controls (2,1.5) and (1,2.25)..(1,3)..controls(1,3.75) and (0,3.5)..(0,4);
            \draw[->,thick]
            (3,1)..controls(3,1.5) and (2,3.25)..(2,4);
            \draw[<-,thick] (4,1)--(4,4);
            \draw[<-,thick] (5,1)--(5,4);
            \draw[<-,thick] (6,1)--(6,4);

            \draw[<-,thick] (4,4)--(4,5);
            \draw[<-,thick] (5,4)..controls (5,4.5) and (6,4.5)..(6,5);
            \draw[<-,thick] (6,4)..controls (6,4.5) and (5,4.5)..(5,5);
            \draw[-,thick] (0,4)--(0,5);
            \draw[-,thick] (1,4)--(1,5);
            \draw[-,thick] (2,4)--(2,5);
            \draw[-,thick] (3,4)--(3,5);

            \draw[-,thick] (0,5)--(0,6);
            \draw[-,thick] (1,5)--(1,6);
            \draw[->,thick] (2,6)..controls (2,6.95) and (5,6.95)..(5,6);
            \draw[->,thick] (3,6)..controls (3,6.5) and (4,6.5)..(4,6);
            \draw[<-,thick] (6,5)--(6,6);

            \draw[decorate,decoration={brace,amplitude=5pt,mirror, raise=-1ex}](0,.65) -- (3, .65);
            \draw (1.5, .25) node{$\textbf{i}$};

            \draw[decorate,decoration={brace,amplitude=5pt,mirror, raise=-1ex}](4,.65) -- (6, .65);
            \draw (5, .25) node{$\textbf{j}$};

            \draw[decorate,decoration={brace,amplitude=5pt, raise=-1ex}](0,7.35) -- (1, 7.35);
            \draw (.5, 7.75) node{$\textbf{i}'$};
            \draw(6,7.5) node{$\textbf{j}'$};

            \draw[decorate,decoration={brace,amplitude=5pt, raise=-1ex,mirror}](8.5,5.1) -- (8.5, 5.9);
            \draw (8.8, 5.5) node{$(3)$};

            \draw[decorate,decoration={brace,amplitude=5pt, raise=-1ex,mirror}](8.5,1.1) -- (8.5, 3.9);
            \draw (8.8, 2.5) node{$(1)$};

            \draw[decorate,decoration={brace,amplitude=5pt, raise=-1ex,mirror}](8.5,4.1) -- (8.5, 4.9);
            \draw (8.8, 4.5) node{$(2)$};

            \draw[decorate,decoration={brace,amplitude=5pt, raise=-1ex,mirror}](8.5,6.1) -- (8.5, 6.9);
            \draw (8.8, 6.5) node{$(4)$};
 
        \end{tikzpicture}
    \end{center}
    
    It is easy to see that this diagram $D$ induces $M$. Moreover, since $\omega$ cannot cross any two braids that were already crossed by $\sigma$, this diagram is also minimal.
    
     In the 2-representation, the diagram $D$ sends $1_{\bar{\textbf{i}}}\otimes 1_{\textbf{j}}$ to \[\tau_{i_1-|\textbf{i}|-1}\dots\tau_{i_k-|\textbf{i}|-1}r(x_{-1},\dots x_{1})1_{\overline{\sigma(\textbf{i})}}\otimes 1_{\omega(\textbf{j})}p(x_{-|\textbf{j}'|},\dots x_{|\textbf{j}|})\tau_{-j_l}\dots\tau_{-j_1} .\] Recall that each $S_{|\textbf{j}'|}$ coset of $S_{|\textbf{j}|}$ has a unique minimal representative such that $\omega^{-1}$ does not swap the symbols $1$ through $|\textbf{j}|-|\textbf{j}'|$ amongst themselves. The various $s_{|\textbf{j}|-j_l}\dots s_{|\textbf{j}|-j_1}$ obtained by the process described above all evidently have this property. Distinct elements of $d_{m-n}(sseq(X),sseq(Y))$ must differ in either $\sigma$, the coset for $\omega$, or the number of dots on some braid. Corollary \ref{cor:kholau2_ef_basis} then gives that the images of these elements are linearly independent. So, this choice of minimal diagrams for each pairing gives a basis of $\Hom_{\mathcal{A}_2(C)}(X,Y)$. 

    We now prove the theorem under the assumption that $X=q^mE_{\textbf{i}}F_{\textbf{j}}$ for some $m\in \Z$ and $\textbf{i},\textbf{j}\in$ Seq and with no restriction on $Y$. Write $Y=q^nE_{\textbf{l}}$ for some $\textbf{l}\in $ SSeq. If $Y$ has the form $q^nE_{\textbf{i}'}F_{\textbf{j}'}$ for some $n\in \Z$, $\textbf{i}',\textbf{j}'\in$ Seq, then the arguments above apply. Otherwise, we may write $Y=AF_jE_iB$ for some $A,B$. Lemma \ref{lem:rho_onto} gives us a choice of $B_{X,Y}$ consisting of minimal diagrams that factor through either $A\sigma_{ij}B$ or $AF_jX_i^nB\circ A\eta_i B$. Moreover, for any two distinct $B_{X,Y}$ elements that factor through $A\sigma_{ij}B$, the induced minimal diagrams to $q_i^{-C_{ij}}AE_iF_jB$ yield distinct dotted $\mathcal{A}_2$-pairings of $X$ and $q_i^{-C_{ij}}AE_iF_jB$. We denote by $S_\sigma $ the subset of $\Hom_{\mathcal{A}_2(C)}(X,q_i^{-C_{ij}}AE_iF_jB)$ obtained in this way. We similarly denote by $S_{\eta,n}$ the subset of $\Hom_{\mathcal{A}_2(C)}(X,q_i^{2n}AB)$ induced in this way, and note that the elements of these sets all correspond to distinct dotted pairings. We have therefore partitioned $B_{X,Y}$ as \[B_{X,Y}=A\sigma_{ij}B\circ S_\tau \sqcup \bigsqcup_n AF_iX_i^nB\circ A\eta_iB\circ S_{\eta,n}.\] By Proposition \ref{prop:kk}, each of the sets in this disjoint union factor through a distinct summand of the functor $AF_jE_iB$ and are therefore linearly independent from each other. So, we need only prove that each of the sets $S_\tau$ and $S_{\eta,n}$ are linearly independent. This follows from proving the theorem for the source $X$ and the targets $q_i^{-C_{ij}}AE_iF_jB$ and all $q_i^{2n}AB$ for $n\in \N$. Both subcases have fewer ordered pairs $(a,b)$ of entries in the target with $a$ negative, $b$ positive, and $a$ an earlier entry than $b$. So, we may argue by induction on the number of such pairs, with the base case being proven in the previous paragraphs. 

    We finally prove the theorem for arbitrary $X$ and $Y$. If $X$ has the form $q^mE_{\textbf{i}}F_{\textbf{j}}$ for some $m\in \Z$ and $\textbf{i},\textbf{j}\in $ SSeq, then we are in the case considered in the previous paragraph. Otherwise, as above, we may assume $X$ has the form $AF_iE_jB$ for some objects $A$ and $B$. By Lemma \ref{lem:a2_forbidden_pairings}, in any $\mathcal{A}_2$-diagram from $X$ to $Y$, the braids attached to these entries $F_i$ and $E_j$ can neither cross nor coincide. So, if we precompose a minimal $\mathcal{A}_2$-diagram from $X$ to $Y$ by the crossing $A\sigma_{ji}B$, then we obtain a minimal diagram from $q_i^{-C_{ij}}AE_jF_iB$ to $Y$. If we fix a choice of $B_{X,Y}$, then the elements of $\{x\circ A\sigma_{ji}B | x\in B_{X,Y}\}$ are minimal diagrams corresponding to different dotted pairings from $q_i^{-C_{ij}}AE_jF_iB$ to $Y$. It is therefore sufficient to prove the theorem for this new source and with target $Y$. As in the previous paragraph, we may argue by induction and reduce to the case that $X$ has the form $q^mE_{\textbf{i}}F_{\textbf{j}}$. This case is proven in the previous paragraph.
\end{proof}

\begin{rem}\label{rem:2embeddings}
    Let $\mathcal{A}_{2E}(C)$, resp $\mathcal{A}_{2F}(C)$ denote the full subcategory of $\mathcal{A}_2(C)$ generated by the $E_i$, resp. the $F_i$. There is a natural monoidal functor $\mathcal{A}_1(C)\xhookrightarrow{\text{inc}} \mathcal{A}_{2E}(C)$ that is the identity on objects and morphisms. Theorem \ref{thm:a2_2rep_faithful} shows that this is a faithful embedding, and Lemma \ref{lem:a2_htpy} shows that this is full. The composition $\phi\circ \text{inc}$ is a fully faithful embedding into $\mathcal{A}_{2F}(C)$, and it is monoidal if we reverse the monoidal product on $\mathcal{A}_{2F}(C)$.
\end{rem}
This theorem gives us an explicit description of every Hom space in $\mathcal{A}_2(C)$. Because of this, we may reduce many complex algebraic questions about $\mathcal{A}_2(C)$ to topological questions about minimal $\mathcal{A}_2$-diagrams. In particular, we need the following result.

\begin{cor}\label{cor:rho_mono}
For any objects $Y,Z\in \mathcal{A}_2(C)$, the morphism $Y\rho_{ij}Z$ is a monomorphism, i.e. $Y\rho_{ij}Z \circ f=0$ implies $f=0$.
\end{cor}
\begin{proof}
    Let $f:X\rightarrow YR_{ij}Z$ be a morphism such that $Y\rho_{ij}Z\circ f=0$. The morphism $f$ determines a morphism $f_\tau:X\rightarrow q_i^{-C_{ij}}YE_iF_jZ$, and when $i=j$, $f_{\eta,n}:X\rightarrow q_i^nYZ$. By the universal property of the biproduct, $f=0$ whenever $f_\tau$ and all $f_{\eta,n}=0$. Fix a choice of each $B_{X,q_i^{-C_{ij}}YE_iF_jZ}$ and $B_{X,q_i^nYZ}$. Since these sets span their corresponding Hom spaces, we may expand the $f_\tau$ and $f_{\eta,n}$ as $f_\tau=\sum_{k=1}^m a_{\tau,k}f_{\tau,k}$ and $f_{\eta,n}=\sum_{k=1}^ma_{\eta,n,k}f_{\eta,n,k}$ where the $a_{\tau,k},a_{\eta,n,k}\in \kf$ and the $f_{\tau,k}$ and $f_{\eta,n,k}$ are elements of our spanning sets. Then we compute \[0=Y\rho_{ij}Z\circ f=\sum_{k=1}^m (a_{\tau,k}Y\sigma_{ij}Z\circ f_{\tau,k}+\sum_{n}a_{\eta,n,k}YF_jX_i^nZ\circ Y\eta_iZ\circ f_{\eta,n,k}).\]
    This sum is finite by Lemma \ref{lem:a2_lblfd}. Moreover, each term in this sum is a multiple of a minimal diagram corresponding to a distinct element of $d(sseq(X),sseq(YF_jE_iZ))$. The fact that the terms $Y\sigma_{ij}Z\circ f_{\tau,k}$ are minimal follows from Lemma \ref{lem:a2_forbidden_pairings}. By Theorem \ref{thm:a2_2rep_faithful}, all coefficients must be zero, and therefore $f=0$.
\end{proof}

Homomorphisms between similar reduced objects are easy to describe.

\begin{lem}\label{lem:a2_reduced_hom}
Fix $\textbf{i},\textbf{i}',\textbf{j},\textbf{j}'\in $ Seq such that $\alpha_{\textbf{i}}=\alpha_{\textbf{i}'}$ and $\alpha_{\textbf{j}}=\alpha_{\textbf{j}'}$. Then there is an isomorphism of graded $\kf$-vector spaces
\[\Hom_{\mathcal{A}_2(C)}^\bullet(E_\textbf{i}F_\textbf{j},E_{\textbf{i}'}F_{\textbf{j}'})\simeq 1_{\bar{\textbf{i}'}}H_{\alpha_\textbf{i}}(Q)1_{\bar{\textbf{i}}}\otimes_{\kf}1_{\textbf{j}'}H_{\alpha_\textbf{j}}(Q)1_{\textbf{j}}.\]
\end{lem}
\begin{proof}
    By Theorem \ref{thm:a2_2rep_faithful}, this graded Hom space has a basis given by certain  minimal $\mathcal{A}_2$ diagrams from $E_\textbf{i}F_\textbf{j}$ to $E_{\textbf{i}'}F_{\textbf{j}'}$. Since the target of these diagrams is reduced, we may apply Lemma \ref{lem:a2_forbidden_pairings} to see that these diagrams cannot match entries of the target amongst themselves. Each entry of the target is therefore matched with an identically labeled entry of the source. Since $\alpha_{\textbf{i}}=\alpha_{\textbf{i}'}$ and $\alpha_{\textbf{j}}=\alpha_{\textbf{j}'}$, every entry of the source must be matched to an entry of the target. Since these diagrams are minimal, none of the entries of $E_{\textbf{i}}$ can cross entries of $F_{\textbf{j}}$, or else there must be a double crossing. Therefore, after an appropriate choice of bases, the injective horizontal composition map $\sqcup_{m\in \Z}( B_{q^mE_\textbf{i},E_{\textbf{i}'}}\times B_{q^{n-m}F_\textbf{j},q^nF_{\textbf{j}'}})\rightarrow B_{q^nE_{\textbf{i}}F_\textbf{j},E_{\textbf{i}'}F_{\textbf{j}'}}$ is surjective for any $n\in \Z$. By Theorem \ref{thm:a2_2rep_faithful}, this induces $\kf$-vector space isomorphisms \[\bigoplus_{m\in \Z} (\Hom_{\mathcal{A}_2(C)}(q^mE_\textbf{i},E_{\textbf{i}'})\otimes_\kf \Hom_{\mathcal{A}_2(C)}(q^{n-m}F_\textbf{j},F_{\textbf{j}'}))\rightarrow \Hom_{\mathcal{A}_2(C)}(q^nE_\textbf{i}F_\textbf{j},E_{\textbf{i}'}F_{\textbf{j}'}).\]
    The source of this isomorphism is the degree $n$ component of the graded vector space \[\text{Hom}^\bullet_{\mathcal{A}_2(C)}( E_\textbf{i},E_{\textbf{i}'})\otimes_{\kf} \text{Hom}^\bullet_{\mathcal{A}_2(C)}( F_\textbf{j},F_{\textbf{j}'}).\] Remark \ref{rem:2embeddings} shows that \[ \text{Hom}^\bullet_{\mathcal{A}_2(C)}( E_\textbf{i},E_{\textbf{i}'})\simeq \text{Hom}^\bullet_{\mathcal{A}_1(C)}( E_\textbf{i},E_{\textbf{i}'})\simeq 1_{\bar{\textbf{i}'}}H_{\alpha_\textbf{i}}(Q)1_{\bar{\textbf{i}}},\] and \[ \text{Hom}^\bullet_{\mathcal{A}_2(C)}( F_\textbf{j},F_{\textbf{j}'})\simeq  \text{Hom}^\bullet_{\mathcal{A}_1(C)}( E_{\overline{\textbf{j}}},E_{\overline{\textbf{j}'}})\simeq 1_{\textbf{j}'}H_{\alpha_\textbf{j}}(Q)1_{\textbf{j}}.\]
\end{proof}

\begin{rem}
    A similar claim holds for ``coreduced" objects of the form $q^mF_\textbf{i}E_\textbf{j}$, although we will not need this.
\end{rem}
\subsection{Localizing to obtain the quantum boson category}\label{subsec:localize}

We show how to obtain our desired quantum boson category via a localization at a certain set of morphisms containing the $\rho_{ij}$. We show that this is equivalent to a coreflective localization at a slightly larger set of morphisms.

\begin{defn}
    We say that an object $X$ of ${_\textbf{B}\mathcal{A}_2(C)}$ is a \emph{reduced series} if it is a \textbf{B}-biproduct of reduced objects, i.e. if it has the form of a $\textbf{B}$-biproduct $X=\bigoplus_{i\in \Z}q^i(Y_1^{\oplus k_{1,i}}\oplus \dots \oplus Y_n^{\oplus k_{n,i}})$ with each $Y_j$ reduced.
\end{defn}

\begin{defn}
    For any object $X\in {_\textbf{B}\mathcal{A}_2(C)}$, we define an object $R(X)\in {_\textbf{B}\mathcal{A}_2(C)}$ as follows. We define the map to be linear over left-bounded locally finite biproducts, so it is enough to specify it on $\mathcal{A}_2(C)$.  If $X$ is already reduced, then we say $R(X)\coloneqq X$. Otherwise, we can uniquely up to shifts write $X=YF_jE_iZ$ with $Y$ reduced. Then we say \[R(X)\coloneqq R(YR_{ij}Z).\] The terms in the biproduct expansion of $R(X)$ have fewer ordered pairs of entries $(a,b)$ with $a$ negative, $b$ positive, and $a$ earlier than $b$. So, we can inductively define $R$ in this way. This is a well-defined map of  ${_\textbf{B}\mathcal{A}_2(C)}$ objects since the maximum number of such pairs over terms in any left-bounded locally finite biproduct of $\mathcal{A}_2(C)$ objects is finite.
    
    There is a canonical morphism $\rho_X$ from $R(X)$ to $X$ obtained by appropriately composing and taking biproducts of the maps $Y\rho_{ij}Z$, where we just take $\rho_X=\text{id}_X$ if $X$ is reduced.

    In this notation, $\rho_{ij}=\rho_{F_jE_i}$ and $R_{ij}=R(F_jE_i)$.
\end{defn}
The following is immediate from the definitions.
\begin{prop}\label{prop:M_kinda_monoidal}
    For any objects $X$ and $Y$, we have $R(XY)=R(R(X)Y)$ and $\rho_{XY}=\rho_XY\circ \rho_{R(X)Y}$.
\end{prop}

For concreteness, we made a choice in the definition of $\rho_X$ where we begin to take $\rho$'s from the left. The next lemma shows that any order of $\rho$'s gives the same morphism.
\begin{lem}\label{lem:rho_order_doesnt_matter}
    Suppose we have ${_\textbf{B}\mathcal{A}_2(C)}$ objects $Z=AF_jE_iB$ and $Y=AR_{ij}B$ with $A,B\in {_\textbf{B}\mathcal{A}_2(C)}$. Then $A\rho_{ij}B\circ \rho_Y=\rho_Z$.
\end{lem}
The difference here is that $A$ is not assumed to be a reduced series.
\begin{proof}
    By Proposition \ref{prop:M_kinda_monoidal}, we compute 
    \[\rho_Z=\rho_AF_jE_iB\circ R(A)\rho_{ij}B\circ \rho_{R(A)R_{ij}B},\]
     where we have used that $R(A)F_j$ is a reduced series. We also compute
    \[\rho_Y=\rho_A R_{ij}B\circ \rho_{R(A)R_{ij}B}.\]
    The claim follows from noting that 
    \[\rho_AF_jE_iB\circ R(A)\rho_{ij}B=A\rho_{ij}B\circ \rho_AR_{ij}B=\rho_A\rho_{ij}B.\]
\end{proof}
\begin{lem}\label{lem:M_functorial}
    For any morphism $f:X\rightarrow Y$ in ${_\textbf{B}\mathcal{A}_2(C)}$, there exists a unique morphism $R(f):R(X)\rightarrow R(Y)$ such that $f\circ \rho_X=\rho_Y\circ R(f)$. Moreover, $R$ defines an endofunctor of ${_\textbf{B}\mathcal{A}_2(C)}$.
\end{lem}
\begin{proof}
    For existence, we define such a morphism on each summand of $R(X)$. By linearity over biproducts of $R$ and $\rho_Y$, we may assume that $Y\in \mathcal{A}_2(C)$. Then the claim follows from Lemma \ref{lem:rho_onto} and the definition of $\rho_Y$. Uniqueness follows from Corollary \ref{cor:rho_mono} and by noting that the set of monomorphisms is closed under composition and taking biproducts. 

    For the second claim, it is clear that $R(\text{id}_X)=\text{id}_{R(X)}$ for any $X$. The uniqueness of $R(f\circ g)$ shows that $R(f\circ g)=R(f)\circ R(g)$. 
\end{proof}

We can therefore extend Proposition \ref{prop:M_kinda_monoidal} to morphisms. 

\begin{prop}\label{prop:M_sorta_monoidal_functor}
    For any morphism $f:X\rightarrow Y$ and for any object $Z$ in ${_\textbf{B}\mathcal{A}_2(C)}$, we have that $R(fZ)=R(R(f)Z)$.
\end{prop}

Denote by $S$ the set of all identity morphisms in $\mathcal{A}_2(C)$ and all morphisms of the form $Y\rho_{ij}Z$ for some $Y,Z\in \mathcal{A}_2(C)$ and $i,j\in I$. Denote by $S^{\oplus}$ the closure of $S$ under taking left-bounded locally finite biproducts. Let $S^{\oplus \circ}$ be the closure of $S^\oplus$ under composition. It is easy to see that $S^{\oplus \circ}$ is closed under composition, biproducts, shifts and the monoidal product. 

\begin{lem}\label{lem:rightfrac}
    The set $S^{\oplus \circ}$ admits a calculus of right fractions.
\end{lem}
\begin{proof}
    $S^{\oplus \circ}$ contains all identity morphisms in $_{\textbf{B}}\mathcal{A}_2(C)$ since it has all $\mathcal{A}_2(C)$ identities and is closed under biproducts. It is defined to be closed under composition. To prove condition 4 of Definition \ref{defn:calc_of_frac}, we show that each $S^{\oplus \circ}$ element is in fact monic. The class of monic morphisms is closed under composition and biproducts, so it is enough to show that the $S$ elements are monic. This is Corollary \ref{cor:rho_mono}. To show condition 3, suppose we are given a cospan $X\xrightarrow{f} Y \xleftarrow{s} Z$ with $s\in S^{\oplus \circ}$. We wish to complete this to a commutative square with the side opposite $s$ given by an $S^{\oplus \circ}$ morphism. By iterating our argument, it is enough to assume that $s\in S^\oplus$, and by linearity, it is enough to assume $s\in S$. Then we complete our square with the span $X\xleftarrow{\rho_X} R(X) \xrightarrow{\rho_Z\circ R(f)} Z$. We indeed have $\rho_X\in S^{\oplus \circ}$. For commutativity, we need to show that $s\circ \rho_Z\circ R(f)=f\circ \rho_X$. This follows from Lemmas \ref{lem:rho_order_doesnt_matter} and Lemma \ref{lem:M_functorial}.
\end{proof}
\begin{defn}
    We define the quantum boson category $\mathcal{B}(C)$ to be the localization ${_\mathbf{B}\mathcal{A}_2(C)}[S^{\oplus \circ -1}]$
\end{defn}
\begin{lem}
    The 2-representation of $\mathcal{A}_2(C)$ on $\bigoplus_{\alpha} H_{\alpha}(C)-\text{grMod}$ descends to a 2-representation of $\mathcal{B}(C)$.
\end{lem}
\begin{proof}
    The 2-representation extends to ${_\textbf{B}\mathcal{A}_2(C)}$ since $\bigoplus_\alpha H_\alpha(C)-\text{grMod}$ contains all left-bounded locally finite direct sums.  

    Now, recall that Proposition \ref{prop:kk} shows that each $\rho_{ij}$ is inverted in our 2-representation, and since the 2-representation is monoidal, the same is true for each element of $S$. Biproducts and compositions of isomorphisms are isomorphisms, so this 2-representation inverts all of $S^{\oplus \circ}$. By the universal property of localization, this 2-representation descends to one of $\mathcal{B}(C)$.
\end{proof}

Note that $\phi(\rho_{ij})=\rho_{ij}$ and therefore that the symmetry $\phi$ descends to a monoidal equivalence $\mathcal{B}(C)\rightarrow \mathcal{B}(C)^{co}$.

In the rest of this subsection, we show that $\mathcal{B}(C)$ is equivalently a coreflective localization of ${_\textbf{B}\mathcal{A}_2(C)}$ by $R$. Both descriptions have their uses. The localization at $S^{\oplus \circ}$ is useful due to how explicit this set is, and therefore how easy it is to verify the universal property. The coreflective localization is useful due to a somewhat less cumbersome description of the morphisms.

Denote by $\mathcal{B}'(C)$ the image of $R$, i.e., the full subcategory consisting of the reduced series. 
\begin{cor}\label{cor:a2_coreflect}
    The fully faithful inclusion $L:\mathcal{B}'(C)\rightarrow {_\textbf{B}\mathcal{A}_2(C)}$ is left adjoint to $R$.
\end{cor}
\begin{proof}
    The unit of adjunction is the identity transformation. The counit $LR\rightarrow \text{Id}$ is defined objectwise by $R(X)\xrightarrow{\rho_X}X$. Naturality follows from Lemma \ref{lem:M_functorial}. Moreover, Lemma \ref{lem:M_functorial} shows that for any $X\in \mathcal{B}'(C)$ and any $Y\in {_\textbf{B}\mathcal{A}_2(C)}$, we have that the map \[\Hom_{\mathcal{A}_2(C)}(X,R(Y))\xrightarrow{\rho_Y\circ}\Hom_{\mathcal{A}_2(C)}(X,Y)\] is an isomorphism. 
\end{proof}
\begin{cor}
    Let $W$ be the set of all morphisms in ${_\textbf{B}\mathcal{A}_2(C)}$ that are sent to isomorphisms by $R$. Then $W$ admits a calculus of right fractions and $\mathcal{B}'(C)$ is equivalent to the localization $ {_\textbf{B}\mathcal{A}_2(C)}[W^{-1}]$.
\end{cor}
\begin{proof}
    Apply Proposition \ref{prop:coreflect_to_localiz}. 
\end{proof}

\begin{lem}
    $W$ is monoidally closed. The monoidal structure on ${_\textbf{B}\mathcal{A}_2(C)}$ descends to one on ${_\textbf{B}\mathcal{A}_2(C)}[W^{-1}]$.
\end{lem}

\begin{proof}
For the first claim, note that for any morphism $f$ and any object $X\in {_\textbf{B}\mathcal{A}_2(C)}$, Proposition \ref{prop:M_sorta_monoidal_functor} shows that $R(f X)=R(R(f)X)$. If $R(f)$ is an isomorphism with inverse $g$, then for any object $X$,  $R(f)X$ has inverse $gX$. Any functor sends isomorphisms to isomorphisms, so $R(R(f)X)=R(fX)$ is an isomorphism. We use Lemma \ref{lem:rho_order_doesnt_matter} to argue similarly for the reverse multiplication.

Now, for $f:A\rightarrow B$ and $f':X\rightarrow Y\in S$, we compute \[R(ff')=R(fY\circ A f')=R(fY)\circ R(Af').\]
The composition of isomorphisms is an isomorphism, and so the first claim holds. 
The second claim then follows from Proposition \ref{prop:monoidal_localiz}.
\end{proof}
The corresponding monoidal structure on $\mathcal{B}'(C)$ is given by $X\otimes Y\coloneqq R(XY)$, $f\otimes X \coloneqq R(fX)= R(R(f)R(X))$, and likewise for the reverse product.
\begin{lem}
Denote by $\mathcal{L}$ the localization functor ${_\textbf{B}\mathcal{A}_2(C)}\rightarrow \mathcal{B}(C)$. Then the composite $\mathcal{L}\circ L$ is a lax monoidal equivalence of categories $\mathcal{B}'(C)\rightarrow \mathcal{B}(C)$.
\end{lem}

\begin{proof}
    This functor is clearly lax monoidal with the defining natural transformations $\mathcal{L}\circ L(X)\otimes \mathcal{L}\circ L(Y)\xrightarrow{\rho_{XY}^{-1}}\mathcal{L}\circ L (R(XY))$ and identity natural transformation of units. Essential surjectivity follows from the isomorphism $X\simeq R(X)$ in $\mathcal{B}(C)$. For fully-faithfulness, observe that for any reduced series $X\in \mathcal{B}'(C)$, the only $S^{\oplus\circ}$-morphism whose target is $X$ is the identity on $X$. One can then quickly verify that the equivalence class of the span $X\xleftarrow{=}X\xrightarrow{f} Y$ contains only this span. So, \[\Hom_{\mathcal{B}(C)}(\mathcal{L}\circ L (X),\mathcal{L}\circ L(Y))\simeq\Hom_{{_\textbf{B}\mathcal{A}_2(C)}}(L (X), L(Y))\simeq\Hom_{\mathcal{B}'(C)}(X,Y).\]
\end{proof}

Note that this result also shows that $\mathcal{L}$ preserves all biproducts of ${_\textbf{B}\mathcal{A}_2(C)}$. Note that $\mathcal{B}'(C)\simeq \mathcal{B}(C)$ is strictly graded since $_\textbf{B}\mathcal{A}_2(C)$ is.

We obtain an inverse equivalence as follows. We observe that $S^{\oplus\circ}\subset W$, and so by the universal property of $\mathcal{B}(C)$, the functor $R$ on ${_\textbf{B}\mathcal{A}_2(C)}$ descends to a functor on $\mathcal{B}(C)$. It is easy to check that this is indeed an inverse equivalence. 

\subsection{Summary of \texorpdfstring{$\mathcal{B}(C)$}{B(C)}}
We review the structure of $\mathcal{B}(C)$ and state precisely its universal property.

From the definition, the category $\mathcal{B}(C)$ is a monoidal category with the following structure.
\begin{itemize}
    \item Objects are \textbf{B}-biproducts in finite sequences in the letters $E_i$ and $F_j$ for $i,j\in I$.
    \item Morphisms from $X$ to $Y$ are given by equivalence classes of spans in ${_\textbf{B}\mathcal{A}_2(C)}$ of the form $X\xleftarrow{s} Y \xrightarrow{f} Z$ with $s\in S^{\oplus\circ}$. In particular, for any object $X$, we have an isomorphism $X\simeq R(X)$, where $R(X)$ is a certain $\textbf{B}$-biproduct of reduced objects. 
    
    \item For any $\textbf{i},\textbf{j}\in $SSeq and any $m,n\in \Z$, we have that $\Hom_{\mathcal{B}(C)}(q^mE_\textbf{i},q^nE_{\textbf{j}})$ contains all $\mathcal{A}_2$-diagrams of degree $m-n$ from $E_\textbf{i}$ to $E_\textbf{j}$. If we pick a minimal diagram associated to each element of $d_{m-n}(\textbf{i},\textbf{j})$, then the corresponding elements of this Hom space are linearly independent. Moreover, if $E_\textbf{i}$ and $E_{\textbf{j}}$ are both reduced, then this set forms a basis. Morphisms between other objects are determined by the biproduct rules.
    
    \item There is a monoidal structure given as follows. For $m,n\in \Z$ and $\textbf{i},\textbf{j}\in $ SSeq, we have that $q^mE_\textbf{i}\otimes q^nE_{\textbf{j}}=q^{m+n}E_{\textbf{ij}}$, and the monoidal product of two $\mathcal{A}_2$-diagrams is given by horizontal composition. This product distributes over biproducts in the natural way.
\end{itemize}

An equivalent description is given as follows. We use again the notation $\mathcal{B}'(C)$ to point out that the set of objects is different.
\begin{itemize}
    \item There are reduced objects of the form $q^mE_{\textbf{i}}F_{\textbf{j}}$ for each $m\in \Z$ and $\textbf{i},\textbf{j}\in $ Seq. The objects of $\mathcal{B}'(C)$ are all \textbf{B}-biproducts in these reduced objects.
    \item The morphisms between two reduced objects are the linear combinations of the $\mathcal{A}_2$-diagrams between the corresponding signed sequences. We obtain a basis for each such Hom space by picking a minimal diagram associated to each appropriately graded dotted pairing of the source and target. Morphisms between other objects are obtained by the biproduct rules.
    \item The monoidal product is $X\otimes Y\coloneqq R(XY)$, with the product of morphisms given by $f\otimes g=R(fg)$. The concatenation refers to the monoidal product in ${_\textbf{B}\mathcal{A}_2(C)}$.

\end{itemize}

The universal property of $\mathcal{B}(C)$ is as follows. An additive graded monoidal functor $P:\mathcal{B}(C)\rightarrow \mathcal{D}$ is an additive graded monoidal functor $P:\mathcal{A}_1(C)\rightarrow \mathcal{D}$ such that
\begin{itemize}
    \item Each object $P(E_i)$ has a right dual $P(F_i)$. We therefore can extend uniquely up to isomorphism the functor $P$ to $\mathcal{A}_2(C)$
    \item The target $\mathcal{D}$ contains all $\textbf{B}$-biproducts of the image objects of $\mathcal{A}_2(C)$. We therefore can extend uniquely up to isomorphism the functor $P$ to ${_\textbf{B}\mathcal{A}_2(C)}$.
    \item Each of the morphisms $P(\rho_{ij})$ is an isomorphism in $D$. The functor $P$ therefore descends uniquely to $\mathcal{B}(C)$.
\end{itemize}

\subsection{Identification of \texorpdfstring{$K_0$}{K0}}\label{subsec:K0}
We introduce the appropriate Grothendieck ring for decategorification and show that it is isomorphic to $_\mathbf{B}B(C)$. The proofs in this subsection are based on those in \cite{kholau3} for the categorification of the idempotented quantum group.

The category $\mathcal{B}(C)$ admits certain infinite direct sums which are not captured by the usual Grothendieck group. The correct notion of Grothendieck group to use here is a kind of \emph{topological split Grothendieck group}, which we also denote as $K_0$.  Naisse and Vaz developed the theory of these groups in \cite{nava}. Their notion of topological split Grothendieck groups carries the structure of a topological $\Z[[q]][q^{-1}]$-module, which essentially says that sums like $(1+q+q^2+q^3+\dots )[M]=[N]$ make sense. We have only allowed \textbf{B}-biproducts in $\mathcal{B}(C)$ so that our ring of coefficients is a subring of $\Q(q)$. So, we only expect a \textbf{B} action on our Grothendieck group. The hypotheses of the results of \cite{nava} are easily modified to allow for a smaller ring of coefficients.

Since our category $\mathcal{B}(C)$ is abstractly defined, we need to verify that the topological split Grothendieck group of its idempotent completion is well-behaved. In what follows, we introduce the relevant properties from \cite{nava}.

\begin{defn}
    Denote by $\mathcal{B}(C)^{i}$ the idempotent completion of $\mathcal{B}(C)$. The category $\mathcal{B}(C)^{i}$ inherits a $\Z$-grading and monoidal product from $\mathcal{B}(C)$.
\end{defn}
\begin{defn}
    An additive, strictly $\Z$-graded category $\mathcal{C}$ is called \emph{\textbf{B}-additive} if every \textbf{B}-coproduct in $\mathcal{C}$ is a biproduct. We say that a \textbf{B}-additive category is \emph{\textbf{B}-complete} if it admits all \textbf{B}-biproducts.
\end{defn}
Lemma \ref{lem:a2_biprod} shows that $\mathcal{B}(C)^{i}$ is \textbf{B}-complete.
\begin{defn}
    Let $\mathcal{C}$ be a $\textbf{B}$-additive category. We say that the object $X$ of $\mathcal{C}$ is \emph{\textbf{B}-small} if every map of the form $f:X\rightarrow \bigoplus_{i\subset I} Y_i$ with the target a $\textbf{B}$-biproduct factors through some $\bigoplus_{j\in J} Y_j$ for $J$ a finite subset of $I$.
\end{defn}
\begin{prop}\label{prop:a2_smallobj}
    In $\mathcal{B}(C)^{i}$, any direct summand of a reduced object is \textbf{B}-small.
\end{prop}
\begin{proof}
    Direct summands of \textbf{B}-small objects are \textbf{B}-small, so it is sufficient to prove the claim for the reduced objects. This follows from the proof of Lemma \ref{lem:a2_biprod}.
\end{proof}
\begin{defn}
    A \textbf{B}-additive category $\mathcal{C}$ is called \emph{\textbf{B}-Krull-Schmidt} if every object of $\mathcal{C}$ is isomorphic to a \textbf{B}-biproduct of \textbf{B}-small objects that have local endomorphism rings.
\end{defn}
\begin{lem}\label{lem:a2_locally_KS}
    The category $\mathcal{B}(C)^{i}$ is \textbf{B}-Krull-Schmidt.
\end{lem}
\begin{proof}
    Every object of $\mathcal{B}(C)$ decomposes into a \textbf{B}-biproduct of reduced objects. So, it is sufficient to show the claim for a fixed reduced object $X$. By Corollary \ref{cor:a2_coreflect}, $\Hom_{\mathcal{B}(C)^{i}}(X,X)=\Hom_{\mathcal{A}_2(C)}(X,X)$, and by Lemma $\ref{lem:a2_lblfd}$, this Hom space is finite-dimensional. A reduced object $X$ therefore decomposes in $\mathcal{B}(C)^{i}$ as a finite direct sum of indecomposable objects, and moreover, every indecomposable arises this way. Since $\mathcal{B}(C)^{i}$ is idempotent complete and since endomorphism rings of indecomposables in $\mathcal{B}(C)^i$ are finite-dimensional, the indecomposables have local endomorphism rings. The indecomposables are \textbf{B}-small due to Proposition \ref{prop:a2_smallobj}.
\end{proof}
\begin{prop}[\cite{nava} Theorem 5.13]\label{thm:nv}
The topological split Grothendieck group of a \textbf{B}-complete \textbf{B}-Krull-Schmidt category equipped with the $(q)$-adic topology is a free $\mathbf{B}$-module generated by the classes of indecomposables (up to shifts).
\end{prop}
Lemma \ref{lem:a2_locally_KS} shows that this proposition applies to $\mathcal{B}(C)^{i}$.

The free $\mathbf{B}$-module $K_0(\mathcal{B}(C)^{i})$ inherits the structure of a $\mathbf{B}$-algebra due to the monoidal structure on $\mathcal{B}(C)^{i}$ that distributes over $\textbf{B}$-biproducts. We now begin to show that this Grothendieck group is isomorphic to the quantum boson algebra after extending scalars.

There is a $\Q$-linear involution $\bar{*}$ on $\Q(q)$ defined by $\bar{q}\rightarrow q^{-1}$. The involution $\bar{*}$ is also defined on $\mathbf{B}$ since $\overline{1/(1-q_i^2)}=-q_i^2/(1-q_i^2)$.

\begin{defn}
    There is a $\Z[q,q^{-1}][[q]]$-valued $q$-semilinear form $(*,*)_{K_0}$ defined on $K_0(\mathcal{B}(C)^i)$ as follows. Fix a \textbf{B}-basis of $K_0(\mathcal{B}(C)^i)$ of indecomposables $[M_i]$. Then define \[([M_i],[M_j])_{K_0}\coloneqq \text{grdim(Hom}_{\mathcal{B}(C)^i}^\bullet(M_i,M_j)).\]
     Lemma \ref{lem:a2_lblfd} shows that this is a well-defined element of $\Z[q,q^{-1}][[q]]$. Then for $a\in \mathbf{B}$, we define $(v,aw)_{K_0}=(\bar{a}v,w)_{K_0}\coloneqq a(v,w)_{K_0}$.
     
     Since $\text{Hom}_{\mathcal{B}(C)^i}(*,q*)=\text{Hom}_{\mathcal{B}(C)^i}(q^{-1}*,*)$, we have that for \emph{any} indecomposable objects $M_i$ and $M_j$ that $([M_i],[M_j])_{K_0}$ may be computed as a graded dimension of Hom spaces. It is also clear that $(*,*)_{K_0}$ is linear over finite direct sums in the first argument and $\mathbf{B}$-biproducts in the second argument, i.e., if $X$ is a finite direct sum of indecomposables and $Y$ is any object of $\mathcal{B}(C)^i$, then we have that \[([X],[Y])_{K_0}= \text{grdim(Hom}_{\mathcal{B}(C)^i}^\bullet(X,Y)).\] This holds in particular if $X$ is reduced. 
\end{defn}

Let $\mathcal{A}_1(C)^i$ denote the idempotent completion of $\mathcal{A}_1(C)$. The embedding $\text{inc}:\mathcal{A}_1(C)\xhookrightarrow{\text{inc}} \mathcal{A}_{2E}(C)$ of Remark \ref{rem:2embeddings} yields an embedding $\text{inc}:\mathcal{A}_1(C)^i\xhookrightarrow{\text{inc}} \mathcal{B}(C)^{i}$. In particular, for each $i\in I$ and $a\in \N$, there are divided power objects $E_i^{(a)}\in\mathcal{B}(C)^{i}$ associated to certain idempotents and shifts in $\mathcal{A}_1(C)$. These objects satisfy $[a]_{i}!*E_i^{(a)}\simeq E_i^a$. Again following Remark \ref{rem:2embeddings}, there are similar divided power objects $F_i^{(b)}\in \mathcal{B}(C)^{i}$. For $\textbf{i}\in$ SSeq, $\textbf{j},\textbf{k}\in $ Seq, define $E_{\textbf{i}},E_{\textbf{j}}F_{\textbf{k}}\in B(C)$ in the natural way. We say that these $q^mE_\textbf{j}F_{\textbf{k}}$ are \emph{reduced} elements of $B(C)$. 

\begin{lem}
    There is a $\mathbf{B}$-algebra homomorphism $\gamma: {_\mathbf{B}B(C)}\rightarrow K_0(\mathcal{B}(C)^{i})$ given by $\gamma(E_i^{(a)})=[E_i^{(a)}]$ and $\gamma(F_i^{(b)})=[F_i^{(b)}]$.
\end{lem}
\begin{proof}
   We first show that this homomorphism exists after tensoring both sides with $\Q(q)$, in which case the domain is $B(C)$. The KLR relations on the generating morphisms ensure that the $[E_i]$ satisfy the quantum Serre relations, see \cite{qha} or \cite{khla}. Taking adjoints, we see that the $[F_i]$ do as well. The fact that each $\rho_{ij}$ is an isomorphism in $\mathcal{B}(C)^{i}$ gives the quantum boson relations in $K_0(\mathcal{B}(C)^{i})$.

    Now, we need only note that $K_0(\mathcal{B}(C)^{i})$ is a free $\mathbf{B}$-module and that each $\gamma(E_i^{(a)})$ and $\gamma(F_i^{(b)})\in K_0(\mathcal{B}(C)^{i})$. So, the restriction of $\gamma$ to $_\mathbf{B}B(C)$ has image contained in $K_0(\mathcal{B}(C)^{i})$.
\end{proof}
For any $\textbf{i}\in $ SSeq, we see that $\gamma(E_\textbf{i})=[E_\textbf{i}]$.

\begin{thm}\label{thm:gammaiso}
    The homomorphism $\gamma$ is an isomorphism of $\textbf{B}$-algebras.
\end{thm}
\begin{proof}
    We first prove that $\gamma$ is surjective, i.e., that $K_0(\mathcal{B}(C)^{i})$ is generated as an algebra by the classes $[E_i^{(a)}]$ and $[F_i^{(b)}]$. Recall that every indecomposable object is a direct summand of a reduced object. For a fixed reduced object $E_{\textbf{i}}F_{\textbf{j}}$, denote by \[X_{\alpha_{\textbf{i}},\alpha_{\textbf{j}}}\coloneqq \bigoplus_{\alpha_{\textbf{i}'}=\alpha_{\textbf{i}},\alpha_{\textbf{j}'}=\alpha_{\textbf{j}}}E_{\textbf{i}'}F_{\textbf{j}'},\] which is a finite direct sum. Every indecomposable of $K_0(\mathcal{B}(C)^{i})$ is therefore a direct summand of some $X_{\alpha_{\textbf{i}},\alpha_{\textbf{j}}}$ up to shifts. By Corollary \ref{cor:a2_coreflect} and Lemma \ref{lem:a2_reduced_hom},
    \begin{align*}
    \text{Hom}_{\mathcal{B}(C)^{i}}^\bullet(X_{\alpha_{\textbf{i}},\alpha_{\textbf{j}}},X_{\alpha_{\textbf{i}},\alpha_{\textbf{j}}})&=\bigoplus_{n\in \Z}\bigoplus_{}\Hom_{\mathcal{B}(C)^{i}}(q^nE_{\textbf{i}'}F_{\textbf{j}'},E_{\textbf{i}''}F_{\textbf{j}''})\\
    &= \bigoplus_{n\in \Z}\bigoplus_{}\Hom_{\mathcal{A}_2(C)}(q^nE_{\textbf{i}'}F_{\textbf{j}'},E_{\textbf{i}''}F_{\textbf{j}''})\\
    &\simeq H_{\alpha_\textbf{i}}(Q)\otimes_{\kf} H_{\alpha_{\textbf{j}}}(Q),  
    \end{align*}
    where the interior sum is over all $\textbf{i}',\textbf{j}',\textbf{i}'',\textbf{j}''$ such that $\alpha_{\textbf{i}}=\alpha_{\textbf{i}'}=\alpha_{\textbf{i}''}$ and $\alpha_{\textbf{j}}=\alpha_{\textbf{j}'}=\alpha_{\textbf{j}''}$. By Proposition \ref{prop:klr_tensor_K0}, we have $K_0(H_{\alpha_\textbf{i}}(Q)\otimes_{\kf}H_{\alpha_{\textbf{j}}}(Q))\simeq {_\mathbf{A}U_q^+(\mathfrak{g})}_{\alpha_\textbf{i}}\otimes_{\mathbf{A}} {_\mathbf{A} U_q^+(\mathfrak{g})}_{\alpha_\textbf{j}}$. The algebra ${_\mathbf{A}U_q^+(\mathfrak{g})}\otimes_{\mathbf{A}} {_\mathbf{A} U_q^+(\mathfrak{g})}$ is spanned by tensors of divided powers $E_i^{(a)}\otimes 1$ and $1\otimes E_i^{(b)}$. Unwinding the isomorphisms, this says that for any direct summand $M$ of $X_{\alpha_\textbf{i},\alpha_\textbf{j}}$ that $[M]$ is in the $\mathbf{A}$-subalgebra generated by the $[E_i^{(a)}]$ and $[F_i^{(b)}]$.

    For injectivity, we first show that $\gamma$ intertwines the forms $(*,*)_2$ and $(*,*)_{K_0}$, i.e. that $(v,w)_2=(\gamma(v),\gamma(w))_{K_0}$. The images of the divided powers generate both $B(C)$ and $K_0(\mathcal{B}(C)^i)$ as algebras, so since both forms are $q$-semilinear, is is enough to verify equality on products of divided powers. It is therefore also enough to verify equality on products of the $E_i$ and $F_i$. These products are spanned by the reduced products, so it is enough to assume that the first argument is reduced. So, we must compute the $([E_\textbf{i}F_\textbf{j}],[E_\textbf{k}])_{K_0}$ for $\textbf{i},\textbf{j}\in $ Seq and $\textbf{k}\in$ SSeq. We can simplify further after first showing that $[E_i]$ is left dual to $[F_i]$ for $(*,*)_{K_0}$. We know $F_i$ is right dual to $E_i$ for $(*,*)_2$ by Proposition \ref{prop:qthomform}. 
    
    For $X$ reduced, $Y$ any object of $\mathcal{B}(C)^i$, and any $i\in I$, we have that $E_iX$ and $XF_i$ are also reduced and therefore $([E_iX],[Y])_{K_0}=([X],[F_iY])_{K_0}$ and $([XF_i],[Y])_{K_0}=([X],[YE_i])_{K_0}$. Every indecomposable is a direct summand of a reduced object, so we may therefore deduce the same for $X$ indecomposable. So, $[F_i]$ is right dual to $[E_i]$ for $(*,*)_{K_0}$. So, to check that $\gamma$ intertwines the two forms, we may assume the first argument is $[\1]$. Applying quantum boson relations again, it is sufficient to compute the $([\1],[E_\textbf{i}F_\textbf{j}])_{K_0}$ for $\textbf{i},\textbf{j}\in $ Seq. Both $\1$ and $E_\textbf{i}F_\textbf{j}$ are finite direct sums of indecomposables, so we compute
    \[([\1],[E_\textbf{i}F_\textbf{j}])_{K_0}=\text{grdim(Hom}_{\mathcal{B}(C)^i}^\bullet(\1,E_\textbf{i}F_\textbf{j})).\]
    By Corollary \ref{cor:a2_coreflect}, these Hom spaces in $\mathcal{B}(C)^i$ equal the corresponding Hom spaces in $\mathcal{A}_2(C)$, i.e., are computed via minimal $\mathcal{A}_2$-diagrams from $\1$ to $E_\textbf{i}F_\textbf{j}$. We see that there are no minimal $\mathcal{A}_2$-diagrams from $\1$ to $E_\textbf{i}F_\textbf{j}$ of any degree if $E_\textbf{i}F_\textbf{j}\neq \1$, so we compute that $([\1],[E_\textbf{i}F_\textbf{j}])_{K_0}=0$ if $E_{\textbf{i}}F_\textbf{j}\neq \1$ and $([\1],[\1])_{K_0}=1$. We also compute from the definition that $(1,E_\textbf{i}F_\textbf{j})_2=0$ if $E_\textbf{i}F_\textbf{j}\neq 1$ and $(1,1)_2=1$. So, $\gamma$ intertwines $(*,*)_2$ and $(*,*)_{K_0}$.

    It is enough to prove injectivity of $\gamma$ after tensoring both sides with $\Q(q)$, in which case the domain is $B(C)$. The bilinear form $(*,*)_{K_0}$ extends to $\Q(q)\otimes_{\textbf{B}}K_0(\mathcal{B}(C)^i)$ by $q$-semilinearity since the duality $\bar{*}$ is also defined on $\Q(q)$. Suppose that there exists $v\in B(C)$ for which $\gamma(v)=0$. Then for all $w\in \Q(q)\otimes_{\textbf{B}}K_0(\mathcal{B}(C)^i)$, we have that $(\gamma(v),w)_{K_0}=0$. We have shown that $\gamma$ is surjective, so for all $u\in {B(C})$, we have $(\gamma(v),\gamma(u))_{K_0}=0$. The homomorphism $\gamma$ intertwines $(*,*)_{K_0}$ and $(*,*)_{2}$, so $v$ is in the kernel of $(*,*)_2$. But $(*,*)_2$ is nondegenerate on $B(C)$, so $v=0$. So, $\gamma$ is an isomorphism.
\end{proof}

\begin{rem}\label{rem:missingdiag}
    For $X$ and $Y$ reduced elements of $B(C)$, Theorem \ref{thm:gammaiso} shows that $(X,Y)_2$ is computed as a graded sum over $\mathcal{A}_2$-diagrams between the corresponding objects of $\mathcal{A}_2(C)$. This is not true if $X$ and $Y$ are not reduced. For example, there are no $\mathcal{A}_2$-diagrams from $F_iE_i$ to $E_iF_i$, but $(F_iE_i,E_iF_i)_2=q_i^2/(1-q_i)^2$. We show how to remedy this in Section \ref{sec:graphical} by introducing a larger class of diagrams between elements of $B_\Z(C)$.
\end{rem}

\subsection{Applications to bases}\label{subsec:bases}
We use Theorem \ref{thm:gammaiso} to produce well-behaved bases of $B(C)$.

For each symmetrizable generalized  Cartan matrix $C$, choice of parameter matrix $Q$, choice of base field $\kf$, and a choice of shift for each set of indecomposables $\{q^nX\}_{n\in \Z}$, Theorem \ref{thm:gammaiso} produces a basis of $B(C)$. The basis depends on all of these choices. We denote by $S_2$ the basis of classes of indecomposables of $B(C)$ induced by a fixed choice of $Q$, $\kf$, and shifts, and by $S$ the basis of $U_q^+(C)$ given by classes of indecomposable projective modules over KLR algebras induced by the same choices as in \cite{qha} or \cite{kholau2}. It is known that for $C$ symmetric, $\kf$ algebraically closed and characteristic zero,  a geometric choice of $Q$, and  the self-dual choice of shifts that $S$ agrees with Lusztig's canonical basis \cite{vv}\cite{qha}.

The basis $S$ is interesting because it has positive structure constants for multiplication. Theorem \ref{thm:gammaiso} allows us to strengthen this result. We have an action of $\mathcal{B}(C)^i$ on the appropriate $\textbf{B}$-extension of $\mathcal{H}^{fg}(C)$. By decategorifying, we obtain the following.

\begin{cor}
    The $\N[q,q^{-1},(1/(1-q_i^2))_{i\in I}]$-span of $S_2$ acts on the $\N[q,q^{-1},(1/(1-q_i^2))_{i\in I}]$-span of $S$.
\end{cor}
Note that each element of $\N[q,q^{-1},(1/(1-q_i^2))_{i\in I}]$ can be viewed as a Laurent series in $q$ with coefficients in $\N$.

Proposition \ref{prop:klr_tensor_K0} and the proof of Theorem \ref{thm:gammaiso} give an explicit description of $S_2$ in terms of $S_1$.
\begin{cor}
    Let $G_1:U_q^+(C)\hookrightarrow B(C)$ be the embedding of $\Q(q)$-algebras with $E_i\rightarrow E_i$, and let $G_2:U_q^+(C)\hookrightarrow B(C)$ be the embedding of $\Q(q)$-algebras with $E_i\rightarrow F_i$. Then up to shifts, \[S_2=\{G_1(X)G_2(Y)|X,Y\in S\}.\]
\end{cor}

Theorem \ref{thm:gammaiso} shows that $(*,*)_2$ can be computed from graded dimensions of Hom spaces in $\mathcal{B}_2(C)^i$. The values of $(*,*)_2$ on $S_2$ elements are easily computed. This result also follows from Theorem \ref{thm:qthom_diagram}.
\begin{cor}
    Fix $X,Y,X',Y'\in S$ for which $\text{gr}(X)\leq\text{gr}(X')$ and $\text{gr}(Y)\leq\text{gr}(Y')$. Then \[(G_1(X)G_2(Y),G_1(X')G_2(Y'))_2=(G_1(X),G_1(X'))_2*(G_2(Y),G_2(Y'))_2=(\bar{X},X')_L(\bar{Y},Y')_L.\]
\end{cor}
We can prove a similarly nice property for the symmetric form $(\bar{*},*)_2$, but we defer this to Theorem \ref{thm:bznicebasis}. Note that if $\text{gr}(X)<\text{gr}(X')$ or $\text{gr}(Y)<\text{gr}(Y')$ that this value is zero. Computing values of the bilinear form is more complicated when the inequality is reversed since there are additional cap morphisms to consider. For example, $(E_iF_i,1)_2=1/(1-q_i^2)$.

\section{Graphical interpretation of the bilinear form on \texorpdfstring{$B_\Z(C)$}{BZ(C)}}\label{sec:graphical}
We provide a new description of the bilinear form $(*,*)_\Z$ and use this interpretation to describe an interesting basis of $B_\Z(C)$.

In \cite{lusbook}, Lusztig gives a different construction of $U_q^+(\mathfrak{g})$. The bilinear form $(*,*)_L$ is also defined on the free algebra in the same generators, and $U_q^+(\mathfrak{g})$ is the quotient of the free algebra by the kernel of $(*,*)_L$. Similar results for the idempotented quantum group were established in \cite{kholau3}. We prove a similar result for the bosonic extension $B_\Z(C)$. The arguments in this section are based on those in Section 2 of \cite{kholau3}. In this section, fix a symmetrizable generalized Cartan matrix $(C_{ij})_{i,j\in I}$.

\begin{defn}
    Denote by $\Z I$ the set of all elements in $I$ with formal integers attached. Denote by $\Z$Seq the set of \emph{integer sequences} of $I$, i.e. finite sequences in $\Z I$. We allow the empty sequence.
\end{defn}
\begin{defn}
    For $(C_{ij})_{i,j\in I}$ a symmetrizable generalized Cartan matrix, denote by $F_\Z(C)$ the free $\Q(q)$-algebra on generators $E_{i,n}$ for all $i\in I,n\in \Z$. The algebra $F_\Z(C)$ has a basis of monomials indexed by $\Z$Seq. For $\textbf{i}=(n_1i_1,\dots n_ki_k)\in \Z$Seq, we denote by $E_\textbf{i}$ the corresponding monomial $E_{i_1,n_1}\dots E_{i_k,n_k}$. In particular, denote $E_\emptyset=1$. Denote by $zseq$ the reverse map from a monomial to its induced element of $\Z$Seq. We extend the notation of Definition \ref{defn:seq} to this context.
\end{defn}
We now construct the appropriate generalization of $\mathcal{A}_2$-diagrams.
\begin{defn}
    For $\textbf{i},\textbf{j}\in \Z$Seq, we define \emph{minimal $\mathcal{B}_\Z$-diagrams} from $\textbf{i}$ to $\textbf{j}$ as follows. Let $\R\times [0,1]$ be the rectangular strip with $\R\times \{0\}$ on the bottom and $\R \times \{1\}$ on the top. Pick $|\textbf{i}|$ distinct points on $\R\times \{0\}$ and label them in order from left to right by the entries of $\textbf{i}$. Do the same for $\textbf{j}$ on $\R\times \{1\}$. If $|\textbf{i}|+|\textbf{j}|$ is odd, then we say there are no minimal $\mathcal{B}_\Z$-diagrams from $\textbf{i}$ to $\textbf{j}$. Otherwise, we consider immersions of $(|\textbf{i}|+|\textbf{j}|)/2$ braids $[0,1]$ into $\R\times [0,1]$ so that the endpoints of these braids are all distinct and among the marked points on the boundaries. These immersions determine a matching between the entries of $\textbf{i}$ and $\textbf{j}$ by matching the endpoints of each braid. We require that these immersions satisfy the following properties.
    \begin{enumerate}
        \item The two endpoints of any braid must be labelled by the same element of $I$.
        \item If an entry of $\textbf{i}$ is matched with an entry of $\textbf{j}$, then these entries have the same $\Z$-label. If two entries $mi$ and $ni$ of $\textbf{i}$ are matched and if $ni$ is further right than $mi$, then $n=m+1.$ If two entries $mi$ and $ni$ of $\textbf{j}$ are matched and if $ni$ is further left than $mi$, then $n=m+1.$
        \item No braid crosses itself.
        \item No pair of braids crosses more than once.
        \item There are no triple-points, i.e. the immersion is generic.
    \end{enumerate}

    Below is such a minimal $\mathcal{B}_\Z$-diagram.
    
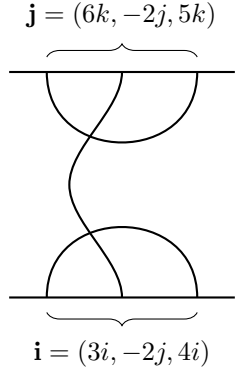
\begin{center}
    \begin{tikzpicture}
        \draw[-,thick] (-.5,0)--(2.5,0);
        \draw[-,thick] (-.5,3)--(2.5,3);
        \draw[-,thick] (0,0) .. controls (0,1.25) and (2,1.25) .. (2,0);
        \draw[-,thick] (2,3) .. controls (2,1.75) and (0,1.75) .. (0,3);
        \draw[-, thick] (1,0) .. controls (1,.5) and (.3,1) .. (.3,1.5) .. controls (.3,2) and (1,2.5)..(1,3);


        \draw[decorate,decoration={brace,amplitude=5pt,raise=-1ex}](0,3.35) -- (2, 3.35);
        \draw (1, 3.75) node{$\textbf{j}=(6k, -2j, 5k)$};

        \draw[decorate,decoration={brace,amplitude=5pt,mirror, raise=-1ex}](0,-.35) -- (2, -.35);
        \draw (1, -.75) node{$\textbf{i}=(3i,-2j, 4i)$};
    \end{tikzpicture}
\end{center}
    The second and third conditions show that each braid is ``$\Z$-oriented" in the sense that the difference in $\Z$ labels between the endpoints of any braid may be computed as a $\pi$-clockwise turning number as the braid is traversed from one endpoint to the other. We therefore label vertical pieces of any braid by the induced $\Z$-label. We consider the boundary-preserving isotopy classes of these diagrams. We say that $\textbf{i}$ is the \emph{source} of the diagram and that $\textbf{j}$ is the \emph{target}. We allow the empty diagram from $\emptyset$ to $\emptyset$. 
\end{defn}
\begin{defn}
    Fix $\textbf{i},\textbf{j}\in \Z$Seq and a minimal $\mathcal{B}_\Z$-diagram $D$ from $\textbf{i}$ to $\textbf{j}$. We define the \emph{degree} of $D$, denoted $\text{deg}(D)$, as follows. Perform an isotopy on $D$ so that each of its crossings locally looks like the diagram below. 
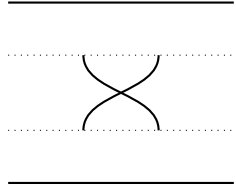
\begin{center}
        \begin{tikzpicture}
            \draw[-,thick](0,0)..controls (0,.5) and (1,.5) .. (1,1);
            \draw[-,thick](1,0)..controls (1,.5) and (0,.5) .. (0,1);
            \draw[-,dotted] (-1,0) -- (2,0);
            \draw[-,dotted] (-1,1) -- (2,1);
            \draw[-,thick] (-1,1.7) --(2,1.7);
            \draw[-,thick] (-1,-.7)--(2,-.7);

  
        \end{tikzpicture}
    \end{center} 
    
    In particular, the lower 2 endpoints of the crossing may be labeled by the induced $\Z$-orientation from the endpoints of the corresponding braid. Suppose that the bottom left endpoint is labeled $mi$ and the bottom right endpoint is labeled $nj$. If $n-m\geq 1$, we say that the degree of this crossing is $(-1)^{n-m}d_iC_{ij}$. If $n-m < 1$, we say that the degree of this crossing is $(-1)^{1+n-m}d_iC_{ij}$. It is easy to see that the degree of the crossing is preserved by rotation and well-defined on the isotopy class of the diagram. See for example the picture below. 
    
    \begin{center}
        \begin{tikzpicture}
            \draw (-1.05,.5) node{deg};
            \draw (-.65,.5) node{\Huge $($};
            \draw[-,thick] (0,0) ..controls (0,.5) and (1,.5)..(1,1);
            \draw[-,thick] (1,0)..controls (1,.5) and (0,.5)..(0,1);
             \draw (-.3,.15) node{$1i$};
            \draw (1.3,.15) node{$3j$};
            \draw (-.3, .85) node{$3j$};
            \draw (1.3, .85) node{$1i$};
            \draw (1.75,.5) node{\Huge$)$};
            \draw (2.7,.5) node{$=$ deg};

            \draw (3.35,.5) node{\Huge $($};
            \draw[-,thick] (4,0) ..controls (4,.5) and (5,.5)..(5,1);
            \draw[-,thick] (5,0)..controls (5,.5) and (4,.5)..(4,1);
             \draw (3.7,.15) node{$3j$};
            \draw (5.3,.15) node{$2i$};
            \draw (3.7, .85) node{$2i$};
            \draw (5.3, .85) node{$3j$};
            \draw (5.75,.5) node{\Huge$)$};

            \draw (6.8,.5) node{$=d_iC_{ij}$};
        \end{tikzpicture}
    \end{center}
    Then we define  $\text{deg}(D)$ to be the sum over the degrees of each crossing. In particular, the degree of a crossingless diagram is $0$.
\end{defn}
\begin{defn}
    Every minimal $\mathcal{B}_\Z$-diagram from $\textbf{i}$ to $\textbf{j}$ induces a perfect matching of the entries of $\textbf{i}$ and $\textbf{j}$. This matching is compatible with the vertex labels and $\Z$-orientation. Let $p'(\textbf{i},\textbf{j})$ denote the set of perfect matchings with compatible vertex and $\Z$-labels. It is easy to see that each element of $p'(\textbf{i},\textbf{j})$ is realized by some minimal diagram. We refer to elements of $p'(\textbf{i},\textbf{j})$ as $(\textbf{i},\textbf{j})$-\emph{pairings}. Note that the degrees of two minimal diagrams inducing the same matching are equal since they are related by the usual triple-crossing relations, and therefore we may also define degrees of $(\textbf{i},\textbf{j})-$pairings. Note also that the information of which braids cross in a minimal diagram is determined by the induced matching in $p'(\textbf{i},\textbf{j})$.
\end{defn}

\begin{defn}
    There is a $\Q$-bilinear form $(*,*)_G$ defined on $F_\Z(C)$ as follows. For two elements $\textbf{i},\textbf{j}\in \Z$Seq and $a,b\in \Q(q)$, we define
    \[(aE_\textbf{i},bE_\textbf{j})_G\coloneqq \bar{a}b(\prod_{i\in I}\frac{1}{(1-q_i^2)^{\kappa_i}})\sum_{D\in p'(\textbf{i},\textbf{j})}q^{\text{deg}(D)},\]
    where $\kappa_i$ denotes the number of braids with $I$-label $i$ in any minimal diagram realizing $D$.
\end{defn}
\begin{ex}
As an example, take $\textbf{i}=(0i, 8j, 1i, 2i)$ and $\textbf{j}=(2i, 1i, 0i, 8j)$. There are three elements of $p'(\textbf{i},\textbf{j})$ with representative minimal diagrams as follows.
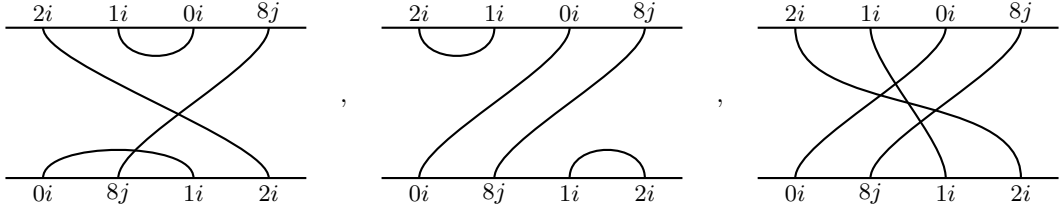
\begin{center}
    \begin{tikzpicture}
        \draw[-,thick] (-.5,0) -- (3.5,0);
        \draw [-,thick] (-.5,2)--(3.5,2);

        \draw[-,thick] (0,0)..controls(0,.5) and (2,.5)..(2,0);
        \draw[-,thick] (1,0)..controls (1,.5) and (3,1.5)..(3,2);
        \draw[-,thick] (3,0)..controls (3,.5) and (0,1.5)..(0,2);
        \draw[-,thick] (1,2)..controls (1,1.5) and (2,1.5)..(2,2);

        \draw (0,-.2) node{\small$0i$};
        \draw (1,-.2) node{\small$8j$};
        \draw (2,-.2) node{\small$1i$};
        \draw (3,-.2) node{\small$2i$};

        \draw (0,2.2) node{\small$2i$};
        \draw (1,2.2) node{\small$1i$};
        \draw (2,2.2) node{\small$0i$};
        \draw (3,2.2) node{\small$8j$};

        \draw (4,1) node{$,$};

        \draw[-,thick] (4.5,0) -- (8.5,0);
        \draw [-,thick] (4.5,2)--(8.5,2);

        \draw[-,thick] (5,0)..controls (5,.5) and (7,1.5)..(7,2);
        \draw[-,thick] (6,0)..controls (6,.5) and (8,1.5)..(8,2);
        \draw[-,thick] (7,0)..controls (7,.5) and (8,.5)..(8,0);
        \draw[-,thick] (5,2)..controls(5,1.5) and (6,1.5)..(6,2);

        \draw (5,-.2) node{\small$0i$};
        \draw (6,-.2) node{\small$8j$};
        \draw (7,-.2) node{\small$1i$};
        \draw (8,-.2) node{\small$2i$};

        \draw (5,2.2) node{\small$2i$};
        \draw (6,2.2) node{\small$1i$};
        \draw (7,2.2) node{\small$0i$};
        \draw (8,2.2) node{\small$8j$};

        \draw (9,1) node{$,$};

        \draw[-,thick] (9.5,0) -- (13.5,0);
        \draw [-,thick] (9.5,2)--(13.5,2);

        \draw[-,thick] (10,0)..controls (10,.5) and (12,1.5)..(12,2);
        \draw[-,thick] (11,0)..controls (11,.5) and (13,1.5)..(13,2);
        \draw[-,thick] (12,0)..controls (12,.5) and (11,1.5)..(11,2);
        \draw[-,thick] (13,0)..controls(13,1.2) and (10,.8)..(10,2);

        \draw (10,-.2) node{\small$0i$};
        \draw (11,-.2) node{\small$8j$};
        \draw (12,-.2) node{\small$1i$};
        \draw (13,-.2) node{\small$2i$};

        \draw (10,2.2) node{\small$2i$};
        \draw (11,2.2) node{\small$1i$};
        \draw (12,2.2) node{\small$0i$};
        \draw (13,2.2) node{\small$8j$};
    \end{tikzpicture}
\end{center}

The degree of the first diagram is $d_jC_{ji}-d_jC_{ji}=0$, and the degree of the second diagram is also $0$ since it has no crossings. The degree of the third diagram is $d_jC_{ij}-d_jC_{ij}-d_iC_{ii}+d_iC_{ii}-d_iC_{ii}=-d_iC_{ii}$. Therefore, we compute \[(E_{\textbf{i}},E_{\textbf{j}})_G=\frac{1}{(1-q_i^2)^3}\cdot\frac{1}{1-q_j^2}\cdot(q^0+q^0+q_i^{-2})=\frac{2+q_i^{-2}}{(1-q_i^2)^3(1-q_j^2)}.\]
\end{ex}
The coefficient $\prod_{i\in I} 1/(1-q_i^2)^{\kappa_i}$ may be understood as coming from the possible placements of dots along a minimal diagram realizing $D$, where a dot on a braid with vertex label $i$ has degree $2d_i$, and diagrams with the same number of dots per braid are considered equal. This should be compared to the $\mathcal{A}_2$-diagrams of the previous section.

For monomials $X$ and $Y$, it is clear from the definition that $(X,Y)_G=0$ if $X\in B_\Z(C)_\alpha$ and $Y\in B_\Z(C)_\beta$ with $\alpha\neq \beta$. The symmetries $\bar{*}$ and $\mathcal{D}$ are also defined on $F_\Z(C)$. It is clear that for any $X,Y\in F_\Z(C)$ that $(X,Y)_G=(\bar{\mathcal{D}}(X),\bar{\mathcal{D}}(Y))_G$. By rotating diagrams by 180 degrees, we have $(X,Y)_G=(\bar{Y},\bar{X})_G$. Composing these, we have $(X,Y)_G=(\mathcal{D}(Y),\mathcal{D}(X))_G$. We will need the following refinement of this result.

\begin{lem}\label{lem:graphic_form_duality}
For any $X,Y\in F_\Z(C)$ and $i\in I$, $n\in \Z$, we have that $(E_{i,n}X,Y)_G=(X,E_{i,n+1}Y)_G$ and $(XE_{i,n},Y)_G=(X,YE_{i,n-1})_G$.
\end{lem}
\begin{proof}
    We prove the first equality since the second equality is dual. By linearity, it is enough to prove the claim when $X=E_\textbf{i}$ and $Y=E_\textbf{j}$ are monomials. We construct a degree-preserving bijection $p'(ni\cdot \textbf{i},\textbf{j})\rightarrow p'(\textbf{i},(n+1)i\cdot \textbf{j})$. For each representative minimal $\mathcal{B}_\Z$-diagram of $p'(ni\cdot \textbf{i},\textbf{j})$, attach a cup to the bottom of the diagram connecting  this $ni$ to the new $(n+1)i$ on the top. This bijection is depicted in the following picture.
    
    
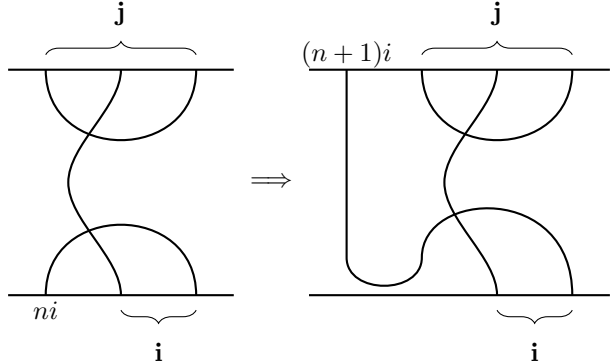
\begin{center}
    \begin{tikzpicture}
        \draw[-,thick] (-.5,0)--(2.5,0);
        \draw[-,thick] (-.5,3)--(2.5,3);
        \draw[-,thick] (0,0) .. controls (0,1.25) and (2,1.25) .. (2,0);
        \draw[-,thick] (2,3) .. controls (2,1.75) and (0,1.75) .. (0,3);
        \draw[-, thick] (1,0) .. controls (1,.5) and (.3,1) .. (.3,1.5) .. controls (.3,2) and (1,2.5)..(1,3);

        \draw[decorate,decoration={brace,amplitude=5pt,raise=-1ex}](0,3.35) -- (2, 3.35);
        \draw (1, 3.75) node{$\textbf{j}$};

        \draw[decorate,decoration={brace,amplitude=5pt,mirror, raise=-1ex}](1,-.35) -- (2, -.35);
        \draw (1.5, -.75) node{$\textbf{i}$};

        \draw (0,-.2) node{$ni$};

        \draw(3,1.5) node{$\implies$};

        \draw[-,thick] (3.5,0)--(7.5,0);
        \draw[-,thick] (3.5,3)--(7.5,3);
        \draw[-,thick] (4,3)..controls (4,1) and (4,1)..(4,.5)..controls (4,0) and(5,0)..(5,.5)..controls (5,1.4) and (7,1.5) .. (7,0);
        \draw[-,thick] (7,3) .. controls (7,1.75) and (5,1.75) .. (5,3);
        \draw[-, thick] (6,0) .. controls (6,.5) and (5.3,1) .. (5.3,1.5) .. controls (5.3,2) and (6,2.5)..(6,3);

        \draw[decorate,decoration={brace,amplitude=5pt,raise=-1ex}](5,3.35) -- (7, 3.35);
        \draw (6, 3.75) node{$\textbf{j}$};

        \draw[decorate,decoration={brace,amplitude=5pt,mirror, raise=-1ex}](6,-.35) -- (7, -.35);
        \draw (6.5, -.75) node{$\textbf{i}$};

        \draw (4,3.2) node{$(n+1)i$};

    \end{tikzpicture}
\end{center}
The reverse bijection is attaching the appropriate cap to $(n+1)i$. This bijection also does not change the number of braids with a fixed $I$-label, so the lemma follows.
\end{proof}
Denote by $\text{ker}(*,*)_G$ the subspace of $F_\Z(C)$ consisting of elements $X$ satisfying $(Y,X)_G=0$ for any $Y$. This is equivalent to requiring that all $(\bar{X},\bar{Y})_G=0$. We say that a monomial $E_{i_1,n_1}\dots E_{i_k,n_k}$ is \emph{reduced} if $n_j\leq n_l$ when $j\leq l$. We say that the monomial is \emph{coreduced} if $n_j\geq n_l$ when $j\leq l$.

\begin{thm}\label{thm:qthom_diagram}
    The following are true.
    \begin{enumerate}
        \item $\text{ker}(*,*)_G$ is a two-sided ideal of $F_\Z(C)$.
        \item The bilinear form $(*,*)_G$ is invariant under the quantum boson and quantum Serre relations, i.e. it descends to $B_\Z(C)$ under the homomorphism $E_{i,n}\rightarrow E_{i,n}$.
        \item The bilinear form $(*,*)_G$ on $B_\Z(C)$ is identical to $(*,*)_\Z$.
        \item There is an isomorphism of algebras $B_\Z(C)\simeq F_\Z(C)/\text{ker}(*,*)_G$
    \end{enumerate}
\end{thm}
\begin{proof}

    To prove (1), suppose that there is $X\in F_\Z(C)$ satisfying $(Y,X)_G=0$ for all $Y\in F_\Z(C)$. To prove that $XF_\Z(C)\subset \text{ker}(*,*)_G$, it is sufficient to show that for any $i\in I$, $n\in \Z$, and $Z\in F_\Z(C)$ that $(Z,XE_{i,n})_G=0$. By Lemma \ref{lem:graphic_form_duality}, this is equivalent to proving that all $(ZE_{i,n+1},X)_G=0$, which holds by assumption on $X$. Showing closure under left multiplication is dual.

    For (2), we first prove the quantum boson relations. For any monomial $E_\textbf{i}\in F_\Z(C)$, any $i,j\in I$, any $n\in \Z$ and $k\in \Z_{\geq 2}$, we produce bijections $p'(\textbf{i},(n+k)i\cdot nj)\rightarrow p'(\textbf{i},nj\cdot (n+k)i)$. In any minimal $\mathcal{B}_\Z$-diagram with target $(n+k)i\cdot nj$, the braids connecting to the top $(n+k)i$ and $nj$ cannot coincide due to the $\Z$ orientation. For those elements  $v\in p'(\textbf{i},(n+k)i\cdot nj)$ where these two braids cross in any representative diagram, we pick a representative $D$ where this crossing is above all other crossings. Let $D'$ be the diagram from $X$ to $nj\cdot (n+k)i$ obtained by removing this crossing. Note that $D'$ is still minimal and that $\text{deg}(D')=\text{deg}(D)-(-1)^{k}d_iC_{ij}$. We then map $v$ to the element of $p'(\textbf{i},nj\cdot (n+k)i)$ represented by $D'$. Likewise, for those $v\in p'(\textbf{i},(n+k)i\cdot nj)$ where these two braids do not cross in any representative diagram, pick any representative $D$. Let $D'$ be the diagram obtained by putting a crossing of these $nj$ and $(n+k)i$ on top of $D$, and note that this is still minimal. We again compute that $\text{deg}(D')=\text{deg}(D)-(-1)^kd_iC_{ij}$. We then map $v$ to the element of $p'(\textbf{i},nj\cdot (n+k)i)$ represented by $D'$.
    
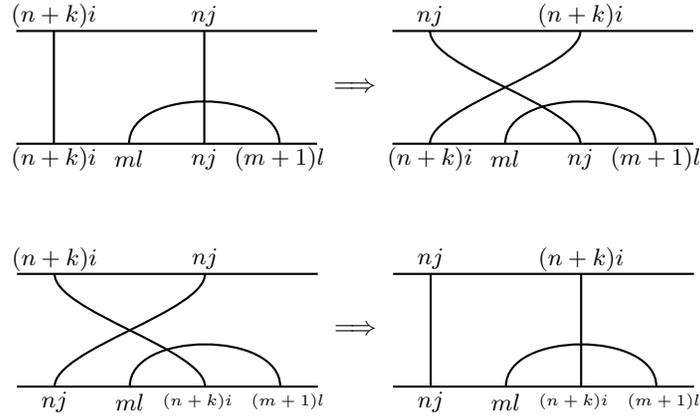
    \begin{center}
        \begin{tikzpicture}
            \draw[-,thick] (0,0) -- (0,1.5);
            \draw[-,thick] (1,0)..controls (1,.75) and (3,.75)..(3,0);
            \draw[-,thick] (2,0)--(2,1.5);

            \draw[-,thick] (-.5,1.5)--(3.5,1.5);
            \draw[-,thick] (-.5,0)--(3.5,0);

            \draw (0,-.2) node{\small$(n+k)i$};
            \draw (1,-.2) node{\small$ml$};
            \draw (2,-.2) node{\small$nj$};
            \draw (3,-.2) node{\small$(m+1)l$};
            \draw (0,1.7) node{\small$(n+k)i$};
            \draw (2,1.7) node{\small$nj$};

            \draw (4,.75) node{$\implies$};

            \draw[-,thick] (5,0)..controls (5,.5) and (7,1).. (7,1.5);
            \draw[-,thick] (6,0)..controls (6,.75) and (8,.75)..(8,0);
            \draw[-,thick] (7,0)..controls (7,.5) and (5,1)..(5,1.5);

            \draw[-,thick] (4.5,1.5)--(8.5,1.5);
            \draw[-,thick] (4.5,0)--(8.5,0);

            \draw (5,-.2) node{\small$(n+k)i$};
            \draw (6,-.2) node{\small$ml$};
            \draw (7,-.2) node{\small$nj$};
            \draw (8,-.2) node{\small$(m+1)l$};
            \draw (5,1.7) node{\small$nj$};
            \draw (7,1.7) node{\small$(n+k)i$};
        \end{tikzpicture}
    \end{center}

    \begin{center}
        \begin{tikzpicture}
            \draw[-,thick] (0,0)..controls (0,.5) and (2,1)..(2,1.5);
            \draw[-,thick] (1,0)..controls (1,.75) and (3,.75)..(3,0);
            \draw[-,thick] (2,0)..controls (2,.5) and (0,1)..(0,1.5);

            \draw[-,thick] (-.5,1.5)--(3.5,1.5);
            \draw[-,thick] (-.5,0)--(3.5,0);

            \draw (0,-.2) node{\small$nj$};
            \draw (1,-.2) node{\small$ml$};
            \draw (1.9,-.2) node{\tiny$(n+k)i$};
            \draw (3.1,-.2) node{\tiny$(m+1)l$};
            \draw (0,1.7) node{\small$(n+k)i$};
            \draw (2,1.7) node{\small$nj$};

            \draw (4,.75) node{$\implies$};

            \draw[-,thick] (5,0)--(5,1.5);
            \draw[-,thick] (6,0)..controls (6,.75) and (8,.75)..(8,0);
            \draw[-,thick] (7,0)--(7,1.5);

            \draw[-,thick] (4.5,1.5)--(8.5,1.5);
            \draw[-,thick] (4.5,0)--(8.5,0);

            \draw (5,-.2) node{\small$nj$};
            \draw (6,-.2) node{\small$ml$};
            \draw (6.9,-.2) node{\tiny$(n+k)i$};
            \draw (8.1,-.2) node{\tiny$(m+1)l$};
            \draw (5,1.7) node{\small$nj$};
            \draw (7,1.7) node{\small$(n+k)i$};
        \end{tikzpicture}
    \end{center}
    
    The map $p'(\textbf{i},(n+k)i\cdot nj)\rightarrow p'(\textbf{i},nj\circ (n+k)i)$ is evidently a bijection that lowers degree by $(-1)^kd_iC_{ij}$. It also does not change the number of braids with a fixed $I$-label. So, we have that $(X,E_{i,n+k}E_{j,n})_G=q_i^{(-1)^kC_{ij}}(X,E_{j,n}E_{i,n+k})_G=(X,q_i^{(-1)^kC_{ij}}E_{j,n}E_{i,n+k})$. We argue similarly for the quantum boson relations $E_{i,n+1}E_{j,n}-q_i^{-C_{ij}}E_{j,n}E_{i,n+1}=\delta_{ij}/(1-q_i^2)$. If $i\neq j$ then the proof is exactly the same. In the case $i=j$, we instead construct a bijection $p'(\textbf{i},(n+1)i\cdot ni)\rightarrow p'(\textbf{i},ni\cdot (n+1)i)\sqcup p'(\textbf{i},\emptyset)$. We have now a third class of elements of $p'(\textbf{i},(n+1)i\cdot ni)$ represented by minimal diagrams in which the braids connecting to the top $(n+1)i$ and $ni$ coincide. In these diagrams, the braid connecting these two cannot cross any other braids, as otherwise the diagram is not minimal. So, if we remove the connecting cup from the diagram, we obtain an element of $p'(\textbf{i},\emptyset)$.

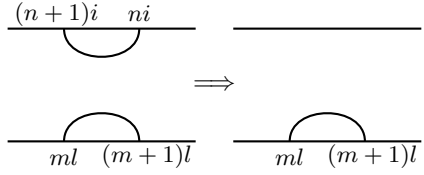
    \begin{center}
        \begin{tikzpicture}
            \draw[-,thick](0,0)..controls (0,.5) and (1,.5)..(1,0);
            \draw[-,thick] (0,1.5)..controls (0,1) and (1,1)..(1,1.5);
            \draw[-,thick] (-.75,0)--(1.75,0);
            \draw[-,thick] (-.75,1.5)--(1.75,1.5);

            \draw(0,-.2) node{\small$ml$};
            \draw(1.1,-.2) node{\small$(m+1)l$};
            \draw(-.1,1.7) node{\small$(n+1)i$};
            \draw(1,1.7) node{\small$ni$};

            \draw(2,.75) node{$\implies$};

            \draw[-,thick](3,0)..controls (3,.5) and (4,.5)..(4,0);
            \draw[-,thick] (2.25,0)--(4.75,0);
            \draw[-,thick] (2.25,1.5)--(4.75,1.5);

            \draw(3,-.2) node{\small$ml$};
            \draw(4.1,-.2) node{\small$(m+1)l$};
        \end{tikzpicture}
    \end{center}
    
    This is clearly a bijection between $p'(\textbf{i},\emptyset)$ and the subset of $p'(\textbf{i},(n+1)i\cdot ni)$ where the upper entries are paired. The degree of a diagram does not change when we remove an uncrossed cup, but since we remove a braid with vertex label $i$ in applying this bijection, we need to include an extra factor of $1/(1-q_i^2)$ when computing $(X,E_{i,n+1}E_{i,n})_{G}$. So, the bijection $p'(\textbf{i},(n+1)i\cdot ni)\rightarrow p'(\textbf{i},ni\cdot (n+1)i)\sqcup p'(\textbf{i},\emptyset)$ gives us that $(X,E_{i,n+1}E_{i,n})_G=(X,q_i^{-2}E_{i,n}E_{i,n+1})_G+(X,1/(1-q_i^2))_G$, as desired.

    We now prove that $(*,*)_G$ is invariant under the quantum Serre relations. Fix $i\neq j\in I$, $n\in \Z$, and the element
     \[Y\coloneqq \sum_{k=0}^{1-C_{ij}}\binom{1-C_{ij}}{k}_iE_{i,n}E_{j,n}E_{i,n}^{1-C_{ij}-k}.\] We show that $Y\in \text{ker}(*,*)_G$. Since $(*,*)_G$ is invariant under quantum boson relations and by (1), it is enough to prove that all $(X,Y)_G=0$ for $X$ a coreduced monomial. Denote by $F_{\Z}(C)_n$ the subalgebra generated by the $\{E_{l,n}\}_{l\in I}$. It follows from the $\Z$-orientation of our diagrams that if $X\notin F_{\Z}(C)_n$ that $(X,Y)_G=0$. So, we may assume that $X\in F_{\Z}(C)_n$. 
     In this case, it is known that $(*,*)_G$ agrees with Lusztig's form $(\bar{*},*)_L$ on the free algebra in the $E_{l,n}$, see \cite{kholau3} Lemma 2.8 and references therein. It is known that $(\bar{*},*)_L$ is invariant under the quantum Serre relations, i.e. that $(X,Y)_G=0$.

    We have shown that $(*,*)_G$ is invariant under the quantum boson and quantum Serre relations. Let $\mathcal{I}$ denote the two sided ideal of $F_\Z(C)$ generated by these relations so that $F_\Z(C)/\mathcal{I}=B_\Z(C)$. By (1), we have that $\mathcal{I}\subset \text{ker}(*,*)_G$, so $(*,*)_G$ descends to $B_\Z(C)$.
    
    To prove (3), we may use Proposition \ref{prop:qthomform}, Lemma \ref{lem:graphic_form_duality}, and the quantum boson relations to reduce to showing that $(1,*)_G=(1,*)_\Z$. By applying quantum boson relations on the target, it is enough to show that $(1,E_\textbf{i})_G=(1,E_\textbf{i})_\Z$ for $E_\textbf{i}$ a reduced monomial. We see that if $E_\textbf{i}\neq 1$ then there are no minimal $\mathcal{B}_\Z$-diagrams from $\emptyset$ to $E_\textbf{i}$ due to the $\Z$-orientation, and otherwise $(1,1)_G=1$. We compute the same for $(1,E_\textbf{i})_\Z$ directly from the definition. So, $(*,*)_\Z=(*,*)_G$ on $B_\Z(C)$.
    
    To prove (4), note that (1), (2), and (3) give us a surjective homomorphism of algebras $B_\Z(C)\rightarrow F_\Z(C)/\text{ker}(*,*)_G$ intertwining $(*,*)_\Z$ and $(*,*)_G$. Due to Proposition \ref{prop:qthomform}, the form $(*,*)_\Z$ is nondegenerate on $B_\Z(C)$. Therefore, this surjective map is also injective due to the arguments of Theorem \ref{thm:gammaiso}.
\end{proof}
\begin{rem}
     Remark \ref{rem:missingdiag} observes that $(F_iE_i,E_iF_i)_2\neq 0$ despite there being no $\mathcal{A}_2$-diagrams from $F_iE_i$ to $E_iF_i$. Theorem \ref{thm:qthom_diagram} explains this discrepancy. The crossing $F_iE_i\rightarrow E_iF_i$ is allowed in an $\mathcal{B}_\Z$-diagram but not in an $\mathcal{A}_2$-diagram.
\end{rem}

We can mimic the results of Subsection \ref{subsec:bases} to produce an interesting basis of $B_\Z(C)$. Let $S_\Z$ be the set of $B_\Z(C)$ elements of the form $\prod_{n\in \Z} G_n(X_n)$ with each $X_n\in S$ and all but finitely many equal to $1$. It follows from the results of \cite{newkashiboson} that $S_\Z$ is a basis for $B_\Z(C)$.

\begin{thm}\label{thm:bznicebasis}
The following are true of the basis $S_\Z$. Fix elements of $S_\Z$ with factorizations $X=\prod_{n\in \Z} G_n(X_n)$ and $Y=\prod_{n\in \Z}G_n(Y_n)$. Let $(*,*)$ be the symmetric bilinear form on $\Z[I]$ determined by $(\alpha_i,\alpha_j)=d_iC_{ij}$.
\begin{enumerate}
    \item  $XY$ is contained in the $\N[q,q^{-1},(1/(1-q_i^2))_{i\in I}]$-span of $S_\Z$.
    \item  If each $\text{gr}(X_n)\leq\text{gr}(Y_n)$, then $(X,Y)_\Z=\prod_{n\in \Z} (\bar{X}_n,Y_n)_L$.
    \item $(\bar{X},Y)_\Z=q^{c}\prod_{n\in \Z}(\text{rev}(X_n),Y_n)_L$, where $c=\sum_{m< n}(\text{gr}(X_m),\text{gr}(X_n))$ and $\text{rev}$ is the $q$-linear antiinvolution of $U_q^+(C)$ that fixes each $E_i$. 
\end{enumerate}
   
\end{thm}
\begin{proof}
    It is sufficient to prove (1) under the assumption that $Y=G_m(Z)$ for some $Z\in S$ and $m\in \Z$. By the defining relations of $B_\Z(C)$, we have that $(\prod_{n>m+1}G_n(X_n))Y=q^{(\alpha_Y,\alpha)}Y\prod_{n>m+1}G_n(X_n)$, where $\text{gr}(Y)=\alpha_Y$, and $\text{gr}(\prod_{n>m+1} G_n(X_n))=\alpha$.  For the $\Q(q)$-algebra embedding $T_m:B(C)\hookrightarrow B_\Z(C)$ given by $E_i\rightarrow E_{i,m}$ and $F_i\rightarrow E_{i,m+1}$, we have that $T_m(S_2)\subset S_\Z$ after perhaps adjusting shifts. Then Theorem \ref{thm:gammaiso} shows that $G_m(X_m)G_{m+1}(X_{m+1})Y$ is contained in the $\N[q,q^{-1},(1/(1-q_i^2))_{i\in I}]$-span of $S_\Z$. The claim follows.

    To prove (2), write each $G_n(X_n)=\sum_{i\leq k_n}a_iG_n(X_{n,i})$, where the $a_i\in \Q(q)$ and $G_n(X_{n,i})$ are monomials in $B_\Z(C)_n$ of the same $\Z[I]$-degree as $G_n(X_n)$. We similarly write $G_n(Y_n)=\sum_{j\leq l_n}b_iG_n(Y_{n,j})$. Then we compute \[(X,Y)_\Z=\sum (\prod_n a_{i_n}G_n(X_{n,i_n}),\prod_n b_{j_n}G_n(Y_{n,j_n})),\]
    where the sum is over all choices of $i_n\leq k_n$ and $j_n\leq l_n$. Note that both $\prod_n G_n(X_{n,i_n})$ and $\prod_n G_n(Y_{n,j_n})$ are reduced. In a minimal $\mathcal{B}_\Z$-diagram from $zseq(\prod_n G_n(X_{n,i_n}))$ to $zseq(\prod_n G_n(Y_{n,j_n}))$, no two entries of $zseq(\prod_n G_n(Y_{n,j_n}))$ can be paired together due to the reducedness, and so each entry of $zseq(G_n(Y_{n,j_n}))$ must be paired with an appropriately labeled entry of $zseq(G_n(X_{n,i_n}))$. Since $\text{gr}(X_{n,i_n})\leq\text{gr}(Y_{n,j_n})$, every entry of $zseq(G_n(X_{n,i_n}))$ is paired in this way. For $a$ an entry of $zseq(G_n(X_{n,i_n}))$ and $b$ an entry of $zseq(G_m(X_{m,i_m}))$ with $n\neq m$, the braids connecting to $a$ and $b$ cannot cross due to minimality. Therefore, we see that \[(\prod_n a_{i_n}G_n(X_{n,i_n}),\prod_n b_{j_n}G_n(Y_{n,j_n}))_\Z=\prod_n(a_{i_n}G_n(X_{n,i_n}),b_{j_n}G_n(Y_{n,j_n}))_\Z=\prod_n(\overline{a_{i_n}X_{n,i_n}},b_{j_n}Y_{n,j_n})_L.\] The claim follows.

    (3) can be proven with a similar method, but it is simpler to use property (7) of Proposition \ref{prop:qthomform} to deduce that $(\bar{X},Y)_\Z=q^{c}\prod_{n\in \Z}(\overline{G_n(X_n)},Y_n)_\Z=q^{c}\prod_{n\in \Z} (\text{rev}(X_n),Y_n)_L$.
    
\end{proof}
Theorem \ref{thm:bznicebasis} suggests that the basis $S_\Z$ is a basis of indecomposable objects in an appropriate categorification of $B_\Z(C)$. Producing such a categorification is work in progress.  Although we have technically produced several different bases of $B_\Z(C)$ depending on $Q$, $\kf$, and shifts, comparing them reduces to the same comparison for $S\subset U_q^+(C)$. Understanding the relation between these bases would also allow one to compare to the bases produced in \cite{newkashiboson}.

\bibliographystyle{amsalpha}
    \bibliography{ref.bib}
\end{document}